\documentclass[12pt]{article}
\usepackage{amssymb,latexsym,pstricks,pst-node,amsmath}

\newtheorem{theorem}{Theorem}[section]
\newtheorem{proposition}[theorem]{Proposition}
\newtheorem{corollary}[theorem]{Corollary}
\newtheorem{definition}[theorem]{Definition}
\newtheorem{conjecture}[theorem]{Conjecture}
\newtheorem{lemma}[theorem]{Lemma}

\newtheorem{notation}[theorem]{Notation}
\newtheorem{summary}[theorem]{Summary}
\newtheorem{question}[theorem]{Question}

\newenvironment{proof}{\paragraph{\it Proof.}}{$\square$\vskip0.4cm}
\newenvironment{remark}{\paragraph{\it Remark.}}{\vskip0.4cm}
\newenvironment{example}{\paragraph{\it Example.}}{\vskip0.4cm}

\def\A{\mathcal{A}}  
  \def\B{\mathcal{B}}
 \def\C{\mathbb{C}} 
   
 \def\H{\mathcal{H}} \def\R{\mathbb{R}} 
\def\T{\mathbb{T}} \def\t{\mathfrak{t}} \def\h{\mathsf{h}}  \def\s{\sigma}
\def\bs{\breve{\sigma}} 
\def\Z{\mathbb{Z}} \def\N{\mathbb{N}}  \def\K{\mathcal{K}}
\def\mathbi#1{\textbf{\em #1}} \def\hp{{\mathbi{h}_P}} \def\hpn{{\mathbi{h}_{P^n}}}
\def\l{l}

\begin{document}


\title{Berezin Quantization From Ergodic Actions of Compact Quantum Groups, and Quantum Gromov-Hausdorff Distance}
\author{Jeremy Sain}

\maketitle

\begin{abstract}

\vspace{24pt}

Over the last 25 years, the notion of ``fuzzy spaces'' has become ubiquitous in the high-energy physics literature.  These are finite dimensional noncommutative approximations of the algebra of functions on a classical space.  The most well known examples come from the Berezin quantization of  coadjoint orbits of compact semisimple Lie groups.

We develop a theory of Berezin quantization for certain quantum homogeneous spaces coming from ergodic actions of compact quantum groups.  This allows us to construct fuzzy versions of these quantum homogeneous spaces.  We show that the finite dimensional approximations converge to the homogeneous space in a continuous field of operator systems, and in the quantum Gromov-Hausdorff distance of Rieffel.  

We apply the theory to construct a fuzzy version of an ellipsoid which is naturally endowed with an orbifold structure, as well as a fuzzy version of the $\theta$-deformed coset spaces $C(G/H)_{\hbar \theta}$ of Varilly.  In the process of the latter, we show that our Berezin quantization commutes with Rieffel's deformation quantization for actions of $\R^d$.

\end{abstract}



\tableofcontents


\begin{section}{Introduction}

The concept of a path integral is essential in quantum field theory.  To compute these integrals for quantum systems of dimension greater than $0+1$, one typically tries to approximate the configuration space of the system by a zero dimensional space and then attempts to take an appropriate limit.   However, if one replaces the configuration space by some lattice of its points, one usually loses most of the symmetry of the configuration space.  This is unappealing from a computational standpoint, since symmetry generally makes computations simpler.  More important than this though is the fact that symmetries of the configuration space give conservation laws of the physical system.  Thus there are typically aspects of the physics lost when one replaces the configuration space by a lattice approximation.  For a basic introduction to quantum field theory, see \cite{Peskin}.

 In the physics literature, it is very common to find the notation of ``fuzzy spaces''  \cite{Majid}, \cite{Flores}, \cite{Dolan2}, \cite{Bietenholz}, \cite{Mondragon}, \cite{Lozano}. These are finite dimensional noncommutative spaces that replace the lattice approximation of a configuration space and have the full symmetry of the   space.  Historically the first example was the fuzzy sphere.    For each $n\in \N$ there is an action of $SU(2)$ on $M_n(\C )$ by conjugation with the $n$-dimensional irreducible representation of $SU(2)$. Moreover, one has $M_n(\C )=\underline{1}\oplus \underline{3}\oplus ...\oplus\underline{2n-1}$ as an $SU(2)$-module, where $\underline{n}$ denotes the $n$-dimensional irreducible representation of $SU(2)$.  
 On the other hand, if one decomposes $C(S^2)$ as an $SU(2)$-module, one has $C(S^2)=\underline{1}\oplus \underline{3}\oplus \underline{5}\oplus ...$.  Since $SU(2)$ (or rather $SO(3)$) is the full isometry group of $S^2$, this suggests that dynamics on $M_n(\C )$ should resemble dynamics on $S^2$ for $n\in \N $ large.  One can go on to check that the matrix product on $M_n(\C )$ converges to the commutative product on $C(S^2)$ as $n\longrightarrow \infty$ in a suitable sense \cite{Balachandran}, so that particle interactions agree in the limit as well \cite{Peskin}.  For this reason, the sequence of matrix algebras $M_n(\C )$ equipped with their actions of $SU(2)$ is called the fuzzy sphere, and it is the optimal setting to use for computing path integrals on the sphere.  See \cite{Balachandran} for more details and computations for general fuzzy $\mathbb{CP}^n$.  

This example is best understood using the language of strict quantization of C$^*$-algebras. 
The algebras $M_n(\C )$ together with $C(S^2)$ form a continuous field of C$^*$-algebras, and the sequence 
 of Berezin adjoint maps $\bs^n:C(S^2)\rightarrow M_n(\C )$ give for any $f\in C(S^2)$  a continuous section 
\[ \mathcal{Q}_\hbar (f)=\left\{ \begin{array}{ll} f & \hspace{.4in} \hbar =0 \\						\bs^n (f) &\hspace{.4in} \hbar =1/n, n\in \N . \end{array} \right.   \]
   The essential property of these Berezin adjoints as well as the corresponding Berezin symbols $\s^n:M_n(\C )\rightarrow C(S^2)$ which are adjoint to them is that they are positive, unital, and $SU(2)$-equivariant.  These Berezin symbols are constructed more generally for coadjoint orbits  of compact semisimple Lie groups, and we write them down in Section 2.6 below.  See \cite{book} for more details of this construction, and particularly for the proof that there is a continuous field of C$^*$-algebras.

In \cite{Rieffel3}, \cite{Rieffel6} Rieffel proves additionally that given any length function on $SU(2)$, there are corresponding quantum metrics (that is Lip-norms, see Definition \ref{defLip} below) on $C(S^2)$ and on each $M_n(\C )$.  Moreover, he uses the Berezin symbol and adjoint to show that the quantum metric spaces $M_n(\C )$ converge to $C(S^2)$ in quantum Gromov-Hausdorff distance.  He  proves the corresponding result for general coadjoint orbits.

The goal of this paper is to generalize these results by constructing Berezin quantizations from actions of compact quantum groups, and then  generalizing Rieffel's results on inducing quantum metrics, and  proving convergence in quantum Gromov-Hausdorff distance. 

Throughout this paper, we denote by  $A$ a compact quantum group, $\A$ its polynomial subalgebra, $\Delta$ its comultiplication, $\varepsilon$ its counit, $\kappa$ its coinverse, $\h$ its Haar state, and $\mathsf{m}$ its multiplication map.

Let $A$ be coamenable and of Kac-type (i.e. $\kappa^2=id$).  Let $(u,\H ,\xi )$ be an  irreducible unitary representation of $A$. Let $\xi \in \H$ be  a distinguished vector  with the property that the cyclic subrepresentation generated by $\xi^{\otimes n}$  in $\H^{\otimes n}$ is irreducible for all $n\in \N$ (like a highest weight vector for Lie groups).
Given this  data we construct a coideal $A^H\subseteq A$  which is the analog of  a coadjoint orbit for Lie groups, and we construct a sequence of finite dimensional operator systems $\B^n$, and positive unital mappings $\s^n:\B^n \rightarrow A^H$ and $\bs^n: A^H\rightarrow \B^n$.  We prove the following theorem:

\begin{theorem}\label{thm1}
The mappings $\bs^n:A^H\rightarrow \B^n$ give a strict quantization of the operator system $A^H$.  If $L_A$ is any right-invariant Lip-norm on $A$, then $L_A$ induces Lip-norms on $L$ and $L_n$ on $A^H$ and $\B^n$, and  $(\B^n, L_n)$ converges to $(A^H, L)$ in quantum Gromov-Hausdorff distance as $n\rightarrow \infty$.
\end{theorem}

The most peculiar feature of our construction is that our quantum homogeneous spaces (see Definition \ref{quantumhomo}) will be operator systems and not necessarily C$^*$-algebras.  This facet is linked to the fact that our stabilizers are not quantum subgroups of $A$.  There is very little  discussion of orbits and stabilizers in the literature.  One of the few serious discussions is in \cite{Jurco}, where the stabilizer is taken to be largest quantum \emph{subgroup} of $A$ fixing a distinguished element.  Based on our calculations in Chapter 4, we find that the correct definition is that the stabilizer should be the largest quantum \emph{subset} of $A$ fixing the distinguished element.  For ordinary groups these two notions coincide, but we will see in Section 3.3 that they are distinct even for cocommutative quantum groups. 

One difficulty we must overcome is to define a strict quantization of an operator system.  There has been much discussion of strict quantization for C$^*$-algebras, and many definitions have been given.  A very good summary of these is given in Section 2 of Hawkins paper \cite{Hawkins}.  In Definition \ref{defdeformquant} below we give a working definition for a strict quantization of an operator system, which we define essentially just to be a continuous field of operator systems together with a distinguished collection of the continuous sections.  In the language of Hawkins, this is what he would call a $0^{\text{th}}$-order strict deformation quantization in the C$^*$-case.  Here, the ``order'' essentially measures how well the quantization interacts with the Poisson bracket (that is, a distinguished Hochschild cohomology class) on the C$^*$-algebra. In his paper, Hawkins argues that anything less than a $2^\text{nd}$-order strict deformation quantization should really not be considered to be a quantization at all.  We hope that later a more complete definition in the case of operator systems will be given that will  include a discussion of Poisson brackets for operator systems.  

We believe that the Kac condition in our theorem should be able to be removed by suitably twisting our operator systems.  Evidence for this belief can be seen from  the results of 
 \cite{Jurco}, where the authors study coherent states for  $q$-deformations of the classical simple  Lie groups. In Section 8 the authors give a version of Berezin quantization for these groups.  In \cite{Harikumar} the authors extend this construction in the case of $SU_q(2)$ by giving  Dirac operators on the $q$-deformed fuzzy spheres, which by Connes' philosophy is akin to giving a Riemannian metric on them. They check that the spectrum of their Dirac operator agrees with that of a standard Dirac operator on the $q$-deformed sphere up to some finite cutoff, suggesting that the $q$-deformed fuzzy spheres are converging to the $q$-deformed sphere as Riemannian manifolds, and not just as topological spaces.  
We hope to soon be able to include these examples into our framework.  We remark that the example given in the appendix to this paper shows that it is unlikely that  the coamenability assumption  can be removed.

We would also like to relate the constructions of this paper to the results of Bhowmick and Goswami on quantum isometry groups.  In \cite{Bhowmick1} Bhowmick and Goswami define the quantum isometry group of a spectral triple, and they compute examples in \cite{Bhowmick2} and \cite{Bhowmick3}.
In \cite{Goswami} Goswami defines the quantum isometry group of a general (classical) metric space.  It is a very important fact for quantum field theory that the finite dimensional approximations that comprise a ``fuzzy'' version of a metric space have the same isometry group as the original space.  We would like to understand what is the quantum isometry group of a quantum metric space, and we would like to know whether our finite dimensional approximations $\B^n$ have the same quantum isometry group as the coideal $A^H$.

In the later sections of this paper, we will extend our results to certain operator systems constructed from ergodic actions of $A$ on C$^*$-algebras.  The literature on ergodic actions of quantum groups is extensive.  We give a few examples in Section 6, but the list there is far from exhaustive.  First of all, in \cite{Pinzari2} Pinzari and Roberts extend the Tannaka-Krein duality for quantum groups by showing that any quasitensor functor on the representation category of a quantum group $A$ comes from an ergodic action of $A$ on a C$^*$-algebra.   Although not strictly an example, there is a related construction by Banica in \cite{Banica1} that associates a Popa system to any finite dimensional unitary representation of $A$.  These two results together are the foundation for a philosophy that there is a strong connection between subfactors and  actions of quantum groups.  In \cite{Pinzari3} the authors associate an ergodic action of $C(SU_q(2))$ to any inclusion  of II$_1$ factors with finite Jones index.  Other examples of actions related to subfactors (although not all are ergodic) can be found in \cite{Banica3}, \cite{Banica4}, \cite{Masuda1}, \cite{Masuda2}, and \cite{Nikshych}. For a good survey article, see \cite{Banica5}.  A completely different family of examples are described  in \cite{Wang2}, where Wang gives examples of ergodic actions of certain compact quantum groups on the hyperfinite factor of type III$_\lambda$ for any $\lambda \neq 0$.  He also gives examples of ergodic actions of Kac-type quantum groups on the hyperfinite II$_1$ factor, and on certain type III$_\lambda$ factors.

In Section 2 we go over the necessary background and prove a few preliminary results.  The new material is as follows.  In Section 2.1 we give the definition of a continuous field of operator systems, and we also give a working definition of a strict quantization of an operator system.  At the end of Section 2.3,  we discuss the adjoint $u^\dagger :\H\otimes L^2(A)\rightarrow \H$ of a unitary representation $u:\H \rightarrow \H \otimes A$ of $A$, and we give four computations in Proposition \ref{prop1} through Lemma \ref{lem5.2} that we use repeatly throughout this paper.  Finally, in Section 2.5 we extend Li's results from Section 8 of \cite{Li} to ergodic actions of compact quantum groups on order unit spaces. We also define length functions for compact quantum groups and give their connections to Lip-norms, extending the results of \cite{Rieffel2}.

The material from Section 3 is of its own interest.  Here we develop a theory of stabilizers for actions of quantum groups.  This theory works for general compact quantum groups, not just Kac-type or coamenable.  We find it useful to leave the category of compact quantum groups in this section. The result of Lemma \ref{lem1} and the proof of Theorem \ref{thm2} show that in a wide variety of concrete categories, a universal mapping $\Pi :A\rightarrow B$ stabilizing an element in a unitary representation of $A$ exists.  In Section 3.2 we investigate this universal mapping for the category of (full) compact quantum groups.  In Section 3.3 we look at this universal mapping for the category of order unit spaces.

 In Sections 4.1 and 4.2 we develop Berezin quantization for a ``coadjoint orbit'' $A^H$  of a Kac-type coamenable compact quantum group $A$. We see here that for the purposes of Berezin quantization, the correct stabilizer to use is the one arising from the category of order unit spaces.  Beginning in Section 4.3 we consider the interaction of this quantization with Lip-norms, and we prove the above theorem in Section 4.5.

 In Section 5 we develop and prove a generalization.  Intuitively, we associate to any ergodic action of $A$ on a C$^*$-algebra a certain noncommutative vector bundle whose fibres are isomorphic to $A^H$.  We extend our results on Berezin quantization and convergence in quantum Gromov-Hausdorff distance to these spaces as well.  We will see in Section 6 that these examples should be viewed as noncommutative orbifolds.

In Section 6 we give examples of compact quantum groups, and we explicitly compute stabilizers and Berezin quantizations from representations of them.  In particular, we show that our Berezin quantization commutes with the Rieffel deformation of compact quantum groups described by Wang \cite{Wang3}, and show that the quantum coadjoint orbit of $C(G)_{\hbar J}$ is given by the quantum homogeneous space $C(G/K)_{\hbar \theta}$ of Varilly \cite{Varilly}.  We also use the generalization of Section 5 to give an explicit strict quantization of the orbifold $S^2/\Z_n$.

Finally, in the appendix we consider whether or not stabilizers have Haar states.  We will show that they do for the representations considered in Section 4.  We will also give an example which shows that the coamenability assumption cannot be dropped from Theorem \ref{thm1}. 

I would like to thank Daniel Kaschek and Stefen Waldmann for helpful comments about their paper \cite{Kaschek1}, as well as Shuzhou Wang for many helpful discussions.  Finally, I would like to thank Marc Rieffel for endless help and support as my advisor.

   \end{section}

\begin{section}{Preliminaries}

\begin{subsection}{Order Unit Spaces and Operator Systems}

For our purposes, we define an order unit space to be a triple $(V,V^+,e)$, where $V$ is a complex $^*$-vector space, $V^+\subseteq V_{sa}$ is a cone with $V^+-V^+=V_{sa}$, and $e\in V^+$ is a distinguished vector such that for any $v\in V^+$ one has $(re-v)\in V^+$ for some $r\in \mathbb{R}^+$.  We refer to $e$ as the order unit of $V$.  As usual, $V^+$ induces a partial order on $V_{sa}$ by $v\leq w$ whenever $(w-v)\in V^+$.  The order unit then gives a so-called order seminorm on $V_{sa}$ by $\| v \| =\inf \{ r\in \mathbb{R}^+ \: | \: -re \leq v \leq re\} $. This seminorm is finite since $V_{sa}=V^+-V^+$.   It is shown in \cite{Paulsen} that there exist various extensions of this norm to all of $V$, but the particular extension will not make any difference to us. In fact, for most purposes we will only concern ourselves with the self-adjoint part of $V$.  We consider only Archimedian order unit spaces; that is, order unit spaces such that this induced seminorm is actually a norm.  It is shown in \cite{Paulsen} that any order unit space has a maximal quotient space which is Archimedian.  Note that when  many of the standard algebraic constructions (in particular, quotients) are applied to Archimedian spaces, 
 the resulting spaces need not be Archimedian. In such situations we will always tacitly replace the resulting spaces with their maximal Archimedian quotients.

It is important to note that the norm on $V_{sa}$ is a secondary construction. As such, it is quite possible in general that $V_{sa}$ might not be complete in the order norm.

The direct sum of two order unit spaces $(V,V^+,e_V)$ and $(W,W^+,e_W)$ is given by $(V\oplus W, V^+\oplus W^+, e_V\oplus e_W)$, where we interpret $V^+ \oplus W^+ =\{ v\oplus w \: | \: v\in V^+, w\in W^+\}$.  Likewise, the tensor product $V\otimes W$ is again an order unit space with cone $\{ \Sigma_i v_i\otimes w_i \: | \: v_i \in V^+, w_i \in W^+\}$. Again, since order unit spaces need not be complete, the algebraic tensor product is a perfectly good order unit space.

A morphism of order unit spaces is a linear map $\Phi :V\rightarrow W$ which satisfies $\Phi (V^+)\subseteq W^+$ and $\Phi (e_V)=e_W$. Since $V_{sa}=V^+-V^+$, a morphism is necessarily $*$-preserving.   A linear map such that $\Phi (V^+)\subseteq W^+$ will be called positive.  A linear map such that $\Phi (e_V)=e_W$ will be called unital.  It is known that a linear map between order unit spaces $\Psi :V\rightarrow W$ is a morphism if and only if it is unital, and $||\Psi ||=1$.  For a given order-unit space $(V,e)$, of particular interest are the morphisms from $(V,e)$ to the order-unit space $(\C ,1)$.  A positive unital linear functional $\varphi :V\rightarrow \C$ is called a state. The set of all states of $V$ is called the state space of $V$ and is denoted $S(V)$.  For Archimedian order-unit spaces, the state space of $V$ seperates the points of $V$ \cite{Paulsen}.

An order ideal of $V$ is a linear, $*$-closed subspace $K\subseteq V$ such that $k\in K\cap V^+$ and $v\in V$ with  $0\leq v\leq k$  imply $v\in K$. Clearly, the kernel of a positive unital map is an order ideal.  Conversely, if $K\subseteq V$ is an order ideal, then the natural quotient map $V\rightarrow V/K$ can be used to give an order unit structure on $V/K$, where an element of $V/K$ is defined to be positive if it has a positive preimage, and the order unit of $V/K$ is defined to be the image of $e_V$.

A related concept is that of an operator system \cite{Arveson}.  An operator system is a unital, $*$-closed linear subspace of a C$^*$-algebra which is closed in the norm topology.  It is known that operator systems have enough positive elements so that the positive elements generate the self-adjoint elements \cite{Arveson}. Thus an operator system is an order unit space.   We will see in Section 3 that our stabilizers may not be closed in the order norm, so they are just order unit spaces.  However, the corresponding spaces of fixed elements $A^H$ are operator systems, as are their finite dimensional approximations $\B^n$.  In the literature, one usually takes morphisms of operator systems to be completely positive maps \cite{Arveson}, but we emphasize that we are only working with positive maps.  Indeed, the Berezin adjoint $\bs $ we define below is not $2$-positive in general.

We will be primarily concerned with the idea of quantizing operator systems.  Strict quantizations in the C$^*$-setting are well-known, see \cite{book} for example.  By no means will we fully work out what one should mean by a strict quantization of an operator system, but we motivate and give  a simple definition which will suffice for our purposes.  First, we define a field of operator systems.

\begin{definition}\label{defcontfield}

Let $X$ be a locally compact topological space.  Let $\mathfrak{C}$ be an operator system, and let $\{ \mathfrak{U}_x \}_{x\in X}$ be a family of operator systems parameterized by $X$.  For each $x\in X$, let $\phi_x:\mathfrak{C}\rightarrow \mathfrak{U}_x$ be a distinguished positive  unital map.  Suppose that these maps satisfy the following properties:
\begin{itemize}
\item $||\gamma||= \sup_{x\in X} ||\phi_x(\gamma )|| $ for all $\gamma \in 
	\mathfrak{C}$, and 
\item for any $f\in C_0(X)$ and any $\gamma \in \mathfrak{C}$ there exists an 
	element $f\gamma \in \mathfrak{C}$ such that $\phi_x(f\gamma )= f(x)
	\phi_x(\gamma )$ for all $x\in X$.
\end{itemize}
Then we refer to $(\mathfrak{C},\{ \mathfrak{U}_x\} , \{ \phi_x\} )$ as a field of operator systems over $X$.  The elements of $\mathfrak{C}$ are called sections of the field.

Let $\gamma \in \mathfrak{C}$, and suppose in addition that the function $x\mapsto ||\phi_x(\gamma )||$ is continuous.  Then we call  $\gamma \in \mathfrak{C}$ a continuous section.  If all sections are continuous then we say that the field is continuous.

Similarly we define upper semicontinuous and lower semicontinuous sections and fields.
\end{definition}

We concern ourselves with one example.  Let $X=\N \cup \{\infty \}$.  Suppose that we have a family of operator systems $M_1, M_2,..., M_\infty$, and  we have two families of positive unital maps 
\begin{eqnarray*}
\s^n & : & M_n\rightarrow M_\infty , \text{ and} \\
\bs^n &  : & M_\infty \rightarrow M_n.
\end{eqnarray*}
Suppose further that the sequence of composition maps $\s^n \circ \bs^n :M_\infty\rightarrow M_\infty$ converges strongly to the identity on $M_\infty$.  Then we have the following proposition.

\begin{proposition}\label{confield1}
In the scenerio described above, let 
\[ \mathfrak{C} =\{ \xi \in \Pi_{x\in \N\cup\{\infty\} } M_x \: | \: \xi_\infty =
	\lim_{n\rightarrow \infty} \s^n(\xi_n) \} .\]
For each $x\in \N\cup\{\infty\}$, let $\phi_x:\mathfrak{C}\rightarrow M_x$ be the map which is evaluation at $x$.  Then $(\mathfrak{C} , \{ M_x\} , \{ \phi_x\})$ is an upper semicontinuous field of operator systems.

Moreover, for each $\omega \in M_\infty$ and each $x\in\N\cup\{\infty\}$, let 
\[ \mathcal{Q}_x (\omega ) =\left\{
	\begin{array}{ll} \omega & x=\infty \\
			  \bs^x (\omega ) & x\neq \infty . \end{array}\right. \]
Then for each $\omega \in M_\infty$, we have $\mathcal{Q}(\omega )$ is a continuous section of $  (\mathfrak{C} , \{ M_x\} , \{ \phi_x\})$. 

\end{proposition}

\begin{proof}
We begin by showing that $\mathfrak{C}$ is an operator system.  Since each of the maps $\s^n:M_n\rightarrow M_\infty$ is positive and unital, it is clear that $\mathfrak{C}$ is $*$-closed and contains the unit of $\Pi_{x\in \N\cup\{\infty\} } M_x$.  We must show it is norm closed.  Suppose $\xi^k \in \mathfrak{C}$ and $\xi^k$ converges to $ \xi \in \Pi_{x\in \N\cup\{\infty\} } M_x$ in norm. This means that $\sup_x ||\xi^k_x-\xi_x||\rightarrow 0$ as $k\rightarrow \infty$.  Then the estimate
\[ ||\xi_\infty -\s^n(\xi_n)||\leq ||\xi_\infty  -\xi^k_\infty ||+
	||\xi^k_\infty -\s^n(\xi^k_n)||+||\s^n(\xi^k_n -\xi_n )|| \]
shows that $\xi \in \mathfrak{C}$.  Thus $\mathfrak{C}$ is norm closed and so is an operator system. 

The ordering on $\Pi_{x\in \N\cup\{\infty\} } M_x$ is pointwise, so it's clear that each of the evaluation maps $\phi_x :\mathfrak{C}\rightarrow M_x$ is positive and unital.  The condition on the norm of $\mathfrak{C}$ is just the definition of the product norm, so we have the first condition of Definition \ref{defcontfield}.
If $f\in C(\N\cup\{\infty\})$ and $\xi\in \mathfrak{C}$, it's easy to see that the element $f\xi\in \Pi_{x\in \N\cup\{\infty\} } M_x$ defined by $(f\xi )_x=f(x)\xi_x$ is actually in $\mathfrak{C}$, so we have the second condition as well.  Thus $(\mathfrak{C} , \{ M_x\} , \{ \phi_x\})$ is a field of operator systems.

We check upper semicontinuity.  Let $\xi\in \mathfrak{C}$.  Then we have
\begin{eqnarray*}
||\xi_\infty || & = & \lim_{n\rightarrow \infty}||\s^n(\xi_n)|| \\
	& \leq & \liminf_{n\rightarrow \infty}||\xi_n||.
\end{eqnarray*}
Since $\infty$ is the only accumulation point of $\N\cup\{\infty\}$, upper semicontinuity follows.

Now let $\omega \in M_\infty$.  We have $\s^n (\mathcal{Q}_n(\omega ) )=(\s^n \circ \bs^n)(\omega )$, which converges to $\omega =\mathcal{Q}_\infty (\omega )$ by the assumptions of our scenerio.  Thus $\mathcal{Q}(\omega )$ is an element of $\mathfrak{
C}$.  
Finally, we have that 
\[
||\mathcal{Q}_n(\omega ) ||=||\bs^n(\omega )  ||
\leq ||\omega ||=||\mathcal{Q}_\infty (\omega )||. \]
Combined with the inequality above, we see that $\mathcal{Q}(\omega )$ is a continuous section.

\end{proof}

Let Plank's constant $\hbar$ be $1/x$ for $x\in \N\cup\{\infty\}$.  The statement that $\mathcal{Q}(\omega )$ is a continuous section of $(\mathfrak{C} , \{ M_x\} , \{ \phi_x\})$ intuitively says that the maps $\mathcal{Q}_n :M_\infty\rightarrow M_n$ give a strict quantization of $M_\infty$.  The only aspect we're missing is that one would like a statement about convergence of some Poisson brackets.  We will not attempt to understand Poisson brackets in this paper.   For our purposes, we give the following definition of a strict quantization of an operator system.  

\begin{definition}\label{defdeformquant}
Let $I\subseteq \R$ contain $0$ as an accumulation point.  A strict quantization of an operator system $M$ on $I$ consists of 
\begin{enumerate}
\item a continuous field of operator systems $(\mathfrak{C}, \{ \mathfrak{U}_k\}, \{\phi_k\} )$ over $I$, with $\mathfrak{U}_0=M$, and
\item a positive unital map $\mathcal{Q}:M\rightarrow \mathfrak{C}$ which satisfies $\mathcal{Q}_0(\omega )=\omega$ for all $\omega \in M$.
\end{enumerate}
\end{definition}

\begin{theorem}\label{deformquant}
 Let $M_1, M_2, ..., M_\infty$ be a collection of operator systems.  Let $\s^n:M_n\rightarrow M_\infty$ and $\bs^n:M_\infty\rightarrow M_n$ be two collections of positive unital maps such that the sequence of compositions $\s^n\circ \bs^n : M_\infty\rightarrow M_\infty$ converges strongly to the identity on $M_\infty$.  

Let $\hbar =1/x$ for $x\in \N\cup\{\infty\}$.  Then there is a continuous field of operator systems $(\mathfrak{B}, \mathfrak{U}_\hbar ,\phi_\hbar )$ such that $\mathfrak{U}_\hbar =M_x$ for each  $x\in \N\cup\{\infty\}$, and this continuous field together with the quantization maps
\[ \mathcal{Q}_\hbar (\omega )=\bs^n (\omega ) \hspace{.3in} \forall n\in \N, \forall \omega \in M_\infty \]
gives a strict quantization of the operator system $M_\infty$.  The operator system $\mathfrak{B}$ is the subsystem of continuous sections of the semicontinuous field $\mathfrak{C}$ of Proposition \ref{confield1}.
\end{theorem}

\begin{proof}
Most of the content of this theorem follows from Proposition \ref{confield1}.  The only thing we have left to check is that the subset $\mathfrak{B}$ of continuous sections of $\mathfrak{C}$ is an operator system.  It's clear that $\mathfrak{B}$ is $*$-closed and unital, so we only need to check that it is norm-closed.  If $\{\xi^k\}$ is a sequence in $\mathfrak{B}$ that converges in norm to $\xi\in \mathfrak{C}$, then in particular we have 
\[ \lim_{k\rightarrow\infty} \sup_{x\in \N\cup\{\infty\}} 
	\left| ||\xi^k_x ||-||\xi_x|| \right| =0. \]
For each $\xi^k\in\mathfrak{B}$ we have $\lim_{n\rightarrow\infty} ||\xi^k_n||=||\xi^k_\infty ||$.  Then the estimate
\[ \left| ||\xi_\infty || -||\xi_n|| \right| \leq
	\left| ||\xi_\infty || -||\xi^k_\infty || \right| +
	\left| ||\xi^k_\infty || -||\xi^k_n || \right| +
	\left| ||\xi^k_n || -||\xi_n || \right| \]
gives that $\lim_{n\rightarrow\infty} ||\xi_n||=||\xi_\infty ||$ as well.  Thus $\mathfrak{B}$ is norm closed and is an operator system.
\end{proof}

\end{subsection}

\begin{subsection}{Quantum Metric Spaces}

In \cite{Rieffel1} Rieffel gives the notion of compact quantum metric spaces, and defines the quantum Gromov-Hausdorff distance between them.  We recall the basic ideas and  results of that paper here.

An ordinary compact topological space $X$ can be recovered from its algebra of continuous functions $C(X)$ as the set of pure states endowed with the weak$^*$ topology.  A metric $\rho_X$ determines a Lipschitz seminorm $L$ on $C(X)$ defined by 
\[ L(f) = \sup_{x\neq y} \frac{|f(x)- f(y)|}{\rho_X (x,y)}. \]
It has the following properties:
\begin{itemize}
\item The set $\{ f\in C(X) \: | \: L(f)< \infty \}$ is $*$-closed and dense.
\item $L(f)=0$ if and only if $f$ is a constant function.
\end{itemize}  
Moreover, $L$ in turn induces a metric $\rho$ on the whole state space of $C(X)$.  For $\mu , \nu \in S(C(X))$ we let
\[ \rho (\mu , \nu )=\sup \{ |\mu (f)-\nu (f)| \: : \: L(f)\leq 1\}. \]
It's easy to check that $\rho $ coincides with $\rho_X$ on $X$, when $X$ is viewed as the set of pure states of $C(X)$.  It is also known that the topology induced by $\rho$ on $S(C(X))$ is the weak$^*$ topology.  In summary, the metric space $(X, \rho_X )$ can be completely recovered from the pair $(C(X), L)$, and $L$ has the additional property of inducing the weak$^*$ topology on the whole state space of $C(X)$.
This motivates the following definition.

\begin{definition}\label{defLip}
Let $(A,L)$ be a pair of an order unit space $A$ and a (not necessarily finite valued)  seminorm $L$ on $A_{sa}$.  We say that $L$ is a Lip-norm for $A$ if $L$ has the following three properties:
\begin{itemize}
\item The set $\{ a\in A_{sa}\: | \: L(a)<\infty \}$ is dense in $A_{sa}$ for the order-unit norm.
\item One has $L(a)=0$ if and only if $a$ is a scalar multiple of the order unit $e_A$.
\item Consider the metric $\rho_L$ on $S(A)$ given by
\[ \rho_L (\mu ,\nu )=\sup \{ |\mu (a)-\nu (a)|\: : \: L(a)\leq 1\}. \]
Then the topology induced by $\rho_L$ must be the weak$^*$ topology of $S(A)$.
\end{itemize}
Moreover, if $A$ is an order unit space and $L$ is a Lip-norm on $A$, then we call the pair $(A,L)$ a compact quantum metric space.
\end{definition}

Notice that we have defined a Lip-norm to be a seminorm on $A_{sa}$.  Many times a Lip-norm is defined to be a seminorm on all of $A$, with the additional property that $L(a)=L(a^*)$.  When $A$ is a C$^*$-algebra, $L$ is often required to be finite on a dense $*$-subalgebra of $A$.

In practice, the condition that the metric $\rho_L$ on $S(A)$ induces the weak$^*$ topology is difficult to verify directly.  However, there is a convienent condition which Rieffel shows is equivalent to it.

\begin{proposition}{\cite{Rieffel2}}\label{rief1}
 Let $A$ be an order unit space.  Let $L$ be a  seminorm on $A$ that is finite on a dense subset of $A$ and vanishes exactly on the scalars.  Let $\tilde{A}=A/{\C1_A}$ and let $\tilde{L}$ be the quotient seminorm of $L$ on $\tilde{A}$.  Let $\B =\{\tilde {a}\in \tilde{A} \: | \: \tilde{L}(\tilde{a})\leq 1\}$. Then we have:
\begin{itemize}  
\item $S(A)$ is bounded for the metric $\rho_L$ if and only if $\B$ is bounded in the quotient norm.
\item $\rho_L$ gives the weak$^*$ topology on $S(A)$ if and only if $\B$ is totally bounded in the quotient norm.    
\end{itemize}
\end{proposition}

Next we need a method to measure how different two quantum metric spaces are.  First, we recall the classical notions of Hausdorff distance and Gromov-Hausdorff distance.  

Let $X$ be a compact metric space, and let $Y$ and $Z$ be compact subsets of $X$.  Then the Hausdorff distance between $Y$ and $Z$ is the smallest $\delta \geq 0$ such that $Y$ is contained in the $\delta$ neighborhood of $Z$ and $Z$ is contained in the $\delta$ neighborhood of $Y$.  Since we consider only compact sets, the Hausdorff distance between two subsets of $X$ is zero if and only if they are equal.  Symmetry is built into the definition, and the triangle inequality is also clear, so that Hausdorff distance is a metric on the space of compact subsets of $X$.

Now suppose $(Y,\rho_Y)$ and $(Z,\rho_Z)$ are two compact metric spaces.  Let $X=Y\dot{\cup}Z$ be the topological disjoint union of $Y$ and $Z$.  For any metric $\rho$ on $X$ that restricts to $\rho_Y$ on $Y$ and $\rho_Z$ on $Z$, one can compute the Hausdorff distance between $Y$ and $Z$ with respect to $\rho$.  Then the Gromov-Hausdorff distance $dist^{GH}$ between $(Y,\rho_Y)$ and $(Z,\rho_Z)$ is defined to be the infimum of the Hausdorff distances between $Y$ and $Z$ with respect to all such $\rho$.  It is known that Gromov-Hausdorff distance is a metric on the isometry classes of compact metric spaces.  Furthermore, the space of all isometry classes of compact metric spaces is complete with respect to this metric.  Again, for details see \cite{Rieffel1}.

This motivates the definition of quantum Gromov-Hausdorff distance for quantum metric spaces.  Namely, let $(A,L_A)$ and $(B,L_B)$ be two quantum metric spaces. Then the state spaces $S(A)$ and $S(B)$ become ordinary compact metric spaces with the metrics $\rho_{L_A}$ and $\rho_{L_B}$ described above.   Consider the order unit space $A\oplus B$.  We have natural quotient maps from $A\oplus B$ to $A$ and to $B$.  Let $\mathcal{M}$ be the family of all Lip-norms $L$ on $A\oplus B$ such that the quotient seminorms of $L$ on $A$ and on $B$ are $L_A$ and $L_B$ respectively.  For any $L\in \mathcal{M}$ we can compute the Hausdorff distance between $S(A)$ and $S(B)$ inside $(S(A\oplus B),\rho_L)$.  Then the quantum Gromov-Hausdorff distance $dist_q^{GH}$ between $(A,L_A)$ and $(B,L_B)$ is defined as
\[ dist_q^{GH} (A,B) = \inf_{L\in \mathcal{M}} dist(S(A),S(B)). \]
  Rieffel shows that quantum Gromov-Hausdorff distance is a complete metric on the space of all isometry classes of compact quantum metric spaces.

In \cite{Rieffel3} Rieffel  uses Berezin quantization to construct a sequence of finite dimensional quantum metric spaces $(B^n, L_n)$ that converge in quantum Gromov-Hausdorff distance to the two-sphere, or any coadjoint orbit,  in a way that respects the $SU(2)$ symmetry of the sphere.  This is the result that we generalize in this paper.

\end{subsection}

\begin{subsection}{Compact Quantum Groups}

The basic objects we consider in this paper are compact quantum groups and their unitary representations.  We begin with the definitions.  (See \cite{Maes}.) 

\begin{definition}
A compact quantum group is a pair $(A,\Delta )$, where $A$ is a unital C$^*$-algebra and $\Delta$ is a $^*$-homomorphism $\Delta :A\rightarrow A\otimes A$ with the following properties:
\begin{itemize}
\item $(\Delta \otimes id )\Delta =(id\otimes \Delta )\Delta$
\item $\Delta A(1\otimes A)$ and $(A\otimes 1)\Delta A$ are both total in $A$.
\end{itemize}

It is known that there is a dense $^*$-subalgebra $\A \subset A$ and maps $\varepsilon :\A \rightarrow \C$ and $\kappa :\A \rightarrow \A$ that make $\A$ into a Hopf algebra.  We call $\A$ the polynomial subalgebra of $A$.  If in addition the coinverse $\kappa$ satisfies $\kappa^2=id$, then we call $A$ a compact quantum group of Kac-type.

We will assume that our quantum groups are full; that is, we have $A$ is the universal enveloping C$^*$-algebra of $\A$.  In this case we have that the counit $\varepsilon$ extends to all of $A$.  For $A$ of Kac-type, the coinverse $\kappa$ extends to all of $A$.
\end{definition}
It is known that if $A$ is any compact quantum group, there exists a full compact quantum group $A_u$ that has the same polynomial subalgebra \cite{Bedos1}, so our last condition is not restrictive.

Let $G$ be a compact group. The motivating example of a compact quantum group is $C(G)$ with coproduct $\Delta f(x,y)=f(xy)$, where we identify $C(G\times G)\cong C(G)\otimes C(G)$. The counit and coinverse are given by $\varepsilon (f)=f(e)$ and $\kappa f(x)=f(x^{-1})$.  In particular, we have that $C(G)$ is always of Kac-type.  If $\Gamma$ is a discrete group, then the universal group C$^*$-algebra $C^*(\Gamma )$ can be made into a quantum group with the coproduct $\Delta (\delta_\gamma )=\delta_\gamma \otimes \delta_\gamma$.  The counit is given by $\varepsilon (f)=\Sigma_{\gamma \in \Gamma} f(\gamma )$.  The coinverse is given by $\kappa (\delta_\gamma )=\delta_{\gamma^{-1}}$, so $C^*(\Gamma )$ is also of Kac-type.
Other examples are $\theta$-deformations of ordinary groups \cite{Rieffel5} \cite{Wang3}, the $q$-deformations of Lie groups \cite{Woro2}, and the universal compact quantum groups $A_u(Q)$ and $B_u(Q)$ of Wang \cite{Wang1}.  The $\theta$ deformations of a Kac-type quantum group are still of Kac-type, but the $q$-deformations of classical Lie groups are generally not of Kac-type.

It is known that for any compact quantum group there is a unique state $\h \in S(A)$ such that $(\h \otimes id)\Delta a=(id\otimes \h )\Delta a =\h (a)1_A$ for all $a\in A$.  This state is called the Haar state of $A$. \cite{Maes}  The Haar state is always faithful on the polynomial subalgebra $\A$.  If $\h$ is actually faithful on $A$, we call $A$ coamenable.  We use the notation $L^2(A)$ for the Hilbert space generated by the GNS-construction of $A$ with its Haar state.  The coproduct $\Delta :A\rightarrow A\otimes A$ induces a unitary representation of $A$ on $L^2(A)$, called the right regular representation. See \cite{Maes} for details. Then $A$ is coamenable if and only if the right regular representation is faithful.

In general, $\h$ is not a trace on $A$.  In fact, the Haar state is a trace exactly when $A$ is Kac-type \cite{Woro1}.

Let $a\in A$ and let $\phi ,\psi \in A'$.  Then there are notions of convolution defined by:
\begin{eqnarray}
\phi *a & = & (id\otimes \phi )\Delta a \nonumber \\
a *\phi & = & (\phi \otimes id)\Delta a \nonumber \\
\phi *\psi & = & (\phi \otimes \psi )\Delta \nonumber
\end{eqnarray}
We note that  $\phi *a$ and $a*\phi $ are elements of $A$, while $\phi * \psi $ is an element of $A'$.  The definitions are given such that any ``associative'' law one might write down holds.  For example, $\phi * (\psi * a)=(\phi *\psi )*a$.

Let $(A,\Delta_A)$ and $(B,\Delta_B)$ be two compact quantum groups.  A morphism from $A$ to $B$ is a C$^*$-algebra homomorphism $\Phi :A\rightarrow B$ such that $\Delta_B \circ \Phi =(\Phi \otimes \Phi)\Delta_A$.  If $\Phi$ is surjective, we say that $(\Phi ,B)$ is a quantum subgroup of $A$.  If there is no confusion about the ``embedding'' $\Phi$, we will simply say that $B$ is a quantum subgroup of $A$.

If $\Phi :A\rightarrow B$ is a morphism of quantum groups, then $K=\ker (\Phi )$ is a Hopf ideal of $A$.  That is, $K$ is a closed two-sided ideal with the property $\Delta (K)\subseteq (K\otimes A)\oplus (A\otimes K)$.  Conversely, if $K$ is a Hopf ideal of $A$, then $\Delta $ descends to a coproduct on $A/K$.  One can show that $\Delta (A/K)(1\otimes A/K)$ and $(A/K \otimes 1)\Delta (A/K)$ are both total in $A/K\otimes A/K$, so that $A/K$ is a quantum subgroup of $A$.  \cite{Pinzari1}

We now define unitary representations of compact quantum groups.  We use the same symbol $\langle, \rangle$ for the inner product on a Hilbert space $\H$, and the corresponding $A$-valued inner product  $ \langle  \eta \otimes a ,\theta \otimes b  \rangle= \langle  \eta ,\theta  \rangle b^*a$ on $\H \otimes A$.

\begin{definition}
Let $A$ be a compact quantum group.  A unitary representation of $A$ is a pair $(u,\H )$ where $\H$ is a Hilbert space and   $u:\H \rightarrow \H \otimes A$ is a linear map with the following properties:  
\begin{itemize}
\item $(u\otimes id )u=(id \otimes \Delta )u$
\item  $(id\otimes A)u(\H )$ is total in $\H \otimes A$
\item For all $\xi ,\eta \in \H$ , we have
$\langle u(\xi ),u(\eta ) \rangle = \langle \xi ,\eta  \rangle 1_A$.
\end{itemize}
\end{definition}

In the literature a  unitary representation of a compact quantum group  is typically defined as a pair $(u,\H )$ where $u\in M(\K (\H )\otimes A)$ is unitary with the property $(id\otimes \Delta)u =u_{(12)}u_{(13)}$. Here, $M$ denotes the multiplier algebra  and $(a\otimes b)_{(13)}=a\otimes 1\otimes b$ (and similarly for $u_{(12)}$).  We have a natural map $M(\K (\H )\otimes A)\hookrightarrow   B(\H )\otimes A$, so if $u$ is representation in this latter sense, we can view $u$  as a mapping $u:\H \rightarrow \H \otimes A$ by $u(\xi )=u(\xi \otimes 1)$.  One can check that this map has the properties described above, and every such mapping arises in this way \cite{Pinzari1}.  We prefer the above definition because it makes a representation look like a left action, which we define in \ref{defaction}.

\begin{definition}
Let $u:\H \rightarrow \H \otimes A$ and $v:\K \rightarrow \K \otimes A$ be two unitary representations of $A$.  Let $T\in B(\H ,\K )$.  We say $T$ intertwines $u$ and $v$ if one has $(T\otimes id)u = v \circ T$.  Two unitary representations are called equivalent if there is a unitary operator that intertwines them. 

A representation $(u, \H )$ is called irreducible if there is no proper projection $P\in B(\H )$ that intertwines $u$ with itself (that is, $u$ has no invariant subspaces).  The collection of all irreducible unitary representations of $A$ up to equivalance is denoted $\hat{A}$. 

\end{definition}

Again, we can define convolutions. Let $(u, \H )$ be a unitary representation of $A$. Let $\xi \in \H$, let $\phi \in A'$, and let $\psi \in \H '$.  We have
\begin{eqnarray}
\phi * \xi & = & (id\otimes \phi)u(\xi ) \in \H \nonumber \\
\xi * \psi & = & (\psi \otimes id)u(\xi )\in A \nonumber
\end{eqnarray}
Again, all associative laws hold.  If $A=C(G)$ is the algebra of  continuous functions on an ordinary group, and $\delta_g\in A'$ is evaluation at $g\in G$, then $\delta_g*\xi$ is just $u_g(\xi )$.  For this reason, we will often use the notation $u_\phi (\xi )=\phi * \xi$.  Similarly, if $\psi \in \H '$ we will often write $_\psi u (\xi )$ for $\xi *\psi$.

It is known that every unitary representation of a compact quantum group decomposes into a direct sum of irreducible representations.  Furthermore, all of the irreducible representations are finite dimensional \cite{Woro1} \cite{Maes}. For each $\beta \in \hat{A}$,  choose an orthonormal basis $e_1,...,e_n$ for the Hilbert space $H_\beta$ on which $A$ acts.  Define elements $u_{kj}^\beta$ of $\A$ by the formula $\beta (e_j)=\Sigma_k e_k\otimes u_{kj}^\beta $.
Then the collection of $\{ u_{kj}^\beta \}$ as $\beta$ runs through all irreducible unitary representations of $A$ forms an algebraic linear basis for $\A$; that is, every element of $\A$ can be written uniquely as a finite linear combination of the $u_{kj}^\beta$ \cite{Woro1}.
The decomposition of arbitrary unitary representations into irreducibles means that for any unitary representation $u:\H \rightarrow \H \otimes A$ of $A$ on a Hilbert space $\H$, we can find a basis $\{ e_{Ij}^\beta \}$ where $\beta$ runs through the  irreducible representations, $j$ is an index from $1,...,dim H_\beta$, and $I$ is an index from $1,..., mult_\beta (\H )$, and  the action of $A$ is then given by $u(e_{Ij}^\beta )= \Sigma_k e_{Ik}^\beta \otimes u_{kj}^\beta$ \cite{Maes}. The proof uses heavily the basis $\{ \rho_{kj}^\beta \}$ of $\A '$ dual to $\{ u_{kj}^\beta \}$, as well as the functionals $\rho^\beta =\Sigma_k \rho_{kk}^\beta$.  We will use these functionals for a similar argument below in \ref{action}.  

Our main use for this decomposition is in proving the following proposition.  In the case where the representation is the  regular representation of $A$ on $L^2(A)$, this result is called the strong right invariance of the Haar state \cite{Pinzari1}.  This statement follows immediately from using the basis described above using the fact $\kappa (u_{kj}^\beta)=(u_{jk}^\beta )^*$ \cite{Woro1}, but we give the computation since we use this proposition repeatedly throughout this paper.

\begin{proposition}\label{prop1}
Let $u:\H \rightarrow \H \otimes A$ be a unitary representation of a compact quantum group $A$.  Let $\xi, \eta \in \H$.  We have $\langle \xi \otimes 1, u(\eta ) \rangle = \langle (id\otimes \kappa )(u(\xi )), \eta \otimes 1 \rangle$.
\end{proposition}

\begin{proof}
Consider the basis $\{ e_{Ij}^\beta \}$ described above.  We compute:
\begin{eqnarray}
\langle e_{Ij}^\beta \otimes 1, u(e_{Kl}^\gamma )\rangle & = &
	\langle e_{Ij}^\beta \otimes 1, \Sigma_m e_{Km}^\gamma \otimes u_{ml}^\gamma 
	\rangle \nonumber \\ 
	& = &
	 \delta_{IK} \delta^{\beta \gamma}(u_{jl}^\gamma )^* \nonumber \\ 
& = &
	\langle \Sigma_m e_{Im}^\beta \otimes (u_{jm}^\gamma )^*,  e_{Kl}^\gamma 		\otimes 1 \rangle \nonumber \\ 
& = &
	\langle \Sigma_m e_{Im}^\beta \otimes \kappa (u_{mj}^\gamma ),  e_{Kl}^\gamma 		\otimes 1 \rangle \nonumber \\ 
& = &
 \langle (id\otimes \kappa )(u(e_{Ij}^\beta )), e_{Kl}^\gamma  \otimes 1 \rangle \nonumber
\end{eqnarray}
\end{proof}

A unitary representation $u:\H \rightarrow \H \otimes A$ induces a map $u:\H \rightarrow \H \otimes L^2(A)$.  We denote the inner product on $\H \otimes L^2(A)$ by $(\cdot , \cdot )$ to distinguish it from the $A$-valued inner product on $\H \otimes A$.  We can reinterpret the above proposition as:
\begin{corollary}\label{cor1}
Let $u:\H \rightarrow \H \otimes A$ be a unitary representation of a compact quantum group $A$.  Consider the corresponding map $u:\H \rightarrow \H \otimes L^2(A)$.  It's adjoint $u^\dagger :\H \otimes L^2(A) \rightarrow \H$ is given on the dense set $\H \otimes \A$  by
$ u^\dagger (\xi \otimes a) = (id\otimes \h )([(id\otimes \kappa )u (\xi )](1\otimes a))$.  In other words
$ u^\dagger = (id\otimes \h \circ \mathsf{m} )([(id\otimes \kappa )u]\otimes id)$.
\end{corollary}
\begin{proof}
Let $\xi , \eta \in \H$ and let $a\in \A$. We denote by $\phi_a\in A'$  the linear functional $\phi_a (b) = \h (ba)$.  We compute:
\begin{eqnarray}
\langle u^\dagger (\xi \otimes a ), \eta \rangle & = &
\left( (\xi \otimes a ), u(\eta )\right) \nonumber \\
& = & \phi_a (\langle \xi \otimes 1 , u(\eta )\rangle ) \nonumber \\
& = &
\phi_a (\langle (id\otimes \kappa )u(\xi ), \eta \otimes 1\rangle ) \nonumber \\
& = & 
\h (\langle [(id\otimes \kappa )u(\xi )](1\otimes a), \eta \otimes 1\rangle ) 
\nonumber \\ & = &
\langle (id\otimes \h )([(id \otimes \kappa )u(\xi )] (1\otimes a)), \eta \rangle . \nonumber 
\end{eqnarray}

\end{proof}

Note: We will frequently abuse notation and write $u^\dagger :\H \otimes A \rightarrow \H$ or $u^\dagger :\H \otimes \A \rightarrow \H$ since $A$ and $\A$ map naturally into $L^2(A)$.

We end with two computational lemmas that will greatly simplify calculations later on.

\begin{lemma}\label{lem5.1}
Let $u: \H \rightarrow \H \otimes A$ be a unitary representation  of $A$.   Consider the corresponding mapping $u :\H \rightarrow \H \otimes L^2(A)$ as well as  its adjoint $u^\dagger :\H \otimes L^2(A)\rightarrow \H$.  Then $u^\dagger \circ u =id_\H$.
\end{lemma}
\begin{proof}
We saw in  Corollary \ref{cor1} that $u^\dagger  = (id\otimes \h \circ \mathsf{m} )([(id\otimes \kappa )u]\otimes id)$.  Let $\xi \in \H$.  We compute: 
\begin{eqnarray*}
u^\dagger (u(\xi )) &  = &   (id\otimes \h \circ \mathsf{m} )((id\otimes \kappa \otimes          id )(u\otimes id)(u (\xi )) ) \\
& = &
       (id\otimes \h \circ \mathsf{m} )((id\otimes \kappa \otimes id )(id
       \otimes \Delta )(u (\xi )) ). 
\end{eqnarray*}
We know that $\mathsf{m} (\kappa \otimes id)\Delta = \varepsilon (\cdot )1_A$, so that $\h \circ \mathsf{m} (\kappa \otimes id)\Delta = \varepsilon$.  Inserting this into the above calculation gives $u^\dagger (u(\xi )) = (id \otimes \varepsilon )(u(\xi )) = \xi$.
\end{proof}
\begin{lemma}\label{lem5.2}
Let $u$ and $u^\dagger$ be as above.  Then $u\circ u^\dagger :\H \otimes L^2(A) \rightarrow \H \otimes L^2(A)$ is given by $u\circ u^\dagger = (u^\dagger \otimes id)(id\otimes \Delta )$ on the dense set $\H \otimes \A$.
\end{lemma}
\begin{proof}
$u^\dagger (\xi \otimes a) = (id\otimes \h )([(id\otimes \kappa )u(\xi )](1\otimes a)) =
(id\otimes \h)((1\otimes \kappa^{-1}a)u(\xi ))$, so that
\begin{eqnarray} u ( u^\dagger (\xi \otimes a)) & = & (id\otimes id\otimes \h)((1\otimes 1\otimes \kappa^{-1}a)(u\otimes id)u(\xi )) \nonumber \\
 & = & (id\otimes id\otimes \h)((1\otimes 1\otimes \kappa^{-1}a)(id\otimes \Delta )u(\xi )) \nonumber  \\
 & = & (id\otimes \kappa \otimes \h)((1\otimes \Delta (\kappa^{-1} a))u(\xi )_{13})  \nonumber \\
 & = & (id\otimes \kappa \otimes \h)((1\otimes \vartheta (\kappa^{-1} \otimes \kappa^{-1}) \Delta  a)u(\xi )_{13}) \nonumber \\
& = & (id\otimes \h \otimes \kappa )((1\otimes (\kappa^{-1} \otimes \kappa^{-1}) \Delta  a)u(\xi )_{12}) \nonumber \\
& = & (id\otimes \h \otimes id)(([(id\otimes \kappa )u(\xi )]\otimes 1)(1\otimes \Delta a)) \nonumber \\
& = & (u^\dagger \otimes id)(\xi \otimes \Delta a). \nonumber
\end{eqnarray}
In the fourth line, we have used the fact that $\Delta \circ \kappa = \vartheta (\kappa \otimes \kappa )\Delta$, where $\vartheta :A\otimes A\rightarrow A\otimes A$ is the flip map. 
\end{proof}

\end{subsection}

\begin{subsection}{Actions of Algebraic Quantum Groups}

We begin with the definition of an action of a compact quantum group on a C$^*$-algebra.  See \cite{Podles}.

\begin{definition}
Let $A$ be a compact quantum group, and let $B$ be a C$^*$-algebra.  An action of $A$ on $B$ is a C$^*$-homomorphism $\alpha :B\rightarrow B\otimes A$ such that
\begin{itemize}
\item $(\alpha \otimes id)\alpha =(id \otimes \Delta )\alpha$, and
\item $(id\otimes A )\alpha (B)$ is total in  $B\otimes A$.
\end{itemize}
When $\alpha :B\rightarrow B\otimes A$ is an action of $A$ on a C$^*$-algebra $B$, we will call $B$ an $A$-module algebra.
\end{definition}
 
\begin{remark} The above is actually a definition for a left action of $A$ on a C$^*$-algebra.  Until we come to Section 5, we will have no need for right actions of $A$, so we omit the adjective ``left'' until that time.

This choice of nominclature stresses that we are really viewing $A$ to be the set of functions on some imaginary space with a group structure.  In the literature for Hopf algebras, what we call a left action of the quantum group $A$ is typically called a right coaction of the Hopf algebra $A$.  
\end{remark}

We will use actions of quantum groups on C$^*$-algebras extensively in this paper, but we will also have need of actions of $A$ on order unit spaces.  There is a difficulty here, because an order unit space need not be complete in its order norm.  

A way around this difficulty is by working with algebraic quantum groups (of compact type).  The definition of an algebraic quantum group is essentially the same as the definition for a compact quantum group, except that tensor products all become algebraic, and no $*$-structure is assumed.  For our purposes, it is sufficient to think of an algebraic quantum group as just being the polynomial subalgebra of a compact quantum group. See \cite{Bedos2} for a more complete description.  Using algebraic tensor products, we can give the definition of an action of an algebraic quantum group on a general vector space, and prove the basic structure theorem for such actions.

\begin{definition}\label{defaction}
Let $\A$ be an algebraic quantum group, and let $V$ be a vector space.  A   action of $\A$ on $V$ is a linear map $\alpha :V\rightarrow V\otimes \A$ such that:
\begin{itemize}
\item $(\alpha \otimes id)\alpha =(id \otimes \Delta )\alpha$, and 
\item $(id\otimes \A )\alpha (V)$ spans $V\otimes \A$.
\end{itemize}
\end{definition}

\begin{remark}
If $\alpha :B\rightarrow B\otimes A$ is an action of a compact quantum group on a $C^*$-algebra $B$ as defined above, the argument of  \cite{Podles} shows that there is a dense subalgebra $\B$ of $B$ so that the restriction $\alpha :\B \rightarrow \B \otimes \A$ is an action of the algebraic quantum group $\A$.  Furthermore, since $\B$ is dense, one can recover the full action $\alpha :B\rightarrow B\otimes A$ from the restriction of $\alpha$ to $\B$.
\end{remark}

We now prove the basic structure theorem for actions of algebraic quantum groups on vector spaces.  The proof is almost verbatim from the corresponding result for actions of compact quantum groups on $C^*$-algebras \cite{Podles}.  We include it for completeness.

If $\alpha :V\rightarrow V\otimes \A$ is an action of an algebraic quantum group, we say a linear subspace $Y\subseteq V$ corresponds to $\nu \in \hat{\A}$ if $\dim (Y)=\dim (\nu )$ and there is a basis $e_1,...e_n$ for $Y$ such that $\alpha (e_i)=\Sigma_j e_j \otimes u^\nu_{ji}$ for all $i=1,...,n$.  We write $V_\nu$ for the sum (not direct) of subspaces $Y\subseteq V$ that correspond to $\nu \in \hat{A}$.  We call $V_\nu$ the isotypic component of $V$ corresponding to $\nu$.  By the definition we have $\alpha (V_\nu )\subseteq V\otimes \A_\nu$.  From the action property $(\alpha \otimes id)\alpha =(id\otimes \Delta )\alpha$ we see that in fact $\alpha (V_\nu )\subseteq V_\nu \otimes A_\nu$.

Consider the algebraic linear basis $\{ \rho_{ij}^\mu \}$ of $\A'$ that is dual to $\{ u_{ij}^\mu \}$.  We let $\rho^\nu =\Sigma_i \rho_{ii}^\nu$.  We now show that $(id \otimes \rho^\nu)\alpha$ is an idempotent map from $V$ onto $V_\nu$, and we use this to describe the structure of a general action of $\A$.

\begin{theorem}\label{action}
Let $\alpha :V\rightarrow V\otimes \A$ be an action of an algebraic quantum group of compact type on a vector space $V$.  For $\nu \in \hat{\A}$, let $E^{\nu}:V\rightarrow V$ be the projection $E^{\nu}=(id\otimes \rho^{\nu})\alpha $.  Let $W^{\nu}$ be the range of $E^{\nu}$.  Then
\begin{itemize}
\item  We have $V=\bigoplus_{\nu \in \hat{\A}} W^{\nu}$.
\item  For each  $\nu \in \hat{\A}$ there exists a set $I_\nu$ and vector subspaces $W^\nu_i$ for each $i\in  I_\nu$ such that
 \begin{itemize}
 \item  $W^\nu =\bigoplus_{i\in I_\nu}W^\nu_i$
 \item  $W^\nu_i$ corresponds to $\nu$ for each $i\in I_\nu$.
 \end{itemize}
\item Each subspace $Y\subseteq V$ corresponding to $\nu$ is contained in $W^\nu$.
\item The cardinal number of the set $I_\nu$ is independent of the choice of the $W^\nu_i$.
\end{itemize}
\end{theorem}

\begin{proof}

Define operators $E^\nu_{st}:V\rightarrow V$ by $E^\nu_{st}=(id\otimes \rho^\nu_{st})
\alpha$.  The linear functionals $\rho^\nu_{st}$ have the properties:
$\\ \rho^\nu_{st}*\rho^\mu_{mn}=\delta^{\nu \mu}\delta_{tm}\rho^\nu_{sn},$ and $
\\ \rho^\nu_{st}(a)=\h (x^\nu_{st}a)$, for some $x^\nu_{st}$'s that span $\A$. \cite{Woro1} $\\$
We compute:
\begin{eqnarray}
E^\nu_{st}E^\mu_{mn} & = & [(id\otimes \rho^\nu_{st})\alpha ][(id\otimes \rho^\mu_{mn})\alpha ] \nonumber \\
 & = & (id\otimes \rho^\nu_{st}\otimes \rho^\mu_{mn})(\alpha \otimes id)\alpha \nonumber \\
 & = & (id\otimes \rho^\nu_{st}\otimes \rho^\mu_{mn})(id\otimes \Delta )\alpha \nonumber \\
 & = & \delta^{\nu \mu}\delta_{tm} E^\nu_{sn}. \nonumber
\end{eqnarray}
In particular, $\{ E^\nu_{ss}\: | \: \nu \in \hat{\A}, s=1,...,\dim (\nu )\}$ is a collection of disjoint idempotents in $B(V)$.  Also we have
\begin{eqnarray}
& & span  \{ E^\nu_{st} (y) \: | \: \nu \in \hat{\A}; s,t=1,...,\dim (\nu ); y\in V\} =  \nonumber \\
& & span \{ (id\otimes \h )(id\otimes x^\nu_{st})\alpha (y)\: | \: \nu \in \hat{\A}; s,t=1,...,\dim (\nu ); y\in V\} =  \nonumber \\
& & span  \{ (id\otimes \h )(id\otimes a)\alpha (y)\: | \: a\in \A, y\in V\}  =  \nonumber \\
& &  (id \otimes \h)(V\otimes \A )  =  V. \nonumber
\end{eqnarray}

Let $W^{\nu ,s}=E^\nu_{ss}(V)$, and let $W^\nu =\bigoplus_s W^{\nu ,s}$.  We have $E^\nu_{st}(V)= (E^\nu_{ss} E^\nu_{st})(V)\subseteq  W^{\nu ,s}$, so that
\[ V=\bigoplus_{\nu ,s} W^{\nu ,s} = \bigoplus_\nu W^{\nu} ,\]
which was our first claim.

To establish the second claim, first notice that $E^\nu_{s1}E^\nu_{1s}=E^\nu_{ss}$ and $E^\nu_{1s}E^\nu_{s1}=E^\nu_{11}$, so that $E^\nu_{s1}:W^{\nu ,1}\rightarrow W^{\nu ,s}$ is a linear isomorphism.  Pick any basis $\{ e^\nu_{i1} \}_{i\in I_\nu }$ for $W^{\nu ,1}$, and let $e^\nu_{is}=E^\nu_{s1}(e^\nu_{i1})$.  Then $\{ e^\nu_{is} \}_{i\in I_\nu }$ is a basis for $W^{\nu ,s}$.  Finally, define $W^\nu_i =span \{ e_{is}^\nu \: | \: s=1,...,\dim (\nu )\}$.  Then we have
\[ W^\nu = \bigoplus_{i\in I_\nu} W^\nu_i, \]
which is half of the second claim.

Next we compute $\alpha (e_{is}^\nu )=\alpha (E^\nu_{ss}e_{is}^\nu )=
\alpha (id\otimes \rho^\nu_{ss})\alpha (e_{is}^\nu ) = (id\otimes id\otimes \rho^\nu_{ss})
(\alpha \otimes id)\alpha (e_{is}^\nu )=(id\otimes id\otimes \rho^\nu_{ss})
(id\otimes \Delta )\alpha (e_{is}^\nu ) = (id\otimes E^\nu_{ss})\alpha (e_{is}^\nu )$. 
That is, we have $\alpha (e_{is}^\nu )\in V\otimes span\{ u^\nu_{js}\: | \: j=1,...,\dim (\nu )\}$.  Now we can write $\alpha (e^\nu_{is})=\Sigma_j y_j\otimes u^\nu_{js}$ for some $y_j \in V$.  Applying $(id\otimes \rho^\nu_{ks})$ to both sides we obtain $y_k= E^\nu_{ks}(e^\nu_{is})= e^\nu_{ik}$, so that $\alpha (e^\nu_{is})=\Sigma_k e^\nu_{ik} \otimes u^\nu_{ks}$.  Therefore each of the subspaces $W^\nu_i$ correspond to $\nu$.  This finishes the second claim.

For the third claim, let $e_1,...,e_{\dim (\nu )}$ be a basis for a subspace $Y\subseteq V$ that corresponds to $\nu \in \hat{\A}$, so that $\alpha (e_s)=\Sigma_t e_t\otimes u^\nu_{ts}$.  We have $E^\nu_{r1}(e_1)=e_r$, so that $e_r=E^\nu_{rr}(e_r)\in W^{\nu ,r}\subseteq W^\nu$.  This proves the third claim.  Also, it shows each decomposition of $W^\nu$ can be obtained in the way described in the proof of the second claim, which proves the last claim as well.

\end{proof}

We can now define an action of a compact quantum group on an order unit space. Since order unit spaces need not be complete in the order norm, it is more natural to work at the algebraic level.

\begin{definition}\label{def1}
Let $A$ be a compact quantum group, and let $\A$ be its polynomial subalgebra. Let $V$ be an order unit space.  A   action of $A$ on $V$ is a positive unital map $\alpha :V\rightarrow V\otimes \A$ such that
\begin{itemize}
\item $(\alpha \otimes id)\alpha = (id\otimes \Delta )\alpha$, and
\item $(id\otimes \A )\alpha (V)$ spans $V\otimes \A$.
\end{itemize}
If $V$ is an order unit space and $\alpha :V\rightarrow V\otimes \A$ is an action of $A$ on $V$, then we call $V$ an ordered   $A$-module.
\end{definition}

Finally, we must define an ergodic action.  If $\alpha :V\rightarrow V\otimes \A$ is an action of $\A$ on a vector space $V$, we say $v\in V$ is fixed by $A$ if $\alpha (v)=v \otimes 1_A$.  Since an action of $A$ on an order unit space is defined to be unital, we have that the order unit $e_V$ is always fixed by $A$.  an action $\alpha :V\rightarrow V\otimes \A$ is called ergodic if $v\in V$ is fixed if and only if $v\in \C e_V$.

Note that if $\alpha :V\rightarrow V\otimes \A$ is any action and $v\in V$, then the vector $\alpha_\h (v)\in V$ is always fixed by $A$, where $\h \in S(A)$ is the Haar state of $A$.  In particular, if $\alpha$ is an ergodic action on an order unit space, then $\alpha_\h (v)$ is a scalar multiple of $e_V$ for any $v\in V$.

This brings us to the major distinction between actions of $A$ on order unit spaces versus actions on C$^*$-algebras.  Namely, there is a theorem of Boca that states if $B$ is an ergodic $A$-module algebra, then the isotypic components of $B$ are finite dimensional \cite{Boca}.  It is not clear whether this result holds for actions on order unit spaces.

\end{subsection}

\begin{subsection}{Lip-norms From Actions of Quantum Groups}

In Section eight of \cite{Li}, Li discusses how to produce Lip-norms on ergodic modules of compact quantum groups.  There, he works with actions of quantum groups on C$^*$-algebras, but most of his arguments work equally well for actions on order unit spaces.  We recall his results here, suitably adapted to our context, and we at least sketch the proofs.

We also show in this section that at least at the abstract level, one can construct Lip-norms on compact quantum groups using an analog of length functions as in \cite{Rieffel2}.  We won't need this result for the rest of the paper.  In fact, it's not totally clear how to produce length functions generally on compact quantum groups.  That being said, the results show that there is a nice similarity between the case of classical groups and our current scenario.

We begin with a brief review of Lip-norms coming from actions of compact quantum groups.  Recall that in our definition of a Lip-norm, a Lip-norm is a seminorm only defined for self-adjoint elements.
 As in Li's paper, the essential property we must have for a Lip norm on $A$ is that it be finite on $\A$.  Li calls such Lip-norms regular. The existance of regular Lip-norms is insured by the fact that $\A$ has a countable algebraic basis $\{ u_{kj}^\beta \}$.   A Lip-norm $L_A$ on $A$ is called right-invariant if $L_A(\psi *a) \leq L_A(a)$ for any $\psi \in S(A)$ and any $a\in A$ which is self adjoint.   In the case where $B$ is a C$^*$-algebra, the following is Theorem 1.4 of \cite{Li}.

\begin{proposition}\label{li1}
Let $L_A$ be a regular Lip-norm on a compact quantum group $A$, and let $\A$ be the polynomial subalgebra of $A$.  Let $\beta :B\rightarrow B\otimes \A$ be an ergodic action of $A$ on an order unit space $B$ which has finite dimensional isotypic components.  Define $L_B$ on  $B_{sa}$  by
\[
L_B(b) = \sup_{\psi \in S(B)} L_A( \: _\psi\beta (b) ).
\]
Then $L_B$ is a Lip-norm on $B$.  
Moreover, if the given Lip-norm $L_A$ is right-invariant, then $L_B$ is invariant, meaning $L_B(\beta_\varphi ( b)) \leq L_B(b)$ for any $\varphi \in S(A)$ and any $b\in B_{sa}$.
\end{proposition}
We will refer to $L_B$ as the Lip-norm on $B$ induced by $L_A$.  Unlike the case where $B$ is a C$^*$-algebra, the Lip-norm described above will actually be finite-valued.  This is because we defined actions on order unit spaces only at the algebraic level.

 As usual when working with Lip-norms, we denote $B/{\C 1}$ by $\tilde{B}$, the quotient norm on $\tilde{B}$ by $|| \cdot ||^\sim$,  and the image of $b\in B$ in $\tilde{B}$ by $\tilde{b}$. 

\begin{lemma}\label{stupid}
Let $B$ be an ergodic ordered $A$-module. For the seminorm defined as in \ref{li1}, we have
$\| \tilde{b} \|^\sim \leq 2r_AL_B(b)$ for any $b\in B_{sa}$, where $r_A$ is the radius of $L_A$.
\end{lemma}
\begin{proof}
A quick check using Proposition \ref{rief1} and the triangle inequality gives that for $a\in A$ we have $||a-\h (a)1_A||\leq 2r_AL_A(a)$.  Applying this inequality to $b*\psi$ for arbitrary $\psi \in S(B)$ and using that $\beta_\h (b)$ is a scalar (by ergodicity) gives the desired result.
\end{proof}

Let $\mathcal{S} \subset \hat{A}$ be finite.  We denote by $B_\mathcal{S}$ the direct sum of the $\nu$-isotypic components of $B$ for all $\nu \in \mathcal{S}$, and similarly for $A_\mathcal{S}$.  The proof of the next lemma is tough to summarize.  Since the statement just involves $A$ (and not our order unit space $B$), we simply refer the reader to Lemma 8.6 of \cite{Li} for the proof.

\begin{lemma}
Assume that $A$ is coamenable.  For any $\delta >0$ and $\varphi \in S(A)$ there exist $\phi \in S(A)$ and a finite subset $\mathcal{S}\subseteq \hat{A}$ such that $\phi$ vanishes on $A_\gamma$ for all $\gamma \in \hat{A} \setminus \mathcal{S}$ and 
\[ |(\varphi -\phi )(a)| \leq \delta L_A(a) \]
for all $a\in A_{sa}$.
\end{lemma}

The following lemma will allow us to prove \ref{li1}.  Also, it will be a key tool for us later on.

\begin{lemma}\label{li2}
Let $B$ be an ergodic ordered $A$-module, and suppose $A$ is coamenable.  Let $\delta >0$ be given.  Then there exists a finite $\mathcal{S} \subseteq \hat{A}$ and a $\phi \in S(A)$ with the following properties:
\begin{itemize}
\item     $\| \beta_\phi (b) \| \leq \| b \|$ for all $b\in B_{sa}$, 
\item       $\beta_\phi ( B)\subseteq B_{\mathcal{S}}$,
\item  $\| b-\beta_\phi (b) \| \leq \delta L_B(b)$ for all $b\in B_{sa}$.
\end{itemize}
Furthermore, $\phi$ and $\mathcal{S}$ can be chosen independent of $B$.
\end{lemma}
\begin{proof}
Use the previous lemma with the same $\delta$ and with $\varphi =\varepsilon$, the counit. Consider the $\phi \in S(A)$ obtained by the lemma.  Since $\phi$ is a state, the first result is clear.   For any $\gamma \in \hat{A}$ we have $\beta (B_\gamma)\subseteq B_\gamma \otimes A_\gamma$.  Since $\phi$ vanishes on $A_\gamma$ for $\gamma \notin \mathcal{S}$, we must have $\beta_\phi ( B)\subseteq B_\mathcal{S}$.  The last result is a simple calculation using   the previous lemma and the fact that $\beta_\varepsilon (b)=b$.
\end{proof}

We can now prove Proposition \ref{li1}.

\begin{proof}

We use Proposition \ref{rief1}.  Again, since we are really using algebraic actions of the algebraic quantum group $\A$,  it is clear that $L_B(b)$ is finite for any $b\in B$, since any $b$ will lie in some finite sum of isotypic components, and because the isotypic components themselves are assumed finite dimensional. 
Clearly $L_B$ vanishes exactly on $\C 1_B$. By Lemma \ref{stupid} we have $||\cdot ||^\sim \leq 2r_AL_B$, so that the $L_B$ unit ball of $\tilde{B}$ is a bounded set in the quotient norm.  For any $\delta >0$ let $\phi$ and $\mathcal{S}$ be as in Lemma \ref{li2}. Then the image of the $L_B$-unit ball of $\tilde{B}$ under $\beta_\phi $ is a bounded set in the finite dimensional space $B_\mathcal{S}$, so is totally bounded.  As $\delta >0$ was arbitrary, the $L_B$-unit ball of $\tilde{B}$ is totally bounded also, which verifies the condition of Proposition \ref{rief1}.

\end{proof}

This finishes the summary of Li's results that will be needed in this paper.  For future reference, we include a discussion on length-functions for compact quantum groups, and how they can be used to give left-invariant Lip-norms.  We begin by recalling the definition for regular groups.

Let $G$ be an ordinary group.  A length function $\l \in C(G)$ is a function with the following properties:
\begin{itemize}
\item $\l (x)\geq 0$, with equality if and only if $x=e$.
\item $\l (x)=\l (x^{-1})$ for all $x\in G$.
\item $\l (xy)\leq \l (x)+\l (y)$ for all $x, y\in G$.
\end{itemize}
If $(\pi , M)$ is an ergodic $G$-module, we have a Lip-norm on $M$ given by 
\[ L(m) =\sup_{x\neq e} \left( \frac{||\pi_x(m)-m||}{l(x)} \right) \]

Motivated by this, we give the following definition for a length function on a compact quantum group.

\begin{definition}
Let $A$ be a compact quantum group.  A length function on $A$ is an element $\l \in A$ such that:
\begin{itemize}
\item $\l \geq 0$, and $\mu (\l )=0$ for $\mu \in S(A)$ iff $\mu =\varepsilon$.
\item $(\mu * \nu )(\l )\leq \mu (\l )+\nu (\l )$ for all $\mu ,\nu \in S(A)$.
\end{itemize}
\end{definition}
Note that by convexity it suffices to check these   conditions for pure states.  In particular, length functions on ordinary groups satisfy the conditions of this definition.  Also, we do not assume an analog of the condition $l(x)=l(x^{-1})$.  For ordinary groups, this condition is imposed to make the metric $\rho_l(x,y)=l(x^{-1}y)$ symmetric, but it is actually not needed to get a Lip-norm on $C(G)$.  This is especially convenient in the quantum group case.  The natural condition would be that $l=\kappa (l)$, but $\kappa$ generally doesn't extend to all of $A$.  Taking $l\in \A$ won't work in general, since for $A$ cocommutative no element of $\A$ can satisfy the first condition (this follows from Peter-Weyl).

\begin{theorem}\label{lippy}
Let $A$ be a compact quantum group, and let $\l \in A$ be a length function on $A$.  Let $(\pi , M)$ be a ergodic ordered $A$-module with finite dimensional isotypic components.  Consider the seminorm on $M$ given by:
\[ L(m) =\sup_{\mu \in S(A),\mu \neq \varepsilon} \left( \frac{||\pi_\mu (m)-m||}{\mu (\l )}\right) \]
Then $L$ is a Lip-norm on $M$.  Moreover, if one obtains a Lip-norm $L_A$ on $A$ by means of this formula, then the Lip-norm on $M$ is the same as the one given in Proposition \ref{li1}.  Also, the Lip-norm $L_A$ is left-invariant.
\end{theorem}
As before, we note that it suffices to take the supremum over pure states.  We note that even though the seminorm constructed here agrees with the one in Proposition \ref{li1}, it is not sufficient to consider the case of the regular module $A$.  Since $A$ is not assumed to be coamenable, one cannot appeal to Li's result.

\begin{lemma}
Let $A$ be a compact quantum group with length function $\l$ and let $(\pi ,M)$ be an ergodic ordered $A$-module.  Consider the seminorm $L$ given above.  The state space $S(M)$ has radius at most $\h (\l )$ for $L$.
\end{lemma}
\begin{proof}
Since $\pi$ is an ergodic action, we have a conditional expectation $E:M\rightarrow \C e_M$ by $E(m)=\pi_\h (m)$.  By identifying $\C e_M$ with the scalars, we can view $E$ as a state on $M$.  Let $\phi \in S(M)$.  Then $\phi (\pi_\h (m))=\pi_\h (m)$.  Thus $|\phi (m)-E(m)|=|\phi (m-\pi_\h (m))|\leq ||m-\pi_\h(m)||\leq L(m) \h(\l ).$ This shows the radius of $S(M)$ is at most $\h (\l )$.
\end{proof}

The following is an analog of Lemma 2.5 from \cite{Rieffel2}.

\begin{lemma}
Let $\mu \in A '$ be a linear functional of the form $\mu (a)=\h (ab)$, some $b\in A$.    Then the operator $\pi_\mu :M\rightarrow M$  is compact.
\end{lemma}
\begin{proof}
Initially take $b \in \A $.  Then $\mu $ is a finite linear combination of the $\rho^\gamma_{ij}$.  Also we have a projection $E^\gamma :M\rightarrow M_\gamma$ by $E^\gamma (m)=\rho^\gamma *m$.  Note that $\rho^\gamma *\rho^\gamma_{ij}=\rho^\gamma_{ij}$.  Thus $E^\gamma (\pi_{\rho^\gamma_{ij}}(m))=\pi_{\rho^\gamma *\rho^\gamma_{ij}}(m)=\pi_{\rho^\gamma_{ij}}(m)$; that is, the image of $\pi_{\rho^\gamma_{ij}}$ is contained in $M_\gamma$.  Since we've assumed that $M_\gamma$ is finite dimensional, we have $\pi_{\rho^\gamma_{ij}}$ is finite rank for each $\rho^\gamma_{ij}$.  Thus $\pi_\mu$ is finite rank whenever $b\in \A$.  If $b$ is an arbitrary element of $A$ and if $\{ b_n\}$ is a sequence in $\A$ such that $\lim_{n\rightarrow \infty}b_n=b$, then for the corresponding $\mu_n$ we  have
\begin{eqnarray}
||\pi_{\mu_n}(m)-\pi_\mu (m)|| & = & ||(id\otimes \h )[\alpha (m) (1_M \otimes (b_n-b))]|| \nonumber \\ 
 & \leq & ||\alpha (m)|| \: || (1_M \otimes (b_n-b))|| \nonumber \\ 
 & \leq & || b_n-b|| \: ||m||, \nonumber 
\end{eqnarray}
so that $\pi_{\mu_n}\rightarrow \pi_\mu$ in norm as $n\rightarrow \infty$.  Thus $\pi_\mu$ is compact for any $\mu$ of the form $\mu (a)=\h (ab)$. 
\end{proof}
\begin{remark} The assumption on the form of $\mu$ actually says that $\mu$ is an element of the dual quantum group of $A$, which is a quantum group of discrete type.  
\end{remark}

We can now prove Theorem \ref{lippy}.

\begin{proof}
We verify that the ball $\B_1 =\{ m: ||m||\leq 1\: \text{and} \: L(m)\leq 1\}$ is totally bounded.  Let $\delta >0$.  Think of  $\l$ as a weak$^*$ continuous function on $S(A)$ by $\l (\phi )\equiv \phi (\l )$.  We have $l(\varepsilon )=0$.

Let $\gamma , \beta \in \hat{A}$ be arbitrary.  By equations 5.35 and 5.36 of \cite{Woro1}, the linear functionals $\rho^\gamma$ considered before are of the form stipulated in the previous lemma, and they have the property that 
\[ (id\otimes \rho^\gamma )u^\beta =\delta^{\gamma \beta}I_{B(H_\beta )} . \]
(This is just a rewording of things we've said before.)  The counit is given by
\[ (id\otimes \varepsilon )u^\beta =I_{B(H_\beta )} . \]
Let $\mathcal{F}(\hat{A} )$ be the set of all finite subsets of $\hat{A}$, ordered by inclusion.  Let $\mathcal{S} \in \mathcal{F}(\hat{A} )$, and 
let $\mu_\mathcal{S}=\Sigma_{\gamma \in \mathcal{S}} \rho^\gamma$.  Then $\mu_\mathcal{S}$ is of the form considered in the previous lemma, and the collection $\{ \mu_\mathcal{S} \}_{\mathcal{S}\in \mathcal{F}(\hat{A})}$ is an increasing net of positive linear functionals on $A$.  By the formulae above, this net converges weak$^*$ to the counit on the polynomial subalgebra $\A$.  We also see that each element of this net is bounded above by $\varepsilon$; in particular, these linear functionals are uniformly bounded in norm.  Therefore they actually converge weak$^*$ to the counit on all of $A$, and we can find a $\mu_\mathcal{S}$ such that $l(\mu_\mathcal{S})\equiv \mu_\mathcal{S} (l)<\delta/2$.

  By the lemma, $\pi_{\mu_\mathcal{S}}$ is compact, so that $\pi_{\mu_\mathcal{S}} (\B_1)$ is totally bounded and can be covered by finitely many balls of radius $\delta /2$.  For any $m\in \B_1$ we have
$||m-\pi_{\mu_\mathcal{S}} (m)||\leq L(m){\mu_\mathcal{S}} (\l )\leq \delta /2$.  Thus $\B_1$ can be covered by finitely many balls of radius $\delta$. This proves that $L$ is a Lip-norm.

Next we check that $L$ is the same Lip-norm on $M$ as the one constructed by Li.  We have 
\[ L_A(a) =\sup_{\mu \in S(A),\mu \neq \varepsilon} \left( \frac{||\Delta_\mu (a)-a||}{\mu (\l )}\right) \]
It's clear from the form of $L_A$ and $L$ that we will have 
$L(m)=\sup_{\varphi \in S(M)} L_A(m*\varphi )$ as long as we have 
$||m||=\sup_{\varphi \in S(M)} ||m*\varphi ||$ for all $m\in M_{sa}$.  One way to see this equality is to use the fact that the set of convex combinations $\{ \Sigma_i c_i \varphi_i \otimes \mu_i : c_i>0, \Sigma c_i=1; \varphi \in S(M) , \mu_i \in S(A)\}$ is dense in $S(M\otimes A)$.  (This follows from the Schmidt decomposition from physics).  Thus we have:
\begin{eqnarray}
\sup_{\varphi \in S(M)} ||m*\varphi || & = &
\sup_{\varphi \in S(M),\mu \in S(A)} |\mu ( m*\varphi )| \nonumber \\
& = & \sup_{\varphi_i \in S(M),\mu_i \in S(A)} |\Sigma_i c_i \mu_i ( m*\varphi_i )| \nonumber \\
& = & \sup_{\psi \in S(M\otimes A)} |\psi (\pi (m))| =|| \pi (m) ||. \nonumber
\end{eqnarray}
Thus we just need to show that $||m||=||\pi (m)||$.  Since $\pi$ is positive and unital, we have $||\pi (m)||\leq ||m||$.  The other inequality follows from $(id\otimes \varepsilon )(\pi (m))=m$. 

Finally, applying this last result to the regular module, we have that $L_A$ is induced from $L_A$ when $A$ is viewed as an ergodic ordered $A$-module.  This says exactly that $L_A$ is left-invariant in the sense of \cite{Li}. 
\end{proof}

It would be interesting to find conditions on the length function $l$ that would insure that $L_A$ is also right-invariant, which would give that $L$ is invariant for any ergodic module $(\pi ,M)$.  For ordinary length functions, the correct condition is $l(x^{-1}yx)=l(y)$.  To find the corresponding condition in the quantum group case one needs to determine how to solve the equation $\mu *\nu =\nu *\omega$ for $\omega$ where $\mu , \nu ,$ and $\omega$ are (pure) states of $A$.

We also note that given  a right action of $A$ on an order unit space $\alpha :N\rightarrow A\otimes N$ with finite dimensional isotypic components, then a length function $l\in A$ will also induce a Lip-norm on $N$.  In particular, we have a Lip-norm on $A$ given by 
\[ L_A'(a) =\sup_{\mu \in S(A), \mu \neq \varepsilon} 
	\left( \frac{|| \: _\mu\Delta (a)-a||}{\mu (l)} \right)   \]
which is right-invariant.

\end{subsection}

\begin{subsection}{Berezin Quantization of Coadjoint Orbits}

In this section we quickly review the Berezin Quantization of coadjoint orbits of compact Lie groups.  We give no proofs, since we will begin proving analogs of these results for compact quantum groups in Section 4.  We include this material so that one can compare our later results with the case of an ordinary compact group. For more details, see \cite{Rieffel3}.

Let $G$ be a compact group, and let $(U,\H )$ be an irreducible unitary representation of $G$.  Then there is an ergodic action $\alpha$ of $G$ on $B(\H )$ by conjugation.  Let $P\in B(\H )$ be a rank 1 projection.  Then we can define the Berezin covariant symbol with respect to $P$ as a map $\s :B(\H )\rightarrow C(G)$ by the formula:
\[ \s_T(x) =tr(T\alpha_x(P)) \hspace{2in} T\in B(\H ), x\in G. \]
Let $H<G$ be the stabilizer of $P$ under the action $\alpha$.  Then $\s$ can be viewed as a map $\s : B(\H )\rightarrow C(G/H)$.  It is positive unital, and it intertwines the action $\alpha$ with the left translation action of $G$ on $C(G/H)$.   

If we consider the Hilbert-Schmidt inner product on $B(\H )$ with normalized trace, and the normalized $G$-invariant measure on $G/H$, then we can ask what is the adjoint of the map $\s :B(\H )\rightarrow L^2(G/H)$.  A simple computation gives the formula:
\[ \bs_f =d_\H \int_{G/H} f(x) \alpha_x(P) dx \hspace{2in} f\in L^2(G/H), \]
where $d_\H$ is the dimension of $\H$.

Finally, the composition $\s \circ \bs :C(G/H)\rightarrow C(G/H)$ is called the Berezin transform.  It is given explicitly by 
\[  (\s \circ \bs ) (f)(x)=\int_G f(xy)h_P(y)\: dy \]
where $h_P(x)=d_\H tr(P\alpha_x (P))$.  In \cite{Rieffel3} it was noted that since $h_P(x)=h_P(x^{-1})$, the expression on the right hand side above is just the convolution product of the functions $f$ and $h_P$.  However, for our purposes below, we will find it necessary to view things in terms of the measure $h_P(y) \; dy$.  The right hand side above becomes $[(id\otimes h_P)\Delta f](x)$, so using the notion of convolution for linear functionals on quantum groups above, we see that $(\s \circ \bs)(f)=h_P*f$, where we view $h_P$ as a linear functional on $C(G)$.

That is the story of the Berezin Symbol, it's adjoint, and the Berezin Transform for coadjoint orbits of compact groups.  To get a quantization of $C(G/H)$ we need a sequence of maps $\bs^n:C(G/H)\rightarrow B(\H^n)$.  For Rieffel's condition \cite{book} for strict quantizations to hold, we must have the dimensions of the Hilbert spaces $\H^n$ tending to infinity.  If we are to use Berezin's maps for the quantizations, we must find rank one projections in each $B(\H^n)$ which have the same stabilizer $H<G$.  This can be accomplished by taking $\H^n=\H^{\otimes n}$ and $P^n=P^{\otimes n}$.  However, there is a difficulty here since the representation $(U^{\otimes n},\H^{\otimes n})$ is not irreducible.  

There is a way around this problem for Lie groups.  If $G$ is a semisimple Lie group, then we can choose $P\in B(\H )$ to be the projection corresponding to a highest weight vector $\xi$.  If $\xi$ has weight $\omega$, then it is known that the representation of $G$ on the cyclic subspace generated by $\xi^{\otimes n}\in \H^{\otimes n}$ is an irreducible representation of $G$ with highest weight vector $\xi^{\otimes n}$ and highest weight $n\omega$.  We then take $\H^n$ to be the cyclic subspace of $\H^{\otimes n}$ generated by $\xi^{\otimes n}$, and we take $P^n=P^{\otimes n}\in B(\H^n)\subseteq B(\H^{\otimes n})$.  We now have a sequence of irreducible representations of $G$, with a sequence of rank one projections that each have stabilizer $H<G$.  Furthermore, it is known that the stabilizer of $P$ is the same as the stabilizer of $\omega$ under the coadjoint action, so that $C(G/H)$ is actually a coadjoint orbit of $G$.

Now we can consider the collection of maps $\bs^n:C(G/H)\rightarrow B(\H^n)$.  If we let $\hbar = 1/n$, it can be shown that this collection of maps actually creates a strict quantization of the C$^*$-algebra $C(G/H)$, and this construction is the Berezin quantization of a coadjoint orbit.  In \cite{Rieffel3}, Rieffel constructs a sequence of $G$-invariant Lip-norms $L_n$ on $B(\H^n)$, such that $(B(\H^n),L_n)$ converges to $C(G/H)$ with its usual $G$-invariant metric in quantum Gromov-Hausdorff distance.

\end{subsection}

\end{section}

\begin{section}{Stabilizers of Quantum Group Actions}

Orbits and stabilizers are absolutely fundamental in the study of group actions.  However, the definitions of both strongly involve appealing to the points of the group as a set.  For actions of quantum groups, it's not entirely clear what orbits and stabilizers should be, because one cannot really appeal to the points in ``the underlying group." Futhermore, orbits and stabilizers have hardly been discussed at all in the literature.  We have no need for orbits in this paper, but a notion of a  stabilizer will be essential.  There are two natural definitions that can be used, depending on whether one wants to view the stabilizer as a compact quantum subgroup, or as a set of points that fix the given element.  We give the definitions here.  Let $u :\H \rightarrow \H \otimes A$ be a unitary representation of compact quantum group $A$ on a Hilbert space $\H$.  We consider the category of pointed vector spaces. That is, the category of pairs $(B,1_B)$ such that $1_B\in B$ is a distinguished nonzero vector which we call the unit of $B$, and morphisms being unit preserving linear maps.
\begin{definition}\label{def2}
Let $\xi \in \H $.  Let $\Pi : A\rightarrow B$ be a surjective morphism in some subcategory of the category of pointed vector spaces.  We say that $\Pi$ stabilizes $\xi$ if $(id \otimes \Pi)(u(\xi )) = \xi \otimes 1_B$.   If there is no confusion about the quotient map, we will say $B$ stabilizes $\xi$.
\end{definition}

\begin{definition}\label{def3}
Let $u, \H ,\xi ,$ and $A$ be as above.  The quantum group stabilizer of $\xi$ is the universal quantum subgroup of $A$ that stabilizes $\xi$.
\end{definition}

\begin{definition}\label{def4}
Let $u, \H ,\xi ,$ and $A$ be as above.  The state stabilizer of $\xi$ is the universal quotient order unit space of $A$ that stabilizes $\xi$.
\end{definition}
As for the nominclature in the last definition, we will see in Proposition \ref{prop3} that via $\Pi$ the state space of the state stabilizer of $\xi$ can be identified with $\{ \varphi \in S(A) \: |\: u_\varphi (\xi )=\xi \}$.  Note also that the quantum group stabilizer of $\xi$ is a quotient of the state stabilizer (as an order unit space).

Although it is clear that the set of quantum group morphisms $\Pi : A\rightarrow B$ form the objects of a category (whose morphisms are commutative diagrams), it is not obvious that there exists a universal object.  Similarly for the set of positive unital maps.  In the following subsections we will prove the existance of the quantum group stabilizer and state stabilizer.

Both the quantum group stabilizer and the state stabilizer of $\xi$ have properties that resemble the usual properties for stabilizers.  It turns out for the purposes of this paper, the state stabilizer gives the results we desire.  This shouldn't be surprising, since order unit spaces are the appropriate category to consider for quantum Gromov-Hausdorfff distance.  However, quantum group stabilizers will likely be important for other purposes, so we develop a few of their properties here.

\begin{subsection}{General Results}

Let $(u, \H )$ be a unitary representation of $A$, and let $\xi \in \H$. In this section we show that in a wide variety of categories, a universal morphism $\Pi :A\rightarrow B$ stabilizing $\xi$ exists. In the following lemma, $B$ is an object in the category of pointed vector spaces.  

\begin{lemma}\label{lem1}
  Let $\mathcal{J}_\xi \subseteq A$ be the set $\mathcal{J}_\xi=
 \{ \langle u(\xi ), \eta \otimes 1_A \rangle -  \langle  \xi ,\eta  \rangle 1_A \: |\: \eta \in \H \}$. Let $\Pi :A\rightarrow B$ be a linear, unital map. Then $\Pi :A\rightarrow B$ stabilizes $\xi$ if and only if $\mathcal{J}_\xi \subseteq ker(\Pi )$.
\end{lemma}

\begin{proof}
Suppose that $\Pi :A\rightarrow B$ stabilizes $\xi$.  We have $(id \otimes \Pi)(u(\xi )) = \xi \otimes 1_B$.  Take $\eta \in \H$, and apply $ (\langle \cdot , \eta \rangle \otimes id)$ to both sides.  We obtain
$\Pi (\langle u(\xi ), \eta \otimes 1_A \rangle) = \langle  \xi ,\eta  \rangle 1_B$.  Since $\Pi$ is linear and unital,
$\Pi (\langle u(\xi ), \eta \otimes 1_A \rangle) - \langle  \xi ,\eta  \rangle 1_A)=0$, so $(\langle u(\xi ), \eta \otimes 1_A \rangle) - \langle  \xi ,\eta  \rangle 1_A)\in ker(\Pi )$.  The converse is similar.

\end{proof}

It's important to notice that the set $\mathcal{J}_\xi$ is more than just a linear subspace of $A$.  It is a 2-sided coideal.  The following proposition is similar to  Lemma 5.4  from \cite{Jurco}.  It will make the computation of quantum group stabilizers much easier later.

\begin{proposition}\label{lem3}
The set $\mathcal{J}_\xi$ given above is a two-sided coideal of $A$. That is, we have $\Delta (\mathcal{J}_\xi)\subseteq (\mathcal{J}_\xi \otimes A) \oplus (A\otimes \mathcal{J}_\xi)$.
\end{proposition}
\begin{proof}
It's actually a little easier to work with elements of $\kappa ^{-1}(\mathcal{J}_\xi)$ rather than elements of $\mathcal{J}_\xi$ itself.  Since $u$ is unitary, a generic element of $\kappa ^{-1}(\mathcal{J}_\xi)$ is given by
$(\langle \xi \otimes 1_A, u(\eta ) \rangle) - \langle  \xi ,\eta  \rangle 1_A)$.  We compute:
\begin{eqnarray}
 \Delta (\langle \xi \otimes 1_A, u(\eta ) \rangle) & - & \langle  \xi  ,\eta  \rangle 1_A)  = 
         \langle  \xi \otimes 1_A \otimes 1_A,(u\otimes id) u(\eta ) \rangle
       - \langle  \xi ,\eta  \rangle (1_A \otimes 1_A)   \nonumber  \\
     &   = & \{ ( \langle \xi \otimes 1_A, u(\cdot ) \rangle -
       \langle \xi , \cdot \rangle
       1_A) \otimes id\}
        (u(\eta )) +
       1_A\otimes (\langle  \xi \otimes 1_A,u(\eta )
       \rangle - \langle  \xi , \eta  \rangle 1_A)   \nonumber \\
      &  \in & (\kappa ^{-1}(\mathcal{J}_\xi) \otimes A)\oplus (A\otimes \kappa ^{-1}(\mathcal{J}_\xi)) \nonumber
\end{eqnarray}
Of course, this is all done for $\kappa ^{-1}(\mathcal{J}_\xi)$.  But we have $\Delta \circ \kappa = \vartheta (\kappa \otimes \kappa )\Delta$, where $\vartheta :A\otimes A \rightarrow A\otimes A$ is the flip map.  Thus $\Delta (\mathcal{J}_\xi) \subseteq
(\mathcal{J}_\xi \otimes A)\oplus (A\otimes \mathcal{J}_\xi)$ as well.
\end{proof}

Since we will usually be working with positive unital maps, we will not need to consider the entire set $\mathcal{J}_\xi$.  In fact, for a positive unital map $\Pi :A\rightarrow B$, it suffices to check the value of $\Pi$ on a single element.

\begin{proposition}\label{cool}

Suppose that  $B$ is a  normed space and that $\Pi :A\rightarrow B$ is a norm nonincreasing unital map.  Let $\s_\xi \in A$ be the element 
$\s_\xi =\langle  u(\xi ), \xi \otimes 1_A \rangle$.  Then $\Pi$ stabilizes $\xi$ if and only if $\Pi (\s_\xi )=||\xi ||^2 1_B$.

\end{proposition}

\begin{proof}

Since $u$ is unitary we have $||u(\xi )||^2=||\xi ||^2$.  Since $\Pi$ is norm nonincreasing, we have $\\ ||(id\otimes \Pi)u(\xi )||^2\leq ||\xi ||^2$.  But $\Pi (\sigma_\xi)=||\xi ||^2 1_B$ says exactly that
 $ \langle  (id\otimes \Pi)u(\xi ), \xi \otimes 1_B \rangle = ||\xi ||^2 1_B$.  Thus the Cauchy-Schwarz inequality gives $(id\otimes \Pi)u(\xi )=\xi \otimes 1_B$, so $\Pi$ stabilizes $\xi$.    Conversely, if $\Pi$ stabilizes $\xi$ then $\Pi (\sigma_\xi-||\xi ||^2 1_A)=0$ since $(\sigma_\xi-||\xi ||^2 1_A )\in  \mathcal{J}_\xi$.

\end{proof}

\begin{remark}
The notation $\s_\xi$ is chosen here  deliberately to match the notation for the Berezin symbol used later.
\end{remark}

We want to show that the quantum group stabilizer and the state stabilizer of $\xi$ both exist.  A quotient map $A\rightarrow A/K$ of a compact quantum group $A$ is a map of quantum groups if and only if the kernel $K$ is a Hopf ideal. Similarly, a quotient map $A\rightarrow A/K$ of an order unit space is a map of order unit spaces if and only if the kernel $K$ is an order ideal. Both Hopf ideals and order ideals have the property that the intersection of an arbitrary family of Hopf ideals (order ideals) is a Hopf ideal (order ideal). Thus there exists a minimal Hopf ideal (order ideal) generated by the set $\mathcal{J}_\xi$ above.  This reasoning and  Lemma \ref{lem3} above immediately give the following:

\begin{theorem}\label{thm2}
Let $u:\H \rightarrow \H \otimes A$ be a unitary representation of a compact quantum group $A$, and let $\xi \in \H$. The quantum group stabilizer of $\xi$ and the state stabilizer of $\xi$ both exist.

 Let $\mathcal{J}_\xi =\{ \langle u(\xi ), \eta \otimes 1_A \rangle -  \langle  \xi ,\eta  \rangle 1_A \: |\: \eta \in \H \}$. Let $H( \mathcal{J}_\xi )$ be the Hopf ideal generated by $\mathcal{J}_\xi$, and let $O( \mathcal{J}_\xi )$ be the order ideal generated by $\mathcal{J}_\xi$. Then we have:$\\$
a) The natural map $A\rightarrow A/H( \mathcal{J}_\xi )$ is the quantum stabilizer of $\xi$. $\\$
b) The natural map $A\rightarrow A/O( \mathcal{J}_\xi )$ is the state stabilizer of $\xi$.

\end{theorem}

This theorem gives us an abstract realization of the state stabilizer and the quantum group stabilizer. However, in practice it is quite difficult to understand what is the order ideal or the Hopf ideal generated by some subset of $A$. In the next two sections, we will find more concrete versions.

\begin{remark}
Consider an arbitrary subcategory of the category of pointed vector spaces.  This argument shows that a universal morphism $\Pi :A\rightarrow B$ stabilizing $\xi$ exists if the subcategory has the property that the intersection of any arbitrary collection of kernels is a kernel.
\end{remark}

\end{subsection}

\begin{subsection}{Quantum Group Stabilizers}

For purposes of quantum Gromov-Hausdorff distance, the most elementary category to work in is that of order unit spaces.  So for the purposes of this paper the state stabilizer is more useful than the quantum group stabilizer.  However, the  quantum group stabilizer is technically simpler to work with, so it is good to start with. It also has many nice properties, so we record them here for future use.
Throughout this section  $u:\H \rightarrow \H \otimes A$ is a unitary representation, and $\xi  $ is a distinguished vector in $\H$.

 First, we compute the quantum group stabilizers for $A$ commutative or cocommutative.

\begin{example}
Let $(U,\H )$ be a unitary representation of a compact group $G$, and let $A=C(G)$. Let $\xi \in \H$. View $U:\H \rightarrow \H \otimes C(G)=C(G\rightarrow \H ).$  With this view, we have $U(\xi )(x)=U_x(\xi )$.
Let $\Pi :C(G)\rightarrow B$ be a surjective morphism of quantum groups stabilizing $\xi$.  Since $\Pi$ is a surjective homomorphism, we have that $B$ is isomorphic to $C(H)$ for some $H<G$, and under this identification the morphism $\Pi$ becomes the restriction map $|_H:C(G)\rightarrow C(H)$. Since $C(H)$ stabilizes $\xi$ we have that $U(\xi )|_H (x)=\xi$; that is, $U(\xi )(y)=\xi$ for all $y\in H$. This just says that $H<G$ is contained in the the usual stabilizer of $\xi$.

Let $H<G$ be the usual stabilizer of $\xi$. A similar argument shows that the restriction morphism $|_H :C(G)\rightarrow C(H)$ stabilizes $\xi$. Thus $C(H)$ is the quantum group stabilizer of $\xi$.
\end{example}

\begin{example}
Let $\Gamma$ be a discrete group, and let $A=C^*(\Gamma)$ be the full group C$^*$-algebra of $\Gamma$. Let $u\in M_n(A)$ be an n-dimensional unitary (co)representation.  All irreducible unitary (co)representations of $A$ are one dimensional, and are given by elements of $\Gamma$. Thus we have that $u$ is equivalent to a diagonal matrix $diag[\gamma_1, \gamma_2, ..., \gamma_n]$ for some elements $\gamma_1,...,\gamma_n\in \Gamma$. If $\xi \in \C^n$, we have $u(\xi )_i =\xi_i\gamma_i$ for $i=1,...,n$.

Let $\Pi :C^*(\Gamma )\rightarrow B$ be a quantum group morphism stabilizing $\xi$. Up to isomorphism, we must have $B=C^*(\Gamma /N)$ for some normal subgroup $N \lhd \Gamma$, and $\Pi$ becomes the natural map induced from the quotient group homomorphism $\pi : \Gamma \rightarrow \Gamma /N$. Since $\Pi$ stabilizes $\xi$, we have $\xi_i \pi (\gamma_i )=\xi_i e_{\Gamma /N}$. Suppose $\xi$ is generic, so that $\xi_i\neq 0$ for all $i$. Then we have $\pi (\gamma_i)=e_{\Gamma /N}$ for all $i=1,...,n$.

Conversely, if we let $N\lhd \Gamma$ be the normal subgroup generated by $\{ \gamma_1,...,\gamma_n\}$, then a similar argument gives that the natural map $C^*(\Gamma )\rightarrow C^*(\Gamma /N)$ stabilizes $\xi$. So we have $C^*(\Gamma /N)$ is the quantum group stabilizer of $\xi$.

\end{example}

We saw before than the quantum group stabilizer of $\xi \in \H$ is given by the natural quotient map $A\rightarrow A/H(\mathcal{J}_\xi)$, where $H(\mathcal{J}_\xi)$ is the Hopf ideal generated by a specific subset $\mathcal{J}_\xi$.  We now show that it suffices to compute the closed 2-sided ideal generated by $\mathcal{J}_\xi$.

\begin{lemma}\label{cor2}
Let $\langle \mathcal{J}_\xi \rangle$ be the two-sided ideal generated by $\mathcal{J}_\xi$. Then $\langle \mathcal{J}_\xi \rangle$ is a Hopf ideal.
\end{lemma}
\begin{proof}
Since $\Delta :A\rightarrow A\otimes A$ is a $*$-homomorphism, this result follows from the Proposition \ref{lem3}.
\end{proof}

\begin{proposition}\label{cor3}
The universal C$^*$-algebra stabilizing $\xi$ is in fact a quantum group.  In other words, $A/\langle \mathcal{J}_\xi \rangle$ is the quantum group stabilizer of $\xi$.
\end{proposition}
\begin{proof}
By the same reasoning as in Theorem \ref{thm2}, we have that the natural map $A\rightarrow A/\langle \mathcal{J}_\xi \rangle$ is the universal map of unital C$^*$-algebras stabilizing $\xi$. Since any morphism of quantum groups is in particular a $^*$-homomorphism, it follows that the quantum group stabilizer of $\xi$ must be a quotient of $A/\langle \mathcal{J}_\xi \rangle$.   By Lemma \ref{cor2}, $A/\langle \mathcal{J}_\xi \rangle$ is actually a quantum subgroup of $A$, so it is the quantum group stabilizer of $\xi$.
\end{proof}

\begin{remark} By Proposition \ref{cool}, we have that $\langle \mathcal{J}_\xi \rangle$ is just the principal ideal generated by $(\s_\xi -||\xi ||^2 1_A)$.
\end{remark}

Proposition \ref{cor3} makes computing the quantum group stabilizer much easier than Theorem  \ref{thm2}, simply because it is much easier to visualize the two-sided ideal generated by $\mathcal{J}_\xi$ rather than the Hopf ideal generated by $\mathcal{J}_\xi$.

We check an analog of a simple result from the representation theory of ordinary groups.  If $(U,\H )$ is a unitary representation of a group $G$, then $U^{\otimes n}$ is the diagonal action of $G$ on $\H^{\otimes n}$. If $H<G$ is the stabilizer of some $\xi \in \H$, then $H$ will also stabilize $\xi^{\otimes n}\in \H^{\otimes n}$.  Of course, if $G$ is not connected, then $H$ may not be all of the  stabilizer of $\xi^{\otimes n}$.

\begin{proposition}\label{prop2}
Let $(u,\H )$ be a unitary representation of a compact quantum group $A$, and let $\xi \in \H$. Let $\Pi :A\rightarrow B$ be the quantum stabilizer of $\xi$. Consider the unitary representation $(u^{\otimes n}, \H^{\otimes n})$ of $A$. Then $\Pi :A\rightarrow B$ stabilizes $\xi^{\otimes n}$.
\end{proposition}
\begin{proof}
Denote $J=\mathcal{J}_{\xi}$ and $J_n=\mathcal{J}_{\xi^{\otimes n}}$. We must show that $J_n$ is contained in the two-sided ideal generated by $J$ for all $n\in \N$. This is trivial when $n=1$. Inductively assume that $J_{n-1}$ is contained in the ideal generated by $J$.  A generic element of $J_n$ is of the form
$\langle u^{\otimes n}(\xi^{\otimes n}), (\bigotimes_{i=1}^n \eta_i)\otimes
       1_A \rangle-\langle \xi^{\otimes n},\bigotimes_{i=1}^n\eta_i \rangle
       1_A = \Pi_{i=1}^n\langle u(\xi ),\eta_i\otimes 1_A\rangle -
       \Pi_{i=1}^n\langle \xi , \eta_i\rangle 1_A.$
But notice 
\begin{eqnarray}
\Pi_{i=1}^n\langle u(\xi ),\eta_i\otimes 1_A\rangle & - & 
       \Pi_{i=1}^n\langle \xi , \eta_i\rangle 1_A  = \nonumber \\
	& = &  
(\langle u(\xi ), \eta_1\otimes 1_A\rangle -\langle \xi , \eta_1 \rangle )
       \Pi_{i=2}^n\langle u(\xi ),\eta_i\otimes 1_A\rangle + \nonumber \\
     &  + & \langle \xi , \eta_1 \rangle \left( \Pi_{i=2}^n\langle u(\xi ),
       \eta_i\otimes 1_A\rangle -\Pi_{i=2}^n\langle \xi , \eta_i\rangle 1_A
       \right) 
       \in JA + J_{n-1} \nonumber
\end{eqnarray}
which is in the two-sided ideal generated by $J$ by the induction hypothesis.
\end{proof}

Finally, we consider quantum coset spaces.  They give the simplist examples of ergodic actions of compact quantum groups.  First, we give the definition.

\begin{definition}
Let $A$ be a compact quantum group, and $\Pi :A\rightarrow B$ be a quantum subgroup of $A$.  Then the quantum coset space $A^B $ is the unital C$^*$-subalgebra of $A$ given by 
\[ A^B  =\{ a\in A \: | \: (\Pi \otimes id)\Delta (a)=1_B\otimes a\} . \]
\end{definition}

Intuitively, $A^B $ is the set of elements of $A$ invariant under the right action $(\Pi \otimes id)\Delta :A\rightarrow B\otimes A$ of $B$ on $A$.  Since the coproduct also gives a left action of $A$ on itself, there should be a remaining left action of $A$ on $A^B $.  The following proposition is well known, see \cite{Pinzari1} for example.  We omit the proof, since we will give a verbatim proof for the corresponding result for state stabilizers in Proposition \ref{lem5}.

\begin{proposition}
$A^B $ is a right coideal of $A$.  That is, $\Delta (A^B )\subseteq A^B \otimes A$.
\end{proposition}

Since $B$ is actually a quantum group itself, it has its own Haar state $\h_B$.  As in the case for ordinary groups, integrating over $B$ gives us an expectation from $A$ onto $A^B $.  The following proposition is stated without proof quite often.  See \cite{Pinzari1}, \cite{Wang2} for example.  We give a proof for completeness.

\begin{proposition}
Let $\Pi :A\rightarrow B$ be a quantum subgroup of $A$, and let $\h_B$ be the Haar state of $B$.  Then there is an expectation $E=(\h_B \circ \Pi \otimes id)\Delta :A\rightarrow A^B $.  
\end{proposition}
\begin{proof}
It is clear that the map $(\h_B \circ \Pi \otimes id)\Delta :A\rightarrow A$ is completely positive.  Suppose $a\in A^B $.  Then $(\Pi \otimes id)\Delta a=1_B\otimes a$, so $E(a)=(\h_B \otimes id)(1_B\otimes a) =a$.  Thus the rangle of $E$ contains $A^B $, and $E$ fixes $A^B $.  Thus it remains to show that $E(A)\subseteq A^B $.  Let $a\in A$ be arbitrary.  We compute:
\begin{eqnarray}
(\Pi \otimes id)\Delta (E(a)) 
	& = & 
	(\Pi \otimes id)\Delta [(\h_B\circ \Pi \otimes id)\Delta a] \nonumber \\
	& = & 
	(\h_B\circ \Pi \otimes \Pi \otimes id)(id\otimes \Delta)\Delta a \nonumber \\
	& = & 
	(\h_B\circ \Pi \otimes \Pi \otimes id)(\Delta \otimes id)\Delta a \nonumber \\
	& = & 
	([(\h_B\otimes id)\Delta_B \circ \Pi ]  \otimes id)\Delta a \nonumber \\
	& = & 
	1_B\otimes [(\h_B\circ \Pi \otimes id)\Delta a]=1_B\otimes E(a). \nonumber 
\end{eqnarray}
Thus $E(a)\in A^B $ for all $a\in A$.
\end{proof}

\end{subsection}

\begin{subsection}{State Stabilizers}

Consider a unitary representation $(U,\H )$  of an ordinary group $G$.  For $\xi \in \H$, the stabilizer of $\xi$ is defined to be the set $H_\xi =\{ h\in G \: | \: U_h(\xi )=\xi \}$.  Typically one views $H_\xi$ as being a subgroup of $G$, and then one is led to define stabilizers for quantum groups as being quantum subgroups, as we did in the previous section.  However, one could also forget the group structure of $G$ and think just in terms of its topology.  The the points of $G$ correspond to pure states of $C(G)$, and the points of $H_\xi$ correspond to the pure states  $\mu \in S(C(G))$ such that 
\[ U_\mu (\xi )=\int_G U_x(\xi ) \: \: d\mu (x) =\xi .\]
With this as motivation, we make the following definition.

\begin{definition}\label{defstabilizer}
Let $u:\H \rightarrow \H \otimes A$ be a unitary representation of a compact quantum group $A$.  Let $\xi \in \H$.  We denote by $H_\xi$ the set
\[ H_\xi =\{ \varphi \in S(A) \: |\: u_\varphi (\xi )=\xi \} . \]
If the vector $\xi$ is understood, we will denote $H_\xi$ simply by $H$.  

Because of Proposition \ref{prop3} below, we will often call $H_\xi$ the state stabilizer of $\xi$.  (This was the reason for our choice of the term ``state stabilizer'' to begin with).
\end{definition}

Clearly $H_\xi$ is a closed, convex subset of $S(A)$, so it can be recovered from its set of extreme points.  In particular, for an ordinary group it contains the same information as the usual stabilizer.

Theorem \ref{thm2} above gives us an abstract realization of the state stabilizer $\Pi :A\rightarrow B$ as a mapping of order unit spaces. In practice, the order ideal generated by a particular set is quite difficult to visualize, much less compute.  We seek a simpler expression for the state stabilizer. 

\begin{proposition}\label{prop3}
Let $u:\H \rightarrow \H \otimes A$ be a unitary representation of $A$, and let $\xi \in \H$.  Let $\Pi :A\rightarrow B$ be the state stabilizer of $\xi$. Then $ker  (\Pi )=\bigcap_{\phi \in H_\xi} ker (\phi)$. Furthermore, the pullback of the state space of $B$ by $\Pi$ is $H_\xi$.
\end{proposition}
\begin{proof}
Let $\Phi :A\rightarrow C$ be any surjective morphism of order unit spaces that stabilizes $\xi$.  Let $\mu \in S(C)$.  As a morphism of order unit spaces, $\Phi$ is positive and unital, so that $\mu \circ \Phi \in S(A)$.  We have $(id \otimes \mu \circ \Phi )(u (\xi ))=(id \otimes \mu)(id \otimes \Phi )(u (\xi ))= (id \otimes \mu) (\xi \otimes 1_C) = \xi $, so $\mu \circ \Phi \in H_\xi$.  Since we consider only Archimedian order unit spaces, $S(C)$ separates the points of $C$, so we see that $ker  (\Phi ) \supseteq \bigcap_{\phi \in H_\xi} ker (\phi) \equiv K$ and $\Phi$ factors through the natural map $\Pi : A\rightarrow A/K$.  On the other hand, since $K$ is an intersection of order ideals, it is an order ideal itself and  $\Pi: A\rightarrow A/K$ is a morphism of order unit spaces. So we have seen that any surjective morphism of order unit spaces that stabilizes $\xi$ factors through the natural map $\Pi :A\rightarrow A/K$.

To show that $\Pi :A\rightarrow A/K$ is the state stabilizer of $\xi$, we only have left to show that this map stabilizes $\xi$. As usual, the dual space of $A/K$ pulls back via $\Pi$ to the set of elements in $A'$ that annihilate $K$. Since $K=\bigcap_{\phi \in H_\xi} ker (\phi)$, the Banach extension theorem implies that the pullback of the dual space of $A/K$ is the closed linear span of the elements of $H_\xi$. Since the elements of $H_\xi$ are states, the state space of $A/K$ is the closed convex hull of $H_\xi$, which is $H_\xi$. But the elements of $H_\xi$ stabilize $\xi$, so reversing the argument in the last paragraph shows that $A/K$ stabilizes $\xi$, and thus is the state stabilizer. This proves the last claim also.

\end{proof}

We now define an operator system that will play the role of our quantum coadjoint orbit later on.  Recall our notation $_\varphi\Delta (a) =a*\varphi =(\varphi \otimes id )\Delta a.$

\begin{definition}\label{quantumhomo}
Let $A$ be a compact quantum group, and let $\Pi :A\rightarrow B$ be a positive unital map from $A$ to an order unit space $B$.  We define the quantum homogeneous space corresponding to $\Pi$ to be the set
\[ A^\Pi =\{ a\in A \: | \: (\Pi \otimes id)\Delta a =1_B \otimes a \} .\]
It is an operator system.

Typically, we will interested in the case when $\Pi :A\rightarrow B$ is the state stabilizer of some vector $\xi$ in a unitary representation of $A$.  In this case we will denote $A^\Pi$ as $A^{H_\xi}$, or as $A^H$ when $\xi$ is understood.  In this case $A^{H_\xi}$ is given by
\[ A^{H_\xi}  =\{ a\in A \: | \: _\varphi\Delta (a) = a , \: \: \forall \varphi \in H_\xi \} .\]
\end{definition}

Intuitively, $A^\Pi$ is the set of 
``$\Pi $-invariant'' elements of $A$, when $A$ is considered to be an right ordered  $A$-module via the coproduct $\Delta :A\rightarrow A$.   Of course, it wouldn't be reasonable to call $A^H$ the orbit of $\xi$ under $u$, since an ergodic action of a compact quantum group is not necessarily characterized by the stabilizer of some element. For example, consider an ordinary compact group acting ergodicly on a noncommutative C$^*$-algebra.

\begin{proposition}\label{lem5}
$A^\Pi$ is an operator system, and a right coideal of A.
\end{proposition}
\begin{proof}
The mapping $\Pi :A\rightarrow B$ is unital and positive, so $A^\Pi$ is a unital, norm-closed, and  $*$-closed subspace of $A$. Thus $A^\Pi$ is an operator system.

For the second claim, we must show that if $a\in A^\Pi$, then $\Delta a\in A^\Pi \otimes A$.
We have $(\Pi \otimes id)\Delta a =1_B\otimes a$.  We compute
$  (\Pi \otimes id\otimes id)(\Delta \otimes id)\Delta a=
(\Pi \otimes id\otimes id)(id \otimes \Delta)\Delta a=
(id \otimes \Delta)(\Pi \otimes id)\Delta a=
(1_B\otimes \Delta a)$.

\end{proof}
\begin{remark}
Since $\Pi$ is not necessarily an algebra homomorphism, $A^\Pi$ need not be a subalgebra of $A$.  
\end{remark}

For future reference, we note a simple criterion for whether or not $\varphi \in S(A)$ is in $H_\xi$.   The following lemma is the analog of Proposition \ref{cool} written for states.

\begin{lemma}\label{lem4}
Let $A$ be a compact quantum group and let $u, \H ,$ and $\xi$ be as above. Let $\varphi \in S(A)$. Then $\varphi$ is in $H_\xi$ iff $\langle u_\varphi (\xi ) ,\xi  \rangle =\| \xi \| ^2$.  In other words, $\varphi$ is in $H_\xi$ iff $\varphi (\s_\xi )=||\xi ||^2$ iff $\varphi (\s_\xi^*)=||\xi ||^2$.
\end{lemma}
\begin{proof}
Since $u$ is unitary and $\varphi$ is a state, $u$ and $id\otimes \varphi$ are contractive and $\| u_\varphi (\xi ) \| \leq \| \xi \|$.  By the Cauchy-Schwarz inequality we have $u_\varphi (\xi )= \xi$ iff $\langle  u_\varphi (\xi ) ,\xi \rangle =\| \xi \| ^2$.
\end{proof}

\begin{example}
Let $(U,\H )$ be a unitary representation of a compact group $G$, and let $A=C(G)$.  Let $\xi \in \H$.  We have that the state space of $A$ consists of Borel probability measures on $G$. If $\mu \in S(A)$ we have $U_\mu (\xi )=\int_G U_x(\xi )d\mu (x)$.  Since $||U_x(\xi )||=||\xi ||$ for all $x\in G$ and $\mu$ is a probability measure, Minkowski's inequality gives $U_\mu (\xi )=\xi$ iff $U_x(\xi )=\xi$ on the support of $\mu$.  Thus $H_\xi$ is the set of probability measures supported on the usual stabilizer $H<G$ of $\xi$. It follows that the state stabilizer of $\xi$ is just the restriction map $C(G)\rightarrow C(H)$.
\end{example}

We see that for $A=C(G)$, both the state stabilizer and the quantum stabilizer of $\xi \in \H $ are equal to $C(H)$, where $H<G$ is the usual stabilizer of $\xi$. However, for cocommutative quantum groups, the results are quite different.

\begin{example}
Let $A=C^*(\Gamma )$ for a discrete group $\Gamma$, and let $(u,\C^n )$ be an $n$-dimensional unitary (co)representation of $A$.  As before, we can replace $u$ with an equivalent representation to arrange that $u\in M_n(A)$ is a diagonal matrix with entries $\gamma_1, ...,\gamma_n$.  Let $\xi \in \C^n$ with each coordinate nonzero. We have $u(\xi )_i=\xi_i \gamma_i$.  Let $\varphi \in S(A)$.  We have $u_\varphi (\xi )_i=\varphi (\gamma_i)\xi_i$.  Thus
$\\ H_\xi =\{ \varphi \in S(A)\: | \: \varphi (\gamma_i)=1, \forall i=1,...,n\}$.

The state space of $A$ is given by matrix elements of unitary representations of the group $\Gamma$ (not to be confused with unitary (co)representations of $A$!)  That is, if $\nu$ is a unitary representation of $\Gamma$ and $x\in \H_\nu$ is a unit vector, we have a state
$\varphi_{\nu , x}(\gamma )=\langle \nu_\gamma(x),x\rangle$, and every state on $C^*(\Gamma )$ arises in this way. By the Cauchy-Schwarz inequality we have $\varphi_{\nu ,x}(\gamma )=1$ if and only if $\nu_\gamma (x)=x$.
Since $\nu$ is a representation, we have that if $\nu_{\gamma_i}(x)=x$ for all $i=1,...,n$, then $\nu_\gamma (x)=x$ for any $\gamma$ in the subgroup generated by the $\gamma_i$.  So we see that the state $\varphi_{\nu , x}$ is in $ H_\xi$ if and only if $x\in \H_\nu$ is fixed by the subgroup generated by $\{ \gamma_i \}_{i=1,...,n}$.

In contrast, we had before that the quantum group stabilizer of $\xi$ was given by $C^*(\Gamma /N)$, where $N$ is the normal subgroup generated by the $\gamma_i$.  A state on $C^*(\Gamma /N)$ corresponds to a representation of $\Gamma$ that is trivial on $N$. Intuitively, the quantum group stabilizer of $\xi$ is made out of representations of $\Gamma$ that are trival on $\gamma_1, ..., \gamma_n$, whereas the state stabilizer of $\xi$ is made out of representations of $\Gamma$ which have a vector fixed by $\gamma_1, ...,\gamma_n$.

\end{example}
\begin{remark}
In the example above, let $\Omega \subseteq \Gamma$ be the subgroup generated by $\gamma_1,...,\gamma_n$.  For simplicity, consider the case where $\Omega$ is finite.  Let $p\in C^*(\Gamma )$ be the projection
\[ p=\frac{1}{|\Omega |} \Sigma_{\gamma \in \Omega} \: \gamma .\]
Let $(\nu ,\H )$ be a unitary representation of the group $\Gamma$, and let $x \in \H$.  Then any nonzero vector $y\in \H$ of the form $y=\nu (p) x$ will be fixed by the subgroup $\Omega$, and so will give a state in the state stabilizer $H_\xi$.    It's straightforward to check that every state in $H_\xi$ arises in this way.  Therefore, viewed as a mapping of order unit spaces we have that the state stabilizer of $\xi$ is given by the natural map $C^*(\Gamma )\rightarrow pC^*(\Gamma )p$.  

From Theorem 2.4 of \cite{Franz}, we see that the state stabilizer of $\xi$ is a compact quantum subhypergroup of $C^*(\Gamma )$.  We will not discuss hypergroups much in this paper, but in the appendix we will give an example which suggests that state stabilizers are not even quantum subhypergroups in general.

\end{remark}

The previous example actually indicates the general difference between the state stabilizer and the quantum stabilizer, as we now show.

\begin{example}

Let $(u,\H )$ be a unitary representation of a compact quantum group $A$, and let $\xi \in \H$ be a unit vector.  By Lemma \ref{lem4} we have 
$H_\xi= \{ \varphi \in S(A)\: | \: \varphi (\s_\xi )=1\}$.  Let $\varphi \in S(A)$. By the Gelfand-Naimark Theorem  there is a $*$-homomorphism $(\pi ,\H_\pi )$ of the underlying C$^*$-algebra of $A$ and a unit vector $\nu \in \H_\pi$ such that $\varphi (a)=\langle \pi (a) \nu ,\nu \rangle$.  
By the Cauchy-Schwarz inequality we have $\varphi \in H_\xi$ if and only if 
$\pi (\s_\xi )\nu =\nu$.  So states in the state stabilizer generally correspond the unit vectors fixed by $\s_\xi$ in the $*$-representations of of the underlying C$^*$-algebra of $A$.

On the other hand, by Proposition \ref{cor3} the quantum group stabilizer $\Pi :A\rightarrow B$ is actually the universal map of C$^*$-algebras that stabilizes $\xi$.  By Proposition \ref{cool} a map of C$^*$-algebras  stabilizes $\xi$ if and only if it sends $\s_\xi$ to the identity.  Let $\pi :A\rightarrow B(\H_\pi )$ be a $*$-representation of the underlying C$^*$-algebra of $A$.  If $\pi$ is the pullback of some $*$-representation of $B$, then clearly $\pi (\s_\xi )=1$.  Conversely, if $\pi (\s_\xi )=1$, by the C$^*$-universality of $B$, we must have that $\pi$ factors through $\Pi$, so that $\pi$ is the pullback of a $*$-representation of $B$.  Again, by Gelfand-Naimark, states on $B$ correspond to pairs $(\pi , \nu )$ where $\pi$ is a $*$-representation of $B$ and $\nu \in \H_\pi$. Thus the states of $A$ that come from the quantum group stabilizer of $\xi$ generally correspond to unit vectors in the $*$-representations of the underlying C$^*$-algebra of $A$ that send $\s_\xi$ to the identity operator.  We summarize as follows.
\end{example}

\begin{summary}\label{summary}
To compute either the quantum group stabilizer or the state stabilizer of $\xi$, first consider the $*$-representation theory of the underlying C$^*$-algebra of $A$.  The quantum group stabilizer is computed by determining which  $*$-representations of $A$  send $\s_\xi$ to the identity operator.  The state stabilizer is computed by determining which $*$-representations of $A$ have a fixed vector for $\s_\xi$.
\end{summary}

The reasoning of the last example allows us to prove the analog of Proposition \ref{prop2} for state stabilizers.

\begin{proposition}\label{nthpower}

Let $(u,\H )$ be a unitary representation of a compact quantum group $A$, and let $\xi \in \H$.  Consider the unitary representation $(u^{\otimes n} ,\H^{\otimes n})$ of $A$, and the vector $\xi^{\otimes n}$.  Then the state stabilizer of $\xi$ also stabilizes $\xi^{\otimes n}$.

\end{proposition}

\begin{proof}
For simplicity, take $\xi$ to be a unit vector.  Consider the elements $\s_\xi$ and $\s_{\xi^{\otimes n}}$ of $A$.  Because of Lemma \ref{lem4}, it suffices to show that if $\varphi \in S(A)$ and $\varphi (\s_\xi )=1$ then $\varphi (\s_{\xi^{\otimes n}})=1$ also.  From the reasoning in the example, we have 
$\varphi (a) = \langle \pi (a) \nu ,\nu \rangle $ for some $*$-representation of $A$ and some unit vector $\nu$ that is fixed by $\s_\xi$.  Notice that $
\s_{\xi^{\otimes n}} = \langle u^{\otimes n}(\xi ^{\otimes n}), \xi^{\otimes n}\otimes 1\rangle = (\langle u(\xi ), \xi \otimes 1\rangle )^n =\s_\xi^n$.  Thus if $\pi$ is a $*$-representation and $\pi (\s_\xi )(\nu ) =\nu$ we have 
$\pi (\s_{\xi^{\otimes n}})(\nu )=\pi (\s_\xi^n)(\nu )=(\pi (\s_\xi ))^n(\nu )=\nu$ also.  Thus if $\varphi \in S(A)$ and $\varphi (\s_\xi )=1$, then $\varphi (\s_{\xi^{\otimes n}})=1$ also.

\end{proof}

We now turn to the issue of whether $H_\xi$ is closed under the coinverse operation of $\mathcal{A}'$.  If $\varphi \in \mathcal{A}'$ then the coinverse of $\varphi$ is defined by  $(\kappa \varphi ) (a) = \overline{\varphi (\kappa (a)^*)}$.  There are a couple of difficulties here.  The first is that it's not immediately clear that states on $\A$ are norm-bounded. Thus, it's not clear that our Lemma \ref{lem4} can be used.
However, we have assumed that all of our quantum groups are full, and it is known for $A$ full  that states of $\A$ are norm-bounded \cite{Bedos1}.  But this is not the only problem.  Generally we have $*\circ \kappa  = \kappa ^{-1} \circ *$, so $\kappa $ is not positive unless $\kappa ^2 =id$.  So $\varphi \in S (\mathcal{A})$ implies $\kappa \varphi \in S(\mathcal{A})$ only if  $\kappa ^2 =id$.

\begin{proposition}\label{prop4}
Let $u, \H ,\xi$ be as above and let $H_\xi$ be the state stabilizer of $\xi$.  Suppose that $A$ is full and Kac-type.  Then $S(A)$ and $H_\xi$ are closed under the coinverse on $\mathcal{A}'$.
\end{proposition}
\begin{proof}
Generally we have any state on $A$ restricts to a state on $\mathcal{A}$, and the restriction map is injective since $\mathcal{A}$ is dense in $A$.  On the other hand, by Theorem 3.3 of \cite{Bedos1} for $A$ full states on $\mathcal{A}$ are norm bounded, so that $S(A)=S(\mathcal{A})$ for $A$ full.  Hence we can define $\kappa : S(A)\rightarrow \mathcal{A}'$.  The coinverse map of $\A'$ is antimultiplicative, so that $\kappa (\varphi \varphi^*)=\kappa (\varphi^*)\kappa (\varphi )$.  For $\kappa =\kappa ^{-1}$ we have $\kappa $ commutes with $*$, so $\kappa $ is positive.   Since $\kappa $ is  unital, we have  $\kappa : S(A)\rightarrow S(A)$.

By Lemma \ref{lem4}, $\varphi \in H_\xi$ is equivalent to $\langle \xi , u_\varphi (\xi ) \rangle =\| \xi \| ^2$. Because $\varphi$ is a state, we have $\kappa \varphi =\varphi \circ \kappa$.  Finally, because $u$ is unitary we have $\langle \xi , u_\varphi (\xi ) \rangle =\langle u_{\varphi \circ \kappa} (\xi ),  \xi  \rangle =\langle u_{\kappa \varphi} (\xi ),  \xi  \rangle$.   So $\varphi \in H_\xi$ implies $\kappa \varphi \in H_\xi$.
\end{proof}

Above we defined $A^{H_\xi}=\{ a\in A\: | \: _\varphi\Delta (a )=a ,\: \forall \varphi \in H_\xi \}$.  This definition suggests a natural question.  Namely, if $\psi \in S(A)$ and one has $_\psi\Delta (a)=a$ for all $a\in A^{H_\xi}$, does one necessarily have $\psi \in H_\xi$?  It is not obvious that this should be the case, but for Kac-type quantum groups, we can use Proposition \ref{prop4} to obtain an affirmative answer.  First we need a simple computational lemma.

\begin{lemma}\label{switch}
Let $\varphi \in A'$ and let $a\in \A$. Suppose $A$ is of Kac-type.  We have $\Delta_\varphi (a)=a$ if and only if $_{\varphi \circ \kappa}\Delta (\kappa a)=\kappa a$.
\end{lemma}
\begin{proof}
This statement follows immediately from the fact that $\Delta \circ \kappa =\vartheta (\kappa \otimes \kappa )\Delta$, where $\vartheta :\A \otimes \A \rightarrow \A \otimes \A $ is the flip map.
\end{proof}

\begin{proposition}\label{dual}
Let $u:\H \rightarrow \H \otimes A$ be a unitary representation of a compact quantum group $A$, and let $\xi \in \H$.  Let $H_\xi$ and $A^{H_\xi}$ be as above.  Let $\psi \in S(A)$.  If $A$ is of Kac-type, then $\psi \in H_\xi$ if and only if $_\psi\Delta (a)=a$ for all $a\in A^{H_\xi}$.
\end{proposition}
\begin{proof}
To avoid technical difficulties, we prove only the case when the isotypic components of $u$ are finite dimensional.  For the general case, one must replace the action $\Delta :\A \rightarrow \A \otimes \A$ with the right regular representation $``\Delta$''$:L^2(A)\rightarrow L^2(A)\otimes A$.

Let $\H =\bigoplus_{\mu \in \hat{A}} \H_\mu$ be the decomposition of $\H$ into isotypic components.  It's easy to see that for $\varphi \in S(A)$, we have $u_\varphi (\xi )=\xi$ if and only if $u_\varphi (\xi^\mu )=\xi^\mu$ for each $\mu \in \hat{A}$, where $\xi^\mu$ is the projection of $\xi $ into $\H_\mu$. That is, we have $H_\xi = \bigcap_{\mu \in \hat{A}} H_{\xi^\mu}$ and thus $A^{H_\xi}=\bigcap_{\mu \in \hat{A}} A^{H_{\xi^\mu}}$.
 So it suffices to consider the case when $\H =\H_\mu$ for some $\mu \in \hat{A}$.  Since we are assuming that the isotypic components of $u$ are finite dimensional, we have in this case that $u$ is finite dimensional.

By the definition of $A^H$, we have $\psi \in H_\xi$ implies $_\psi\Delta a=a$ for all $a\in A^H$.  For the other direction first note that since $\H$ is finite dimensional, we can embed $\H$ as an $\A$-module into $\C^n\otimes \A$ for some $n\in \N$.  Indeed, if $(\mu , \K_\mu)$ is an irreducible representation of $A$, then one can choose a basis $\{ e_1,...,e_m\}$ of $\K_\mu$ such that $u(e_j)=\Sigma_l e_l\otimes u_{lj}^\mu$.  Then the map defined by $e_j\mapsto u_{1j}^\mu$ is an intertwining map from $\K_\mu$ to $\A$.

Thus we can assume that $(u,\H )$ is an $\A$-submodule of $\C^n \otimes \A$, and we can view $\xi \in \C^n\otimes \A$.  Write $\xi$ as a tuple $(a_1,...a_n)$ with $a_j \in \A$ for each $j$. Let $\psi \in S(A)$.  Then one has $(id\otimes \Delta_\psi ) \xi =\xi $ if and only if $\Delta_\psi ( a_j)=a_j$ for all $j$.  So we have shown that there exists a subset $\{ a_1,...,a_n\} \subseteq \A$ such that $\psi \in H_\xi$ if and only if $\Delta_\psi( a_j)=a_j$ for all $j$.  By the lemma, this is equivalent to $_{\psi \circ \kappa}\Delta (\kappa a_j)=\kappa a_j$ for all $j$.  By Proposition \ref{prop4} we have $H_\xi =H_\xi \circ \kappa$, so that these $\kappa a_j$ are in $A^{H_\xi}$.  Furthermore, by replacing $\psi$ with $\psi \circ \kappa$ we have just shown that if $\psi  \in S(A)$ fixes this finite set $\{ \kappa a_1,...,\kappa a_n \} \subseteq A^{H_\xi}$, then $\psi \circ \kappa \in H_\xi$. Again, we use that $H_\xi =H_\xi \circ \kappa$ to get that $\psi \in H_\xi$. This  proves the proposition.

\end{proof}

Finally, if $\Pi :A\rightarrow B$ is a quantum subgroup of $A$, then the Haar measure of $B$ gives a idempotent map from $A$ to $A^\Pi $.  Since the state stabilizer is generally not a quantum subgroup, there's no reason to expect that $H_\xi$ should have a Haar state; that is, an element $\mathbi{h} \in H_\xi $ such that $\mathbi{h} *\varphi =\varphi *\mathbi{h} =\mathbi{h}$ for all $\varphi \in H_\xi$.  Thus we have the question:

\begin{question}
When does the state stabilizer have a Haar state?
\end{question}

In the appendix we use the results of this paper to that for $A$ a theta deformation of an ordinary Lie group, the answer is often yes.  We also give evidence that for a particular representation of $A_u(n)$ the answer is probably not.

\end{subsection}

\end{section}

\begin{section}{Berezin Quantization for  Quantum Groups}

\begin{subsection}{The Berezin Symbol and its Adjoint}

For the remainder of the paper,  $A$ will be a coamenable compact quantum group of Kac-type,  and $u\in B(\H ) \otimes A $ will be an irreducible unitary representation of $A$.   We have an action of $A$ on $B (\H )$ by $\alpha (T)=u(T\otimes 1)u^*$.  It's easy to check that as representations of $A$ we have $\alpha \cong u\otimes  \overline{u}$.  In particular, $\alpha$ is unitary when $B(\H )$ is viewed as a Hilbert space with the Hilbert-Schmidt inner product.  Finally, to define the Berezin symbol we choose a unit vector $\xi \in \H$, and we let $P$ in $B(\H )$ be the corresponding rank $1$ projection.

 We can now define the Berezin symbol $\s :B(\H ) \rightarrow A$ by:
\[  \s_T = (tr \otimes id)[(P\otimes 1) \alpha (T)] \]
where $tr$ is the usual, unnormalized trace.  It is clear from the definition that $\s$ is unital and completely  positive.  We wish to show that the range of $\s$ is contained in $A^H$, where $H$ is the state stabilizer of $P$; that is, $H=\{ \phi \in S(A) \: | \: \alpha_\phi (P)=P \}$.  Here, we retain the shorthand notation $\alpha_\phi =(id\otimes \phi )\alpha $ from the last section. Notice that already we must use that $A$ is of Kac-type.

\begin{proposition}\label{prop5}
We have $\s :B(\H )\rightarrow A^H$.  The kernel of $\s$ is the orthogonal compliment of $\{ \alpha_\varphi(P) \: |\: \varphi \in A'\}$.  In particular, the restriction of $\s$ to the  the set $\{ \alpha_\varphi(P) \: |\: \varphi \in A'\}$ is injective.
\end{proposition}
\begin{proof}
Let $\phi \in S(A)$.  We compute:
\begin{eqnarray}
 _\phi\Delta (\s_T) & = & (\phi \otimes id)\Delta [(tr \otimes id)[(P\otimes 1) \alpha (T)]]  \nonumber \\
 & = & (tr\otimes \phi \otimes id)[(P\otimes 1\otimes 1)(id\otimes \Delta )(\alpha (T))] \nonumber \\
 & = &  (tr\otimes \phi \otimes id)[(P\otimes 1\otimes 1)(\alpha \otimes id )(\alpha (T))] \nonumber \\
 & = & (tr\otimes (\phi \circ \kappa ) \otimes id)[(\alpha (P)\otimes 1)(\alpha (T))_{13}] \nonumber \\
 & = & (tr\otimes id)((\alpha_{\kappa \phi}(P)\otimes 1)(\alpha (T))). \nonumber
\end{eqnarray}
To obtain the fourth line above we have used that $\alpha$ is unitary for the Hilbert-Schmidt inner product on $B(\H )$.  Now suppose that $\phi \in H$.  Then by Proposition \ref{prop4} we have $\kappa \phi \in H$, so the above calculation shows
$ _\phi\Delta (\s_T) =\s_T$.  That is, $\s_T \in A^H$.

For the second claim, let $T\in ker(\s )$.  Let $\mu \in A'$ be arbitrary.  The above calculation shows that $tr (\alpha_{\kappa \phi}(P) \alpha_\mu (T))=0$ for all $\phi \in S(A)$.  Since $S(A)$ spans $A'$, we see that $T\in ker(\s )$ if and only if the subspaces $\{ \alpha_\varphi(P) \: |\: \varphi \in A'\}$ and $\{ \alpha_\mu(T) \: |\: \mu \in A'\}$ of $B(\H )$ are orthogonal. Since $\alpha$ is unitary, we have $\langle \alpha_\varphi(P), \alpha_\mu(T)\rangle = \langle \alpha_{ (\mu \circ \kappa ) *\varphi}(P), T\rangle$, so that $T\in \ker (\s)$ iff $T$ is orthogonal to $\{ \alpha_\varphi(P) \: |\: \varphi \in A'\}$.
In particular, we see that  the restriction of $\s$ to $\{ \alpha_\varphi(P) \: |\: \varphi \in A'\}$ is injective.
\end{proof}

In light of this proposition we denote $\B =\{ \alpha_\varphi(P) \: |\: \varphi \in A'\} $.  Like $A^H$, $\B$ is not necessarily  a subalgebra of $B(\H )$, but is an operator system.  Eventually, we will show that under suitable conditions, if we replace $u$ by $u^{\otimes n}$ and look at the sequence of spaces $\B^n$, they will converge to $A^H$ in quantum Gromov-Hausdorff distance.

Consider the case when $G$ is an ordinary group.  If $\mu \in C(G)'$ is a Borel measure on $G$, then we have $\alpha_\mu (P)=\int_G \alpha_x(P)d\mu (x)$. So we see that in this case $\B$ is the cyclic subspace of $B(\H )$ generated by $P$.  
We show that this is the case in general.    The cyclic subspace generated by $P$ is clearly invariant under the operators $\alpha_\varphi$, so it is certainly no smaller than $\B$.  Thus we have only to  check that $\B$ is an $\alpha$-invariant subspace of $B(\H )$.

\begin{lemma}
We have $\alpha :\B \rightarrow \B \otimes A$.
\end{lemma}
\begin{proof}
By Proposition \ref{prop5}, we have $\ker (\s )=\B ^{\perp}$.  Let $T$ be orthogonal to $\B$, and let $b\in \A$.  It suffices to show that $(T\otimes b )$ is orthogonal to the image of $\alpha (\B )$ in $B(\H )\otimes L^2(A)$.
By Corollary \ref{cor1}, we have $\alpha^\dagger=(id\otimes \h \circ \mathsf{m} )([(id\otimes \kappa )\alpha ]\otimes id)$.  We compute: 
\begin{eqnarray}
\langle T\otimes b, \alpha (\alpha_\varphi(P))\rangle & = &
( \alpha^\dagger(T\otimes b) ,\alpha_\varphi (P) ) \nonumber \\
& = &
\langle (id\otimes \h )([(id\otimes \kappa )\alpha (T)](1\otimes b)), \alpha_\varphi(P) \rangle \nonumber \\ & = &
\langle (id\otimes \h )[(1\otimes \kappa^{-1}b)\alpha (T)], \alpha_\varphi(P) \rangle \nonumber \\ & = &
\langle \alpha_\psi (T), \alpha_\varphi(P) \rangle =0, \nonumber
\end{eqnarray}
where we have set  $\psi (a) =\h ((\kappa^{-1}b) a)$.

\end{proof}

We must compute the adjoint of the Berezin symbol.  We let $\h \in S(A)$ be the Haar state of $A$.  We obtain an inner product on $A^H$ by $\langle a, b \rangle =\h (b^*a)$. This inner product may be degenerate when $A$ is not coamenable.  Eventually we will restrict to this case, but there is no need to just yet. We use on $\B \subseteq B(\H )$ the inner product coming from the normalized trace: $\langle S,T\rangle =d_{\H}^{-1}tr(T^*S)$, where $d_{\H}$ denotes the dimension of $\H$.  We can now compute the adjoint of $\s$.

\begin{proposition}\label{prop6}
Consider the injective map $\s :\B \rightarrow A^H$.  It's adjoint $\bs :A^H\rightarrow \B$ is given by:
\begin{eqnarray*} \bs_a & = & d_\H \alpha^\dagger (P\otimes a) \\
& = & d_{\H}(id\otimes \h )([(id\otimes \kappa )\alpha (P)](1 \otimes a)). 
\end{eqnarray*}
Furthermore,  $\bs$ is unital, positive, and surjective.
\end{proposition}
\begin{proof}
Let $a\in A^H$ and $T\in \B$.  We compute
\begin{eqnarray}
\langle \bs_a, T\rangle_\B & = & \langle a, \s_T \rangle_{L^2(A)} = \h (\s_T^*a)          \nonumber \\ & = & 
       (tr\otimes \h )\{ (P\otimes 1)\alpha (T)^*(1\otimes a)\} \nonumber \\
& = & d_{\H}  \langle P\otimes a, \alpha(T) \rangle_{\B \otimes L^2(A)}     \nonumber \\
&  = &      \langle d_{\H}\alpha^\dagger (P\otimes a), T\rangle_\B . \nonumber
\end{eqnarray}
This proves the first formula for $\bs$.  The second formula follows from Corollary \ref{cor1}.  Since $\s :\B \rightarrow A^H$ is injective, $\bs : A^H\rightarrow \B$ is surjective.  Since $u\in B(\H )\otimes \A$ was assumed irreducible, by Theorem 5.4 of \cite{Woro1} we have the second orthogonality relation for $\alpha$; that is, for any $T\in B(\H )$ we have $(id\otimes \h )\alpha (T)= Tr(FT)/Tr(F)\; 1_{B(\H )}$, for some positive operator $F\in B(\H )$. Since $A$ is of Kac-type, Theorem 5.6 of \cite{Woro1} implies that $F=I_{B(\H )}$.  In particular, we have $\bs_{1_A}= d_{\H}(id\otimes \h )\alpha (P)= tr(P)\; 1_{B(\H )}=1_{B(\H )} $, so that $\bs$ is unital.

To prove positivity, note that  $\kappa^2=id$ implies that $\kappa$ is a $*$-homomorphism from $A$ to $A^{opp}$, so is completely positive.  Thus $(id\otimes \kappa )\alpha (P)$ is positive, viewed as an element of $B(\H )\otimes A^{opp}$.  (Note, it may not be a positive element when viewed in $B(\H )\otimes A$). By Theorem 5.6 of \cite{Woro1}, the condition $\kappa^2=id$ is equivalent to assuming that the Haar state is a trace.  Then $\h$ is also a tracial state on $A^{opp}$.  Let $\cdot_{opp}$ denote the product on $A^{opp}$.
Finally, since $\h$ is a state it is completely positive, and we have
\begin{eqnarray}
\bs_{aa^*} & = & d_{\H}(id\otimes \h )([(id\otimes \kappa )\alpha (P)](1
          \otimes aa^*)) \nonumber \\
       & = &  d_{\H}(id\otimes \h )((1 \otimes a^*\cdot_{opp}a)
        \cdot_{opp} [(id\otimes \kappa )\alpha (P)]) \nonumber \\
       & = &  d_{\H}(id\otimes \h )((1 \otimes a)
        \cdot_{opp} [(id\otimes \kappa )\alpha (P)]
        \cdot_{opp} (1\otimes a^*)) \geq 0 \nonumber
\end{eqnarray}
\end{proof}

\begin{remark}
The argument above is easily modified to give that $\bs$ is a completely positive map when viewed from $A^{opp}$ to $\B$.  This does not imply that $\bs$ is a completely positive map on $A$, since the complete order isomorphism between $A$ and $A^{opp}$ is $\kappa$, rather than the identity.   If $\bs :A\rightarrow \B$ is completely positive also, one would expect that $id :A\rightarrow A^{opp}$ would be complete order isomorphism.  A unital $2$-order isomorphism between unital C$^*$-algebras is a $*$-isomorphism \cite{Kerr}. This suggests that $\bs$ is completely positive only when $A$ is commutative.
\end{remark}

We now have a positive, unital, injective map $\s :\B \rightarrow A^H$ and its positive, unital, surjective adjoint $\bs :A^H\rightarrow \B$.  We also know that we have actions $\Delta :A^H\rightarrow A^H \otimes A$ and $\alpha :\B \rightarrow \B \otimes A$.  This brings us to the final fundamental property of the Berezin Symbol and the Berezin Adjoint; they are equivariant.

\begin{proposition}\label{invariant}
Consider the Berezin symbol and its adjoint  $\s :\B \rightarrow A^H$ and $\bs :A^H\rightarrow \B$.  They are equivariant for actions $\alpha$ and $\Delta$ of $A$ on $B(\H )$ and $A$.  Explicitly, we have
\begin{eqnarray}
(\s \otimes id )\circ \alpha = \Delta \circ \s & and \nonumber \\
(\bs \otimes id)\circ \Delta  =   \alpha \circ \bs &     \nonumber
\end{eqnarray}

\end{proposition}

\begin{proof}
We compute:
\begin{eqnarray} (\s \otimes id)\alpha (T) &  = & (tr\otimes id\otimes id )[(P\otimes 1\otimes 1)(\alpha \otimes id)\alpha (T)] \nonumber \\
&  = &
(tr\otimes id\otimes id )[(P\otimes 1\otimes 1)(id \otimes \Delta )\alpha (T)] =
\Delta (\s_T). \nonumber
\end{eqnarray}
Mimicking this calculation for $\bs$ directly is difficult, but if we work on the image of $\A$ in $L^2(A)$,  we can simply take the adjoint of the first result.  Explicitly, we have
\[ \alpha^\dagger\circ (\bs \otimes id) = \bs \circ \Delta^\dagger .\]
By Corollary \ref{cor1}, we have $\alpha^\dagger=(id\otimes \h \circ \mathsf{m} )([(id\otimes \kappa )\alpha ]\otimes id)$, and similarly for $\Delta^\dagger$. Let $a, b\in \A$. We have 
\begin{eqnarray*}
\alpha^\dagger (\bs \otimes id )(a\otimes b) & = &
(id \otimes \h )([(id\otimes \kappa )\alpha (\bs_a)] (1 \otimes b)) \\
& = &
(id \otimes \h )[(1\otimes \kappa^{-1}b)\alpha (\bs_a)].
\end{eqnarray*}
 Similarly, 
\begin{eqnarray*}
(\bs \circ \Delta^\dagger)(a\otimes b) & = &
\bs [(id\otimes \h )([(id\otimes \kappa )\Delta a](1\otimes b))] \\ & = &
\bs [(id\otimes \h )(1\otimes \kappa^{-1}b) \Delta a]  \\ & = &
(id\otimes \h )[(1\otimes \kappa^{-1}b )(\bs \otimes id )\Delta a]. 
\end{eqnarray*}
  Since $\h$ is nondegenerate on $\A$, this forces $(\bs \otimes id )\Delta a = \alpha (\bs_a)$ for any $a\in \A$.  Since both sides are continuous and $\A$ is dense in $A$, we have that $(\bs \otimes id )\Delta a = \alpha (\bs_a)$ for any $a\in A$.
\end{proof}

\begin{remark}
If one compares the Berezin symbol and Berezin adjoint defined in this section to the ones for ordinary Lie groups in Section 2.6, they will notice that the two differ by a coinverse.  The classical Berezin symbol and adjoint intertwine the action of $G$ on $B(\H )$ by conjugation with the action of $G$ on $C(G/H)$ by left translation. 

 For a quantum group $A$, the natural left action of $A$ on a homogeneous space is the action of $A$ on $A^H$ by right translations, because the left action 
of $A$ on $^H\!A$ by left translations requires one to use the coinverse, which generally does not extend to all of $A$.  For $A$ of Kac-type, the coinverse does extend to all of $A$, but there really is no reason to insist on using left translations, as it just complicates formulas and obscures the overall picture.

This switch between using left and right translations on the corresponding homogeneous spaces is the reason that our Berezin symbols differ from the classical ones by a coinverse.
\end{remark}

\end{subsection}

\begin{subsection}{The Berezin Transform and Berezin Quantization}

First, we compute the composition $(\s \circ \bs)(a)$.  The brute force computation of this composition is very messy.  We will give a much simplier version using the result of Lemma \ref{lem5.2}.  We find that there is an explicit linear functional $\hp \in S(A) $ such that $(\s \circ \bs)(a)= \: _\hp \!\Delta (a)= (\hp \otimes id )\Delta a$.  We make the definition first.

\begin{definition}\label{def5}
Let $\hp \in A'$ be the linear functional dual to the element $d_{\H}\s_P$.  That is $\hp (a) = d_{\H}\h (\s_P a)$.
\end{definition}

\begin{lemma}\label{lem6}
$\hp $ is a state.
\end{lemma}
\begin{proof}
By the second orthogonality relation, $\hp (1)=d_{\H} (tr\otimes \h )((P\otimes 1)\alpha (P))= tr(P \; tr(P)1_{B(\H )})=1$.  Also, since $\s$ is positive, $d_{\H} \s_P=bb^*$, some $b\in A$.  (Actually, an simple computation  shows $b=\sqrt{d_{\H}}\langle u(\xi ), \xi \otimes 1\rangle$ works.)
Thus $\hp (a^*a)= \h (bb^*a^*a)=\h  (b^*a^*ab)\geq 0$, so $\hp $ is a state.
\end{proof}

\begin{proposition}\label{prop7}
The composition $\s \circ \bs :A^H\rightarrow A^H$ is given by
$(\s \circ \bs )(a) =\: _\hp\!\Delta a$.
\end{proposition}
\begin{proof}
Let $a\in A^H$.  We compute
\begin{eqnarray*}
(\s \circ \bs )(a) & = & d_\H \s (\alpha^\dagger (P\otimes a)) \\
	& = & d_\H (tr\otimes id)((P\otimes 1)[\alpha \circ \alpha^\dagger (P\otimes a)]) \\
	& = & d_\H (tr\otimes id)((P\otimes 1)(\alpha^\dagger \otimes id)(P\otimes \Delta a)),   
\end{eqnarray*}
where we have used Lemma \ref{lem5.2} in the last line.

We would like to use that $\alpha^\dagger :\B \otimes L^2(A)\rightarrow \B$ really is the adjoint of $\alpha$, so we take an arbitrary element $b\in A$ and we apply $\h (b^* \cdot)$ to both sides.  We obtain
\begin{eqnarray*}
\langle (\s \circ \bs )(a) ,b\rangle_{L^2(A)} & =  & d_\H
	\langle (\alpha^\dagger \otimes id)(P\otimes \Delta a), P\otimes b
	\rangle_{\B \otimes L^2(A)} \\
& = &	d_\H \langle P\otimes \Delta a, \alpha (P)\otimes b 
	\rangle_{\B \otimes L^2(A)\otimes L^2(A)} \\
& = & 	d_\H (tr\otimes \h \otimes \h )((\alpha (P)\otimes b^*)(P\otimes \Delta a)).  
\end{eqnarray*}
Because the Haar state is nondegenerate, we must have 
\begin{eqnarray*}
(\s \circ \bs )(a) & = & d_\H (tr\otimes \h \otimes id)((\alpha (P)\otimes 1)
	(P\otimes \Delta a)) \\
& = & d_\H (\h \otimes id)((\s_P\otimes 1)\Delta a) =  \: _\hp\!\Delta (a) .
\end{eqnarray*}
\end{proof}

Thus far, we have described how, given a Kac-type compact quantum group $A$, a distinguished unitary irreducible representation $u\in B(\H )\otimes \A $, and a distinguished vector $\xi \in \H$, we can construct the Berezin symbol $\s :\B \rightarrow A^H$ and its adjoint $\bs :A^H\rightarrow \B$.  To construct the Berezin Quantization of $A^H$, we must find a sequence of irreducible representations $(u^n, \H^n)$ with corresponding rank one projections $P^n$ that all have the same state stabilizer $H$.

In the case of compact semisimple Lie groups, one takes the vector $\xi$ above to be the highest weight vector for the representation $u$.  One then looks at the subrepresentation of $(u^{\otimes n},\H^{\otimes n} )$ generated by $\xi^{\otimes n}$.  This subrepresentation is denoted by $(u^n,\H^n )$.  It is known that $(u^n, \H^n)$  is irreducible and that $\xi^n$ is it's highest weight vector.  It is also true in this case that $\B^n$ is all of $B(\H^n)$, and that $A^H=C(G/H)$ is a subalgebra of $A=C(G)$.  So one can define a collection  of $C^*$-algebras $A_\hbar$ for $\hbar \in \{ 0 \}\cup 1/\N$ by $A_0=C(G/H)$, and $A_\hbar =B(\H^{1/\hbar })$, which can be shown to give  a strict quantization of $C(G/H)$.

In order to get a strict quantization  of our operator system $A^H$, we must make one final definition.

\begin{definition}\label{def6}
Let $(u,\H )$ be an irreducible unitary representation of a compact quantum group $A$. Let $\xi \in \H$. We call $\xi$ primitive-like if the cyclic subrepresentation of the representation $u^{\otimes n}\in B(\H^{\otimes n})\otimes \A$ generated by $\xi^{\otimes n}$ is irreducible for each $n\in \N$.
\end{definition}

Suppose our irreducible unitary representation $(u,\H )$ has a primitive-like vector $\xi$.  We let $(u^n, \H^n)$ be the irreducible cyclic subrepresentation of $\H^{\otimes n}$ generated by $\xi^{\otimes n}$, and we write $\alpha^n$ for the corresponding action of $A$ on $B(\H^n)$.  By Proposition \ref{nthpower} we have that the state stabilizer of $P$ also stabilizes $P^n$ for all $n\in \N$.  Hence we have a sequence of maps $\s^n:\B^n\rightarrow A^H$ and $\bs^n:A^H\rightarrow \B^n$.    We now will show that there is a strict quantization of our operator system $A^H$ such that $A_\hbar$ for $\hbar \in \{ 0 \}\cup 1/\N$ given  by $A_0=A^H$, and $A_\hbar =\B^{1/\hbar }$.  In the next subsection we will reinterpret this result in terms of Lip-norms.

We have the sequence of Berezin Transforms $(\s^n\circ \bs^n):A^H\rightarrow A^H$.  By Proposition \ref{prop7} we have that $(\s^n\circ \bs^n)(a)=\: _\hpn\!\Delta (a)$, where $\hpn \in S(A)$
is  given by $\hpn (a)=d_{\H^n}\h (\s_{P^n}^n a)$.  By the proof of Proposition \ref{nthpower}, we also have that $\s_{P^n}^n=(\s_P)^n$.
We will use this observation to show that the sequence $\{ \hpn \}_{n\in \N}$ converges weak$^*$ to the counit $\varepsilon$ on $A^H$.  Among other things, this will imply that the sequence of operators $(\s^n \circ \bs^n ):A^H\rightarrow A^H$ converges strongly to the identity operator on $A^H$.  This explains why we had to use the state stabilizers rather than quantum group stabilizers.  Indeed, if one extends the definition of $\bs$ to $A$ by the same formula, $\s \circ \bs$ maps all of $A$ into $A^H$, so the sequence of maps $(\s^n \circ \bs^n )$ cannot converge to the identity on any subspace of $A$ strictly larger than $A^H$.  In particular, if $\Pi :A\rightarrow B$ is the quantum group stabilizer of $P$, then $(\s^n \circ \bs^n )$ converges to the identity on $A^B$ only when $A^B$ coincides with $A^H$.

We have $\hpn (a)=d_{\H^n} \h (\s_{P^n}^n a)=d_{\H^n} \h ((\s_P)^n a)$.  We have seen that $\hpn$ is a state, so the factor $  d_{\H^n}$ can be thought of as just a normalization to make $\hpn$ unital.  For computational purposes then, we use the expression
 $\hpn (a)= \h ((\s_P)^n a)/\h ((\s_P)^n)$.   We denote  $L^2(A,\h )$ by $\K $, and the corresponding cyclic vector by $\Omega$.  In terms of the inner product on $\K$ we have $\h (a)=\langle a\Omega ,\Omega \rangle$ and
\[
\hpn (a)= \frac{\langle a \Omega, (\s_P)^n \Omega \rangle}{\langle  \Omega, (\s_P)^n \Omega \rangle}
\]
From these expressions it's clear that $\h$ and $\hpn$ extend by weak continuity to $W^*(A) \subseteq B(\K )$.  We use the same notations $\h$ and $\hpn$ for these extensions. For $\gamma >0$, we denote by $\chi_\gamma$ the element of $W^*(\s_P)$ corresponding to the characteristic function of the interval $[1-\gamma, 1]$.  Our first task is to estimate $\hpn ((1-\chi_\gamma)a)$ in terms of $n$ for fixed $\gamma$. For the first time, we use our assumption that $A$ is coamenable.

\begin{lemma}\label{lem8}
For any $a\in A$, any $\gamma >0$, and any $n\in \N$ we have
\[ |\hpn ((1-\chi_\gamma)a)| \leq \frac{\| a\| }{\h (\chi_{\gamma /2})}
       \left( \frac{1-\gamma}{1-\gamma /2}\right)^n
\]
In particular, for $\gamma$ fixed we have $\hpn ((1-\chi_\gamma)a)$ goes to zero for any $a\in A$ as $n\longrightarrow \infty$.
\end{lemma}

\begin{proof}
We have $\hpn  ((1-\chi_\gamma)a) = \langle (1-\chi_\gamma) a \Omega, (\s_P)^n \Omega \rangle/\langle  \Omega, (\s_P)^n \Omega \rangle$.  We look at the numerator first.  Using that $\h $ is a trace and that $\s_P$ commutes with $1-\chi_\gamma$ we compute:
\begin{eqnarray}
|\langle (1-\chi_\gamma) a \Omega, (\s_P)^n \Omega \rangle | & = &
|\h ((\s_P)^n(1-\chi_\gamma)a) | \nonumber \\
& = & |\h ((1-\chi_\gamma) (\s_P)^{n/2}a(\s_P)^{n/2}(1-\chi_\gamma))|  \nonumber \\
& \leq & \| a\| \h ((1-\chi_\gamma)(\s_P)^n)\leq \| a\| (1-\gamma)^n . \nonumber
\end{eqnarray}
In the last inequality we use that $\| (1-\chi_\gamma)(\s_P)^n\| \leq (1-\gamma)^n$.

Now for the denominator.  We have
$\h ((\s_P)^n)\geq \h (\chi_{\gamma/2}(\s_P)^n) \geq (1-\gamma/2)^n\h (\chi_{\gamma/2})$.
Since $A$ is coamenable,  $\h$ is faithful on  $A$.  It is clear that there are positive, nonzero elements of $C^*(\s_P)$ that are less than $\chi_{\gamma/2}$.  Since $\h$ is strictly positive on any such  element, we must have $\h (\chi_{\gamma/2})>0$ as well.  Putting our estimates for the numerator and the denominator together we obtain the desired result.

\end{proof}

Next we  estimate the value of $\hpn (\chi_\gamma a)$ as $\gamma$ goes to zero.

\begin{lemma}\label{lem9}
Let $\Phi :A\rightarrow B$ be the state stabilizer of $P$, and let $a\in A$.  We have
\[
\limsup_{\gamma\longrightarrow 0} \left( \sup_{n\in \N} |\hpn (\chi_\gamma a)|\right) \leq \| \Phi (a)\|  \]
\end{lemma}
\begin{proof}
We begin the estimate as we did in the previous lemma.  We have
\begin{eqnarray*}
|\hpn (\chi_\gamma a)| & = & |\h (\chi_\gamma a (\s_P)^n)|/\h ((\s_P)^n) \\ 
& = &
| \h ((\s_P)^{n/2}\chi_\gamma a \chi_\gamma (\s_P)^{n/2})| /\h ((\s_P)^n) 
\\ & \leq &
\| \chi_\gamma a \chi_\gamma \| \h ((\s_P)^n)/\h ((\s_P)^n) 
\\ & = &
\| \chi_\gamma a \chi_\gamma \| .
\end{eqnarray*}
Define a map $\Psi_\gamma :A\rightarrow  B(\chi_\gamma \K )$ by $\Psi_\gamma (a)
= \chi_\gamma a \chi_\gamma$.  Note that $\Psi_\gamma$ is positive.  Since $\chi_\gamma$ is the unit of $ B(\chi_\gamma \K )$, we have that $\Psi_\gamma$ is also unital.
Finally, we note that $\Psi_\gamma (\s_P) = \chi_\gamma \s_P \chi_\gamma \geq (1-\gamma) \chi_\gamma = (1-\gamma ) 1_{B(\chi_\gamma \K )}$, so that
$(1-\gamma ) 1_{B(\chi_\gamma \K )} \leq \Psi_\gamma (\s_P) \leq 1_{B(\chi_\gamma \K )}$.

Next we show $\ker (\Psi_\gamma )\subseteq \ker (\Psi_\delta )$ for $\gamma > \delta$.  Let $a\in A$.  Then $a\in \ker (\Psi_\gamma )$ if and only if
$\chi_\gamma a \chi_\gamma =0$ if and only if
$\langle \chi_\gamma a \chi_\gamma T\Omega, S\Omega \rangle = 0$ for all $T,S\in B(\K )$.
Let $b,c \in A$ be arbitrary.   Set $T= \chi_\delta b$ and $S=\chi_\delta c $.  We have $0= \langle \chi_\gamma a \chi_\gamma  \chi_\delta b \Omega, \chi_\delta c\Omega \rangle = \langle \chi_\delta a   \chi_\delta b \Omega,  c\Omega \rangle$, where we have used $\chi_\gamma  \chi_\delta = \chi_\delta$ for $\gamma > \delta$.  Since $b,c$ were arbitrary and $A\Omega$ is dense in $\K$, we have
$a\in \ker (\Psi_\gamma )$ implies $a\in \ker (\Psi_\delta )$.

We can now define $\mathcal{I} = {\bigcup_{\gamma \searrow 0} \ker (\Psi_\gamma)}$.  By the preceeding paragraph, $\mathcal{I}$ is an increasing union of complex order ideals, and so it is an order ideal.  Thus the natural map $\Psi :A\rightarrow A/\mathcal{I}$ is positive.  By construction,  $\Psi :A\rightarrow A/\mathcal{I}$ factors through each of the maps $A\rightarrow A/{\ker (\Psi_\gamma)}$.
We saw above that $(1-\gamma ) 1_{B(\chi_\gamma \K )} \leq \Psi_\gamma (\s_P) \leq 1_{B(\chi_\gamma \K )}$. Thus we must have $(1-\gamma )1_{A/\mathcal{I}}
\leq  \Psi (\s_P)  \leq 1_{A/\mathcal{I}}$ for any $\gamma >0$, so that
 $\Psi (\s_P) = 1_{A/\mathcal{I}}$.
By Lemma \ref{lem4} we have  that the state stabilizer  is the universal positive unital map $\Phi :A\rightarrow B$ such that $\Phi (\s_P) = 1_B$, so that $\Psi$ factors through $\Phi$.  Therefore $\| \Psi (a) \| \leq \| \Phi (a) \|  $ for all $a\in A$.

At the beginning of this proof we showed that
$ |\hpn (\chi_\gamma a)| \leq \| \chi_\gamma a \chi_\gamma \| =\| \Psi_\gamma (a) \| $
 for any $n\in \N$.  Therefore we obtain:
\[
\limsup_{\gamma\longrightarrow 0} \left( \sup_{n\in \N} |\hpn (\chi_\gamma a)|\right) \leq \| \Psi (a)\| \leq \| \Phi (a)\| .
\]

\end{proof}

\begin{lemma}\label{lem10}
Let $\Phi :A\rightarrow B$ be the state stabilizer of $P$, and let $a\in A^H$.  Then $\Phi (a)=\varepsilon (a) 1_B$.
\end{lemma}
\begin{proof}
By definition, $a\in A^H$ means that
 $(\Phi \otimes id)\Delta a= 1_B \otimes a$.
We have

 $\Phi (a)=\Phi (id \otimes \varepsilon )\Delta a=
(id \otimes \varepsilon )(\Phi \otimes id)\Delta a=
(id \otimes \varepsilon )(1_B\otimes a)=\varepsilon (a) 1_B$.
\end{proof}

Finally, we obtain our desired result.

\begin{proposition}\label{prop9}
Let $a\in A^H$.
Then $\lim_{n\rightarrow \infty} \hpn (a) = \varepsilon (a)$.
\end{proposition}
\begin{proof}
Let $a\in A^H$ and let $b= a-\varepsilon (a)1$.  We have $\varepsilon (b)=0$, so by Lemma \ref{lem10} $\Phi (b) =0$.  For any $\gamma >0$, using Lemma \ref{lem8} gives
$\lim_{n\rightarrow \infty} | \hpn (b) - \hpn (\chi_\gamma b)|=
\lim_{n\rightarrow \infty}  | \hpn ((1-\chi_\gamma)b)| =0$.
Thus $\limsup_{n\rightarrow \infty} |\hpn (b)| =
\limsup_{n\rightarrow \infty} |\hpn (\chi_\gamma b)|
\leq \sup_{n\in \N} |\hpn (\chi_\gamma b)| $, for all $\gamma >0$.
Now, using Lemma \ref{lem9} we have
$\limsup_{n\rightarrow \infty} |\hpn (b)| \leq \| \Phi (b)\| = 0$.
 So $\lim_{n\rightarrow \infty} \hpn (b)$ exists and is equal to zero.  Finally, we have $\lim_{n\rightarrow \infty} \hpn (a) =
\lim_{n\rightarrow \infty} \hpn (b+\varepsilon (a)1) = \varepsilon (a)$.

\end{proof}

As a corollary of the weak convergence of $\hpn$, we can show that we have a strict quantization  of the operator system  $A^H$.

\begin{proposition}\label{prop10}
Let $a\in A^H$.  We have $(\s^n \circ \bs^n)(a) \rightarrow a$ in norm as $n\rightarrow \infty$.  Moreover, the maps $\bs^n:A^H\rightarrow \B^n$ give a strict quantization of $A^H$.
\end{proposition}
\begin{proof}
Let $a\in A^H$. By Proposition \ref{prop7} we have
 $(\s^n \circ \bs^n)(a) = (\hpn \otimes id) \Delta a$.
By Lemma \ref{lem5} we have $\Delta a \in A^H\otimes A$.  By Proposition \ref{prop9} we see
 $(\hpn \otimes id)(\omega )\rightarrow (\varepsilon \otimes id)(\omega )$
 for any $\omega \in A^H\otimes A$.
But $(\varepsilon \otimes id) \Delta a = a$, so we have the first part of the proposition.

The second statement follows immediately from the first using Theorem \ref{deformquant}.
\end{proof}

\begin{remark} 
The use of coamenability in the proof of Lemma \ref{lem8} appears to be insignificant, but in the appendix we will give an example that shows it is essential.
\end{remark}

\end{subsection}

\begin{subsection}{The Convergence of $\s^n \circ \bs^n$ in Terms of Lip-norms.}

Let $L_A$ be any regular Lip-norm on $A$.  We have ergodic actions of $A$ on $A^H$ and on each $\B^n$.  Thus $L_A$ induces a Lip-norm $L$ on $A^H$ as well as Lip-norms $L_n$ on each $\B^n$ as described in Section 2.5.  Our goal now becomes to show that $(\B^n, L_n)$ converges to $(A^H, L)$ for quantum Gromov-Hausdorff distance as $n\longrightarrow \infty$.  Our first step is to reinterpret Proposition \ref{prop10} in terms of the Lip-norm.  We begin with a technical lemma.

\begin{lemma}\label{technical}
Let $V$ and $W$ be order unit spaces.  Let $K\subseteq V\otimes W$ be totally bounded for the order norm, and let $\phi \in S(V)$ be such that $(\phi \otimes id)(K)=0$.  Let $\delta >0$.   Then there exists a function $f:S(V)\rightarrow [0, \infty)$, continuous in the $w^*$-topology such that
\begin{eqnarray}
\| (\mu \otimes id)(\omega )\| \leq f(\mu ), \forall \mu \in S(V), \forall \omega \in K, \\
\text{and} \hspace{1.5in} f(\phi )\leq \delta .
\end{eqnarray}
\end{lemma}
\begin{proof}
Since $K$ is totally bounded, it is bounded.  Thus for some $\lambda \in [0, \infty)$ the constant function $g$ defined by $g(\mu )=\lambda$ satisfies equation $(1)$.  Let $f$ be an element of $C(S(V), \R )$ satisfying condition $(1)$.  If $f(\phi )=0$, we're done.  If $f(\phi )>0$, we produce a function $F$ satisfying condition $(1)$ such that $F(\phi )= 2f(\phi )/3$.  Repeating the process sufficiently many times will then give a function satisfying $(1)$ and $(2)$.

Suppose $f$ is a continuous function on $S(V)$ satisfying condition $(1)$, and suppose that $f(\phi) >0$.  There is a Gelfand-like association $\hat{\omega}(\mu ) =\| (\mu \otimes id) (\omega )\|$ from $K$ into $C(S(V))$.  Since $\hat{K}\cup \{ f\} \subseteq C(S(V))$ is totally bounded, by Ascoli Arzela there exists a neighborhood $\mathcal{U}$ of $\phi$ such that for all $\nu \in \mathcal{U}$ and all $\omega \in K$ we have
\begin{eqnarray}
\| (\nu \otimes id)(\omega ) \| = \hat{\omega}(\nu ) < f(\phi )/3  \nonumber \\
f (\nu )>2 f(\phi)/3.   \nonumber
\end{eqnarray}
Let $g\in C(S(V))$ be a nonnegative function such that $g(\phi )=f(\phi )/3 =\| g\|$, and $g(\mu )=0$ for all $\mu \notin \mathcal{U}$.  Then a quick check shows $F =(f-g)$ satisfies condition $(1)$, and has $F(\phi )= 2f(\phi )/3$.  Indeed,  for $\mu \notin \mathcal{U}$ we have $(f-g)(\mu )=f(\mu )\geq \| (\mu \otimes id)(\omega ) \|$ for all $\omega \in K$, while for $\nu \in \mathcal{U}$ we have $  (f-g)(\nu )\geq 2f(\phi )/3 -f(\phi )/3 = f(\phi )/3 > \| (\nu \otimes id)(\omega ) \|$ for all $\omega \in K$.

\end{proof}

We can now reinterpret Proposition \ref{prop10} in terms of  the Lip-norm of $A^H$.

\begin{proposition}\label{estimate1}
Let $L$ be any Lip-norm on $A^H$.  Then there exists a sequence of positive real numbers $\{ \delta_n\}_{n\in \N}$ converging to zero such that
$\| (\s^n \circ \bs^n)(a)-a\| < L(a)\delta_n$ for all $a\in A^H$.
\end{proposition}
\begin{proof}
We show that for any $\delta >0$, there exists a $N_\delta \in \N$ such that $\| (\s^n \circ \bs^n)(a)-a\| < L(a)\delta$ for all $a\in A$ and all $n\geq N_\delta$.

Consider the linear operator $T: A^H\rightarrow A^H\otimes A$ by
$T(a)= \Delta a-1\otimes a$.  Clearly $T$ is bounded and vanishes on $\C  1_A$.  Since $L$ is a Lip-norm, we have $K=T\{ a\in A^H \: | \: L(a)\leq 1\}$ is totally bounded.
 Also $(\varepsilon \otimes id)T(a) = (\varepsilon \otimes id)(\Delta a- 1\otimes a)=a-a=0$.  By the lemma, there exists a function $f\in C(S(A^H))$ such that $f(\varepsilon )\leq \delta/2$ and
$\| (\mu \otimes id)(\Delta a-1\otimes a)\| \leq f(\mu )$  for all $\mu \in S(A^H)$.  Let $\mu =\hpn$, for some $n$ to be determined.  Then for $L(a)\leq 1$ we have
$\| (\s^n \circ \bs^n)(a)-a\| =
\| (\hpn \otimes id)(\Delta a-1\otimes a)\| \leq f(\hpn )$.
Since $f$ is continuous in the $w^*$-topology and $f(\varepsilon )\leq \delta/2$, we can find a $N_\delta \in \N$ such that $f(\hpn ) <\delta$ for all $n\geq N_\delta$.   Then for all $a\in A^H$ and all $n\geq N_\delta $ we have
$\| (\s^n \circ \bs^n)(a)-a\| \leq L(a) \delta$.

\end{proof}

\end{subsection}

\begin{subsection}{Equivariance and $\bs \circ \s$}

In the last section we estimated $\| (\s^n \circ \bs^n)(a)-a\|$ in terms of the Lip-norm on $A^H$.  Our goal in this section is to derive a similar estimate for $\| (\bs^n \circ \s^n)(T)-T\|$ in terms of the Lip-norm on $\B^n$.  Basically we will be able to do this from applying Proposition \ref{prop9} to the isotypic components of $A^H$.  Namely,  Proposition \ref{prop9} guarantees that $(\s^n \circ \bs^n)$ converges strongly to the identity on $A^H$.  Since the isotypic components of $A^H$ are finite dimensional, on any given component   $(\s^n \circ \bs^n)$ will converge to the identity in norm.  We note that the method of this section can be used to simplify the arguments of Sections 4 and 5 of \cite{Rieffel3}.

Let $\mathcal{S} \subseteq \hat{A}$ be finite.  By  Proposition \ref{invariant} we have restricted maps $\s^n |_\mathcal{S}:\B^n_\mathcal{S}\rightarrow A^H_\mathcal{S}$ and $\bs^n |_\mathcal{S}:A^H_\mathcal{S} \rightarrow \B^n_\mathcal{S}$.
 We will supress the $|_\mathcal{S}$ on $\s$ and $\bs$ when the meaning is clear.  As we noted before, it is not clear in general whether the isotypic components of an ergodic ordered $A$-module must be finite dimensional.  However, since $A^H$ is an $A$-submodule of the regular $A$-module $\Delta :A\rightarrow A\otimes A$, the isotypic components of $A^H$ are contained in the isotypic components of $A$, which are finite dimensional. (This follows either from Boca, since $A$ is an $A$-module algebra \cite{Boca}, or from the Peter-Weyl theorem for quantum groups \cite{Woro1}).
  Proposition \ref{prop9} gives that $(\s^n \circ \bs^n)$ converges strongly to the identity on $A^H$, so $(\s^n \circ \bs^n)|_\mathcal{S}$ converges strongly to the identity on $A^H_\mathcal{S}$.
Since $A^H_\mathcal{S}$ is finite dimensional,  we have that $(\s^n \circ \bs^n)|_\mathcal{S}$ converges to the identity on $A^H_\mathcal{S}$ in norm.  This has some consequences:

\begin{corollary}
There exists an $N_\mathcal{S}\in \N$ such that $(\s^n \circ \bs^n): A^H_\mathcal{S}\rightarrow A^H_\mathcal{S}$ is invertible for all $n\geq N_\mathcal{S}$.  Moreover, we can choose $N_\mathcal{S}$ so that $\| (\s^n \circ \bs^n)^{-1}\| <2$ for all $n\geq N_\mathcal{S}$.
\end{corollary}
\begin{proof}
Both statements follow from the facts that the identity operator on $A^H_\mathcal{S}$ is invertible, the set of invertible operators is norm open, and the inversion map is norm continuous on its domain.
\end{proof}

\begin{corollary}\label{last}
Let $N_\mathcal{S}$  be as given from the last corollary.  We have $\s^n: \B^n_\mathcal{S} \rightarrow A^H_\mathcal{S}$ and
$\bs^n: A^H_\mathcal{S}\rightarrow \B^n_\mathcal{S}$ are both invertible for $n\geq N_\mathcal{S}$, and $\| (\s^n)^{-1}\|$ and $\| (\bs^n)^{-1}\|$ are each less than $2$ for  $n\geq N_\mathcal{S}$.
\end{corollary}
\begin{proof}
By Proposition \ref{prop5}, $\s^n: \B^n_\mathcal{S} \rightarrow A^H_\mathcal{S}$ is injective.  By Proposition \ref{prop6} and by equivariance, $\bs^n: A^H_\mathcal{S}\rightarrow \B^n_\mathcal{S}$ is surjective.  By the previous corollary $(\s^n \circ \bs^n)$ is invertible for $n\geq N_\mathcal{S}$.  Thus for  $n\geq N_\mathcal{S}$ we have $\s^n: \B^n_\mathcal{S} \rightarrow A^H_\mathcal{S}$ is surjective and $\bs^n: A^H_\mathcal{S}\rightarrow \B^n_\mathcal{S}$ is injective so that they are both invertible.  Furthermore we have
$\| (\s^n)^{-1} \| = \| \bs^n\circ (\s^n \circ \bs^n)^{-1}\| \leq 2$ and
$\| (\bs^n)^{-1} \| = \|  (\s^n \circ \bs^n)^{-1} \circ \s^n \| \leq 2$ for all $n\geq N_\mathcal{S}$.
\end{proof}

\begin{corollary}\label{ball}
The closed ball $B_0(2)\subseteq A^H_\mathcal{S}$ of radius $2$ about the origin has the property that $\bs^n (B_0(2)) \supseteq $ (unit ball of $\B^n_\mathcal{S}$) for every $n\geq N_\mathcal{S}$.
\end{corollary}
\begin{proof}
 This is immediate from  Corollary \ref{last}.
\end{proof}

\begin{proposition}
Let $\mathcal{S} \subset \hat{A}$ be any finite subset of $\hat{A}$, and let $N_\mathcal{S}\in \N$ be as above.  Then there exists a $K_\mathcal{S}>0$ such that  for any $n\geq N_\mathcal{S}$ and any $T\in \B^n_\mathcal{S}$ we have
$\| T- (\bs^n \circ \s^n)(T)\| \leq L_n(T) K_\mathcal{S} \delta_n$, where $\delta_n$ is the sequence obtained in Proposition \ref{estimate1}.
\end{proposition}
\begin{proof}
Let $n\geq N_\mathcal{S}$, and let $T\in \B^n_\mathcal{S}$ with $\| T\| \leq 1$.
By Corollary \ref{ball}  there exists an $a_n\in B_0(2)$ such that $\bs^n(a_n) =T$.
We compute: $\\
\| T-(\bs^n \circ \s^n)(T)\| = \| \bs^n (a_n-(\s^n\circ \bs^n)(a_n))\|
\leq \|a_n-(\s^n\circ \bs^n)(a_n)\| \leq L(a_n)\delta_n.\\$
Since $L_A$ was assumed to be a regular Lip-norm, $L_A$ is finite on $A_\mathcal{S}$, which implies $L$ is finite on $A^H_\mathcal{S}$.  Hence the restriction of $L$ to the finite dimensional space $A^H_\mathcal{S}$ is just a usual (finite-valued) seminorm, so that it is continuous with respect to the norm.  Let $J_\mathcal{S}$ be the maximum value of $L$ on the compact set $B_0(2)$. Then for all $T\in  \B^n_\mathcal{S}$ we have $\| T-(\bs^n \circ \s^n)(T)\| \leq \| T\| J_\mathcal{S}\delta_n$.  Since the expression on the left side vanishes on multiples of the identity, Lemma \ref{stupid} gives $\| T-(\bs^n \circ \s^n)(T)\| \leq L_n( T) (2r_A J_\mathcal{S})\delta_n$.
\end{proof}

We can now prove the main result of this section.  Here we use for the first time that $L_A$ on $A$ is right-invariant, so that the induced Lip-norm $L_n$ on $\B^n$ is invariant by Proposition \ref{li1}.

\begin{proposition}\label{estimate2}
There exists a sequence of positive real numbers $\{ \theta_n\}_{n\in \N}$ converging to zero such that $\| T-(\bs^n \circ \s^n)(T)\| \leq L_n( T) \theta_n$ for all $T\in \B^n$.
\end{proposition}
\begin{proof}
Let $\theta >0$.  We show there exists an $N_\theta \in \N$ such that $\| T-(\bs^n \circ \s^n)(T)\| \leq L_n( T) \theta$ for all $T\in \B^n$, for all $n\geq N_\theta$.
By Lemma \ref{li2} there exists a finite subset $\mathcal{S}\subseteq \hat{A}$ and a $\psi \in S(A)$ not depending on $n\in \N$, and such that
$\\ \alpha^n_\psi ( \B^n)\subseteq \B^n_{\mathcal{S}}$,
       $\\ \| \alpha^n_\psi (T) \| \leq \| T \|$ for all $T\in \B^n$, and 
       $\\ \| T-\alpha^n_\psi (T) \| \leq L_n(T)\theta /3 $ for all $T\in \B^n. \\ $
We compute:
\begin{eqnarray} \| T-(\bs^n \circ \s^n)(T)\| & \leq &
\| T- \alpha^n_\psi (T)\|  +  \| \alpha^n_\psi (T)-(\bs^n \circ \s^n )(\alpha^n_\psi (T))\|  +
       \|  (\bs^n \circ \s^n )(\alpha^n_\psi (T) - T)\| \nonumber \\
       & \leq &  L_n(T)\theta /3  +  L_n(\alpha^n_\psi (T)) (K_\mathcal{S}\delta_n)
        +  L_n( T)\theta /3 \nonumber \\
       & \leq & L_n(T) (2\theta/3 + K_\mathcal{S}\delta_n) \nonumber
\end{eqnarray}
Since $\lim_{n\rightarrow \infty} \delta_n =0$, there exists an $N_\theta$ such that for all $n\geq N_\theta$, for all $T\in \B^n$ we have
 $\| T-(\bs^n \circ \s^n)(T)\| \leq L_n(T)\theta .$

\end{proof}
\end{subsection}

\begin{subsection}{Convergence in Quantum Gromov-Hausdorff Distance}

We use the results of Propositions \ref{estimate1} and \ref{estimate2} to show that the quantum metric spaces $(\B^n, L_n)$ converge to $(A^H, L)$ in quantum Gromov-Hausdorff distance as $n\longrightarrow \infty$.  First, we prove a lemma that relates Lip-norms to equivariant positive maps.

\begin{lemma}\label{equivarLip}
Let $A$ be a compact quantum group, and let $(\pi_1, C_1)$ and $(\pi_2, C_2)$ be two ergodic ordered $A$-modules.  Let $L_A$ be a Lip-norm on $A$ and let $L_1$ and $L_2$ be the induced Lip norms on $C_1$ and $C_2$.  Suppose $\rho :C_1\rightarrow C_2$ is a positive, unital, and $A$-equivariant map.  Then for any self-adjoint $x\in C_1$ we have $L_2 (\rho (x))\leq L_1 (x)$.
\end{lemma}
\begin{proof}
We have $\pi_1:C_1\rightarrow C_1 \otimes A$,
$\pi_2:C_2\rightarrow C_2 \otimes A$,  and
 $\pi_2 \circ \rho = (\rho \otimes id )\circ \pi_1$. Let $\varphi \in S(C_2)$ be arbitrary.  Since $\rho$ is positive and unital, we have $\varphi \circ \rho \in S(C_1)$.
Let $x\in C_1$ be self-adjoint.  Then
$L_A(\: _\varphi\pi_2(\rho (x))  ) = 
L_A( \: _{\varphi \circ \rho }\pi_1 (x)) \leq L_1 (x)$.  Taking the supremum over $\varphi \in S(C_2)$ gives the desired result.
\end{proof}

In particular,  for $a\in A^H_{sa}$ and $T\in \B^n_{sa}$ we have $L_n (\bs^n (a)) \leq L(a)$ and $L(\s^n (T))\leq L_n(T)$.  We proceed as in \cite{Rieffel3}.   For any $\epsilon >0$ we define a Lip-norm on $A^H\oplus \B^n$ by
\[ L^{n,\epsilon}(a,T) = L(a)\vee L_n(T) \vee \epsilon ^{-1}\| a-\s^n(T)\|
\hspace{.75in} \text{for } a=a^*,\: T=T^* .  \]
We now can use $L^{n, \epsilon}$ to estimate the quantum Gromov-Hausdorff distance between $(A^H, L)$ and $(\B^n, L_n)$.  We first show that for any $\epsilon >0$, the quotient seminorms of $L^{n,\epsilon}$ on $\B^n_{sa}$ and on $A_{sa}$ are simply $L_n$ and $L$, for sufficiently large $n\in \N$.
\begin{lemma}
For any $n\in \N$ and $\epsilon >0$ we have that the quotient seminorm $L^{n,\epsilon}$ on $\B^n_{sa}$ is just $L_n$.
\end{lemma}
\begin{proof}
It's clear from the formula for $L^{n,\epsilon}$ that $L^{n,\epsilon}(a,T) \geq L_n(T)$ for any $a\in A^H_{sa}$ and any $T\in \B^n_{sa}$, so the quotient seminorm of $L^{n,\epsilon}$ is no smaller than $L_n$.  On the other hand, Lemma \ref{equivarLip} gives that $L^{n,\epsilon} (\s^n(T) , T)=L_n(T)$, so the quotient seminorm is precisely $L_n$.
\end{proof}

\begin{lemma}\label{nbound}
 Let $\epsilon >0$ be given.   The quotient seminorm of $L^{n,\epsilon}$ on $A^H_{sa}$ is  $L$ for all sufficiently large $n$, namely when $\delta_n \leq \epsilon$.
\end{lemma}
\begin{proof}
As before, the quotient seminorm is clearly no smaller than $L$. Let $\delta_n$ be as in Proposition \ref{estimate1}.  For all sufficiently large $n$ we will have  $\delta_n \leq \epsilon$.  Using Proposition \ref{estimate1} and Lemma \ref{equivarLip},  we compute
\[ L^{n,\epsilon}(a, \bs^n_a) =
L(a)\vee \epsilon ^{-1} \| a-\s^n(\bs^n_a)\| \leq
L(a) \vee \epsilon^{-1} \delta_n L(a) =L(a) \]
 so the quotient seminorm of $L^{n,\epsilon}$ is precisely $L$ whenever $\delta_n \leq \epsilon$.
\end{proof}

Let $\epsilon >0$ be given and choose $n$ so that the result of Lemma \ref{nbound}  holds.  Then $L^{n,\epsilon}$ induces a metric on $S(A^H\oplus \B^n)$ via:
\[ \rho_{n, \epsilon}(\mu , \nu )= \sup \{ |\mu (c) -\nu (c)|
: c\in (A^H\oplus \B^n)_{sa} , L^{n,\epsilon}(c) \leq 1\}. \]
For any large enough $n$,  we now estimate the distance between $S(\B^n)$ and $S(A^H)$ in $S(A^H\oplus \B^n)$ for the metric $\rho_{n, \epsilon}$.

\begin{lemma}
Let $\epsilon >0$, and choose $n$ so that the result of Lemma \ref{nbound} holds.  Then $S(A^H)$ is in the $\epsilon$ neighborhood of $S(\B^n)$ as measured by $\rho_{n, \epsilon}$.
\end{lemma}
\begin{proof}
Let $\mu \in S(A^H)\subseteq S(A^H\oplus \B^n)$.  We must find $\nu \in S(\B^n) \subseteq S(A^H\oplus \B^n)$ such that $\rho_{n, \epsilon}(\mu , \nu )\leq \epsilon$.  Since $\s^n$ is positive and unital, we can take $\nu = \mu \circ \s^n$.  Let $(a, T)\in (A^H\oplus \B^n)_{sa}$ be such that $L^{n,\epsilon}(a,T) \leq 1$.  Then in particular, we have $\| a-\s^n_T\| \leq \epsilon$.  Thus
\[
|\mu (a, T)-\nu (a,T)|  =  |\mu (a)-\nu (T)| = |\mu (a)-\mu (\s^n_T)|
        \leq  \| a-\s^n(T)\| \leq \epsilon .
\]  Since this holds for all such $(a,T)$, we have $\rho_{n,\epsilon}(\mu ,\nu )\leq \epsilon$, as desired.
\end{proof}

\begin{lemma}
Let $\epsilon >0$ be given.  Then for all $n\in \N$ such that $\delta_n$ and $\theta_n$ are both less than $ \epsilon$,  we have that $S(\B^n)$ is in the $2\epsilon$ neighborhood of $S(A^H)$ for the metric  $\rho_{n, \epsilon}$.
\end{lemma}
\begin{proof}
Let $n$ be large enough so that $\delta_n$ and $\theta_n$ are both less than $\epsilon$, where $\delta_n$ and $\theta_n$ are from Propositions \ref{estimate1} and \ref{estimate2}.    Let $\nu \in S(\B^n) \subseteq S(A^H\oplus \B^n)$.  We must find $\mu \in S(A^H)\subseteq S(A^H\oplus \B^n)$ such that $\rho_{n, \epsilon}(\mu , \nu )\leq \epsilon$.  As in the previous proof, we can take $\mu = \nu \circ \bs^n$.  Let $(a, T)\in (A^H\oplus \B^n)_{sa}$ be such that $L^{n,\epsilon}(a,T) \leq 1$, so that $\| a-\s^n_T\| \leq \epsilon$ and $L_n(T)\leq 1$.  Then
\[
|\nu (a, T)-\mu (a,T)|  =   |\nu (T)-\nu (\bs^n_a)|   \leq  \| T- \bs^n_a\|
 \leq \| T- (\bs^n\circ \s^n)(T)\| + \| \bs^n ( \s^n_T- a)\|  \leq \theta_n +\epsilon
\]  Since this holds for all such $(a,T)$, we have $\rho_{n,\epsilon}(\mu ,\nu )\leq \theta_n +\epsilon \leq 2\epsilon  $, as desired.
\end{proof}

Combining the results of these four lemmas gives us our desired result.

\begin{theorem}\label{main1}
The compact quantum metric spaces $(\B^n, L_n)$ converge to $(A^H, L)$ in quantum Gromov-Hausdorff distance.
\end{theorem}

\end{subsection}

\end{section}

\begin{section}{Berezin Quantization for General Ergodic $A$-Modules}

Our previous result is actually a special case of a more general theorem.  The constructions in the general case are less intuitive than in the case above, so I thought it'd be best to start with the special case.  Also, we need some important proofs for the special case in order to prove the corresponding facts in the general case. 

One new feature is that we will now have to use both left and right $A$-modules.  For ordinary groups one can get around this, because if $\alpha$ is a left-action of an ordinary group $G$, then $\alpha \circ (\cdot )^{-1}$ is a right-action, and vice versa.  In the Kac case the coinverse is always bounded, so in principle one can still convert right modules to left modules, but doing so obscures the overall picture.  

\begin{definition}\label{rightmodalg}
Let $A$ be a compact quantum group.  A right action of $A$ on a C$^*$-algebra $N$ is a C$^*$-homomorphism $\pi :N\rightarrow A \otimes N$  with the properties
\begin{itemize}
\item  $(id\otimes \pi )\pi = (\Delta \otimes id)\pi$
\item  $(A\otimes 1)\pi (N)$ is total in $ A\otimes N$.
\end{itemize}
When $\pi :N\rightarrow A\otimes N$ is a right action of $A$ on a C$^*$-algebra $N$, we will call the pair $(\pi ,N)$ a right $A$-module algebra.
An ergodic right $A$-module algebra is a right $A$-module algebra $(\pi ,N)$ such that $x\in N$ and $\pi (x)=1_A\otimes x$ only if $x\in \C 1_N$.
\end{definition}

 As before,we start with  a coamenable compact quantum group $A$ of Kac-type. Let  $u:\H \rightarrow \H \otimes A$ be a unitary irreducible representation, and let  $\xi \in \H$ be a unit vector which is a primitive-like for $u$.  As before, we denote the rank $1$ projection corresponding to $\xi$ by $P$, and we let $H$ be the state stabilizer of $P$ for the corresponding action $\alpha :B(\H )\rightarrow B(\H )\otimes A$.
Suppose now we also have an ergodic right $A$-module algebra $\pi :N\rightarrow A\otimes N$.  In the following sections we will construct an analog $\s :M\rightarrow N^H$  of the Berezin symbol and its adjoint $\bs :N^H\rightarrow M$ for a certain operator system $M$.  The Berezin symbol and it's adjoint will be positive unital maps.  As before, we can consider the sequences of maps $\s^n :M^n\rightarrow N^H$.  We aim to prove the following theorem, which is stated here in a slightly imprecise way.

\begin{theorem}\label{ergodicversion}
The spaces $\{ M^n\}_{n\in \N}$ together with $N^H$ form a strict quantization of $N^H$.
Furthermore, the sequence of compact quantum metric spaces $\{ M^n\}_{n\in \N}$ converge to $N^H$ in quantum Gromov-Hausdorff distance.
\end{theorem}

\begin{subsection}{Motivation for M and the Berezin Symbol}

As before, we let $\B \subseteq B(\H )$ be the cyclic subspace of $B(\H )$ generated by $P$, and  we have $\alpha :\B \rightarrow \B \otimes A$.  Of course, the range of $\alpha$ is contained in the subspace $\{ \omega \in \B \otimes A \: | \: (\alpha \otimes id)(\omega )= (id\otimes \Delta )(\omega )\}$.  We will confirm shortly that $\alpha$ is actually an order isomorphism from  $\B$  to this subspace.  We defined before $A^H = \{ a\in A \: | \: (\Pi \otimes id)(a) = 1_B\otimes a\}$, where $\Pi :A\rightarrow B$ is the universal positive unital map stabilizing $P$.  The point here is that $A^H$ is actually the set of elements of $A$ fixed by $\Pi$ whenever $A$ is viewed as a right $A$-module algebra.  This is the reason we start with a right $A$-module algebra $\pi :N\rightarrow A \otimes N$, rather than a left module.  Based on these two observations, we are ready to define $M$.
\begin{definition}\label{defM}
Let $\alpha :\B \rightarrow \B \otimes A$ and $\pi :N\rightarrow A\otimes N$ be as above.  We define $M\subseteq \B \otimes N$ to be the  subspace
\[ M = \{ \omega \in \B \otimes N \: | \: (\alpha \otimes id)(\omega )= (id\otimes \pi )(\omega )\} . \]
\end{definition}
Note:  It is clear from the definition that $M$ is norm closed, unital, and $*$-closed, so that $M$ is an operator system.

\begin{example}
Let $N$ be $A$, viewed as a right $A$-module algebra.  We have $\B \cong M$.  Indeed, we have already noted that in this case $\alpha :\B \rightarrow M$.  We also have $id \otimes \varepsilon :M\rightarrow \B$.  As usual, we have $(id \otimes \varepsilon )\circ \alpha =id_\B$.
 If $\omega \in M$, then $\alpha [(id \otimes \varepsilon )(\omega ) ] = (id \otimes id\otimes \varepsilon )(\alpha \otimes id )(\omega ) =
       (id \otimes id\otimes  \varepsilon )(id\otimes \Delta )(\omega ) = \omega$, so that $\alpha \circ (id \otimes \varepsilon )= id_M$ as well.
Both $\alpha$ and $(id\otimes \varepsilon )$ are positive and unital. Thus  $\alpha :\B\rightarrow M$ is an isomorphism of operator systems.
\end{example}

The first thing we wish to show is that $M$ is in fact finite dimensional. The original motivation for the fuzzy sphere was to replace the 2+1 dimensional QFT on the sphere with the 0+1 dimensional QFT on the fuzzy sphere.  Thus it is natural to want our approximations $M^n$ of $N^H$ to be finite dimensional.

Since $\B$ has a left $A$-action $\alpha :\B \rightarrow \B \otimes A$, we have a dual right action $\alpha ':\B'\rightarrow A\otimes \B'$ of $A$ on the dual space $\B'$ by $\alpha'(\varphi )=(\varphi \otimes id)\alpha$.  

\begin{proposition}
The space $Lin_A(\B', N)$ of right $A$-module maps from $\B'$ to $N$ is linearly isomorphic to $M$.  In particular, $M$ is finite dimensional.
\end{proposition}
\begin{proof}
Let $\omega \in M$.  Define $\hat{\omega}:\B'\rightarrow N$ by $\hat{\omega}(\varphi )=(\varphi \otimes id)\omega$.  We must show that $\pi (\hat{\omega}(\varphi ))=(id\otimes \hat{\omega})(\alpha'(\varphi ))$.  We compute:
\begin{eqnarray}
\pi (\hat{\omega}(\varphi )) & = & \pi [(\varphi \otimes id)\omega ] 
	\nonumber \\
& = & (\varphi \otimes id\otimes id)(id\otimes \pi )\omega \nonumber \\
& = & (\varphi \otimes id\otimes id)(\alpha \otimes id )\omega \nonumber \\
& = & (\alpha'(\varphi )\otimes id)\omega =(id\otimes \hat{\omega})(\alpha'(\varphi )) \nonumber
\end{eqnarray}
where in the last line we view $\alpha'(\varphi )$ as a mapping $\B \rightarrow A$ on the left and as an element of $A\otimes \B'$ on the right. 

For the converse, if $\hat{\omega}:\B'\rightarrow N$ is any map, we can write 
$\hat{\omega}(\varphi )=(\varphi \otimes id)\omega$ for some $\omega \in \B \otimes N$.  Then the calculation above shows that $\hat{\omega}$ is $A$-equivariant only if $(\varphi \otimes id\otimes id)(\alpha \otimes id )\omega =
 (\varphi \otimes id\otimes id)(id\otimes \pi )\omega$ for all $\varphi \in \B'$.  Hence $\omega \in M$.

The second statement follows easily from the first.  Since $\B'$ is finite dimensional, we have the isotypic component $\B'_\gamma \neq 0$ for only finitely many $\gamma \in \hat{A}$.  By ergodicity, any particular isotypic component of $N$ is finite dimensional.  Thus we are essentially counting linear maps between two finite dimensional spaces, and $Lin_A(\B', N)$ is finite dimensional.

\end{proof}

Before we defined the Berezin symbol $\s: \B \rightarrow A^H$ via $\s (T)= (tr\otimes id )((P\otimes 1)\alpha (T))$.  Taking into account the isomorphism in the above example, we make the following definition.

\begin{definition}
Let $A$, $(\alpha, \B )$, $(\pi , N)$, and $M$ be as above.  The Berezin symbol $\s :M\rightarrow N$ is defined as the map:
\[ \s (\omega )=(tr\otimes id)((P\otimes 1)\omega ).\]
\end{definition}
It's clear that $\s$ is completely positive and unital.

As before, we can consider the $H$-invariant elements of $N$.  Namely, we let $N^H= \{ x\in N \: | \: _\phi\pi (x)= x, \phi \in H\}$.  

\begin{proposition}
The   Berezin symbol $\s :M\rightarrow N$ is injective, and its image is contained in $N^H$.
\end{proposition}
\begin{proof}
Let $\phi \in \A'$.  We compute:
\begin{eqnarray}
\: _\phi\pi (\s_\omega ) & = & \: _\phi\pi [(tr\otimes id)((P\otimes 1)\omega )]=
(tr\otimes id)((P\otimes 1)(id\otimes \phi \otimes id)(id\otimes \pi)\omega )
\nonumber \\
 & = & (tr\otimes \phi \otimes id)((P\otimes 1\otimes 1)(\alpha \otimes id)\omega )
= (tr\otimes id)((\alpha_{\phi \circ \kappa}(P)\otimes 1) \omega ). \nonumber
\end{eqnarray}
If $\s_\omega = 0$, then this calculation shows $\omega \in \B \otimes N$ must be zero.  Indeed, we can select from $\{ \alpha_\phi (P)\}$ an orthonormal basis for $\B$.  Writing $\omega$ as a sum of elementary tensors in terms of this basis and using the above calculation shows $\omega =0$.    Also, if $\phi$ in the above calculation is taken to be in $H$, then $\phi \circ \kappa \in H$ as well, so that
$\: _\phi\pi (\s_\omega) = \s_\omega$.

\end{proof}

Most of the proof of Theorem \ref{ergodicversion} will be done by reducing to the case where $N=A$, which we have proven before.  Our main tool for accomplishing this will be using the linear functionals on $N$.  We give the first such result here.

\begin{lemma}\label{image}
Let $x\in N$ and let $\psi \in S(N)$ be a state on $N$, so that $\pi_\psi  = (id\otimes \psi )\pi :N\rightarrow  A$.  Then $\Delta (\pi_\psi (x)) =(id\otimes \pi_\psi  )(\pi (x)) $.  If $x \in N^H$, then $\pi_\psi (x)\in A^H$. In other words, we have $\pi :N^H\rightarrow A^H \otimes N$.
\end{lemma}
\begin{proof}
We have 
\begin{eqnarray*}
\Delta (\pi_\psi (x)) & = & \Delta [(id\otimes \psi )\pi (x)] \\ & = &
(id\otimes id\otimes \psi )( \Delta  \otimes id )(\pi (x)) \\ & = &
(id\otimes id\otimes \psi )( id\otimes \pi )(\pi (x))   = 
(id\otimes \pi_\psi  )(\pi (x)). 
\end{eqnarray*}
Let $\Pi :A\rightarrow B$ be the state stabilizer of $P$.  For $x\in N^H$ we have $(\Pi \otimes id)\pi (x) = 1_B\otimes x$.  Thus
\begin{eqnarray*} 
(\Pi \otimes id)\Delta (\pi_\psi (x)) &  =  & (\Pi \otimes id)(id \otimes \pi_\psi )(\pi (x)) \\ & = &
(id \otimes \pi_\psi )(\Pi \otimes id)(\pi (x)) \\ & = &
 (id \otimes \pi_\psi )(1_B\otimes x) = 1_B\otimes (\pi_\psi (x)).
\end{eqnarray*}
\end{proof}

Finally, we relate the Berezin symbol for $M$ to the one on $\B$.  If $\omega \in M$ and $\psi \in S(N)$, we can compute the Berezin symbol of $(id\otimes \psi )\omega \in \B$.

\begin{lemma}\label{functional1}
Let $\omega \in M\subseteq \B \otimes N$ and $\psi \in S(N)$.  Let $\s :M \rightarrow N^H$ be the Berezin symbol on $M$ and let $\varsigma :\B \rightarrow A^H$ be the Berezin symbol on $\B$.  We have:
\[ \varsigma [(id\otimes \psi )\omega ] = \pi_\psi ( \s_\omega). \]
\end{lemma}
\begin{proof}
\begin{eqnarray}
\varsigma [(id\otimes \psi )\omega ] &  = &
       (tr\otimes id)[(P\otimes 1)\alpha ((id\otimes \psi )\omega )]
       \nonumber \\
& = & (tr\otimes id)[(P\otimes 1) (id\otimes id\otimes \psi )
       ((\alpha \otimes id) \omega )] \nonumber \\
& = & (tr\otimes id\otimes \psi )[(P\otimes 1\otimes 1)
       ((id\otimes \pi ) \omega )] \nonumber \\
& = & \pi_\psi [(tr\otimes id)((P\otimes 1) \omega )] = \pi_\psi ( \s_         \omega ). \nonumber
\end{eqnarray}

\end{proof}

\end{subsection}

\begin{subsection}{The Berezin Adjoint and the Berezin Transform}

In this section we give the formula for $\bs :N^H\rightarrow M$ and develop its properties.  Unlike in the case for $N=A$, there is no natural inner product on $N^H$, so $\bs$ really won't be the adjoint of $\s$ in any sense.  We could just write down the formula for $\bs$ by comparison to the case with $N=A$, but that would not tell us ``why" the image of $\bs$ is contained in $M$.  To motivate the definition then, we first look again at the map $\alpha^\dagger :\B \otimes A\rightarrow \B$.  The following proposition is motivated  from Section 5 of  \cite{Pinzari1}. 
\begin{proposition}\label{idempotent}
Let $E: \B \otimes N\rightarrow \B \otimes N$ be given by $E=(\alpha^\dagger \otimes id)(id\otimes \pi )$.  Then $E$ is an idempotent mapping from $\B \otimes N$ onto $M$.
\end{proposition}
\begin{proof}
By Lemma \ref{lem5.2} we have 
\begin{eqnarray*}
(\alpha \otimes id)E & = & (\alpha \otimes id )(\alpha^\dagger \otimes id)(id\otimes \pi ) \\ & = &
(\alpha^\dagger \otimes id\otimes id)(id\otimes \Delta \otimes id)(id\otimes \pi )  \\ & = &
 (\alpha^\dagger \otimes id\otimes id)(id\otimes id \otimes \pi )(id\otimes \pi ) \\ & = &
(id\otimes \pi )(\alpha^\dagger \otimes id)(id\otimes \pi ) =
(id\otimes \pi )E,
\end{eqnarray*}
so that the range of $E$ is contained in $M$.  Also by Lemma \ref{lem5.1} we have if $\omega \in M$ then $\\
E(\omega ) = (\alpha^\dagger \otimes id)(id\otimes \pi )(\omega ) =
(\alpha^\dagger \otimes id)(\alpha \otimes id )(\omega ) = (id\otimes id)(\omega )= \omega ,\\ $ which shows $E$ maps onto $M$ and is idempotent.

\end{proof}

\begin{remark}
Consider the map $\alpha^\dagger :B(\H )\otimes A\rightarrow B(\H )$.  The proof of Proposition \ref{prop6} essentially shows that for $A$ of Kac-type, then $\alpha^\dagger$ is positive.  In this case, the map $E:\B \otimes N\rightarrow M$ is positive, so it is almost like a conditional expectation.  (The only difference is that $\alpha^\dagger$, and hence $E$, is not completely positive.)  It seems that the  main difficulty in defining Berezin quantization for  general coamenable compact quantum groups is that this map $E$ is not positive.
\end{remark}

We can now define the Berezin Adjoint.  Consider the case when $N=A$ from before.  In this case we have defined  $\bs_a= d_\H \alpha^\dagger (P\otimes a)\in \B$. Taking into account the identification $\alpha :\B \stackrel{\sim}\longrightarrow M$ we have that 
\begin{eqnarray*}
\alpha (\bs_a) & = & d_\H (\alpha \circ \alpha^\dagger )(P\otimes a) \\ & = & d_\H (\alpha^\dagger \otimes id )(id\otimes \Delta )(P\otimes a) \\
& = & d_\H E(P\otimes a) \in M.
\end{eqnarray*}
  This now motivates our definition for $\bs$ in the current setup.

\begin{definition}
Let $\alpha :\B \rightarrow \B \otimes A$ and $\pi :N\rightarrow A\otimes N$ be as above.  The Berezin Adjoint is the mapping $\bs : N^H\rightarrow M$ via
\begin{eqnarray}
 \bs_x & = & d_\H E(P\otimes x) \nonumber \\
& = & d_\H (\alpha^\dagger \otimes id )(id\otimes \pi)(P\otimes x) \nonumber \\
& = & d_\H (id\otimes \h \circ \mathsf{m} \otimes id)((id\otimes \kappa )\alpha (P)\otimes \pi (x)). \nonumber 
\end{eqnarray}
\end{definition}
By the same computation as in the proof of Proposition \ref{prop6}, we see that $\bs$ is a positive unital map from $N^H$ to $M$.  Note that here we must use that $N$ is a C$^*$-algebra, so that $x\geq 0$ is equivalent to $x=yy^*$.

It is not clear from the definition that $\bs : N^H\rightarrow M$ should be  surjective.  Since $\bs$ isn't really an adjoint of $\s$ in any sense, we cannot appeal to the fact that $\s$ is injective as we did before.  One can show by a  calculation similar to the one in the proof of Lemma \ref{lem5.2} that $E(\alpha_\phi (T)\otimes x) = E(T\otimes \: _\phi\pi(x))$ for any $\phi \in \A '$.  Setting $T=P$, it follows that the extension of $\bs$ to $N$ defined by the same formula is surjective. Moreover, taking $\phi \in H$, it follows that $\bs (x) = \bs (\: _\phi\pi (x))$ for any $x\in N$ and any $\phi \in H$.  However, this is not enough.  Since we are using order unit stabilizers, there's no reason to expect that $H$ has a Haar measure, so there is no way by ``averging" over $H$ to get that $\bs :N^H\rightarrow M$ is surjective.

To get around the surjectivity problem (and others), we will relate the Berezin Adjoint on $N^H$ to the one on $A^H$.  If $x\in N^H$ and  $\psi \in S(N)$ we have $\pi_\psi(x) \in A^H$ by Lemma \ref{image}, so we can compute the Berezin adjoint of $ \pi_\psi(x) $.

\begin{lemma}\label{functional2}
Let $\bs :N^H\rightarrow M$ be the Berezin adjoint  on $N^H$, and let $\breve{\varsigma} :A^H\rightarrow \B$ be the Berezin adjoint  on $A^H$.  We have $\breve{\varsigma} (\pi_\psi(x)) = (id\otimes \psi )\bs (x)$ for any $x\in N^H$ and $\psi \in S(N)$.
\end{lemma}
\begin{proof}
We compute: 
\begin{eqnarray*}
\breve{\varsigma} (\pi_\psi(x)) & = & d_\H \alpha^\dagger (P\otimes [(id\otimes \psi)\pi (x)]) \\ & = &  d_\H (id\otimes \psi )(\alpha^\dagger\otimes id)(id\otimes \pi )(P\otimes x) \\ & = &  (id\otimes \psi )\bs (x).
\end{eqnarray*}

\end{proof}

As before, we wish to estimate how far $\s \circ \bs$ and $\bs \circ \s$ are from being the identity operators on $M$ and on $N^H$.  The resulting expression for $\s \circ \bs $ is exactly the same as it was in Proposition \ref{prop7}.   

\begin{proposition}\label{transform}
Consider the composition $\s \circ \bs :N^H\rightarrow N^H$.  Let $x\in N^H$.  We have $(\s \circ \bs)(x) = \: _\hp\pi (x)=(\hp \otimes id )\pi (x)$.

\end{proposition}
\begin{proof}
We have
\begin{eqnarray}
(\s \circ \bs)(x) & = & d_\H \s ((\alpha^\dagger \otimes id)(id\otimes \pi)(P\otimes x)) \nonumber \\
 & = & d_\H (tr\otimes id)[(P\otimes 1) ((\alpha^\dagger \otimes id)(id\otimes \pi)(P\otimes x))] \nonumber \\
 & = & d_\H (tr\otimes \h \otimes  id)[(\alpha (P) \otimes 1) (P\otimes \pi (x))] \nonumber \\
& = & d_\H ( \h \otimes  id)[(\varsigma_P\otimes 1) \pi (x)] \nonumber \\
& = & (\hp \otimes id )\pi (x). \nonumber
\end{eqnarray}

\end{proof}

To obtain a sequence of operator systems $M^n$, we proceed exactly as before.  Namely, we replace the unitary representation  $(u,\H )$ of $A$ by $(u^n, \H^n)$, the cyclic subrepresentation of $\H^{\otimes n}$ generated by $\xi^{\otimes n}$.  Again, we assume that $(u^n, \H^n)$ is irreducible for all $n\in \N$.  By Proposition \ref{nthpower} we have $\s^n:M^n\rightarrow N^H$ for all $n$.  Similarly, we have $\bs^n:N^H\rightarrow M^n$.  The following proof is virtually identical to that of Proposition \ref{prop10}.

\begin{corollary}\label{Berezin2'}
The sequence of operators $(\s^n \circ \bs^n ):N^H\rightarrow N^H$ converges strongly to the identity operator on $N^H$.  
\end{corollary}
\begin{proof}
Let $x\in N^H$.  By the previous proposition we have $(\s^n \circ \bs^n)(x) = (\hpn \otimes id)\pi (x)$.   By Lemma \ref{image} we have $\pi (x)\in A^H\otimes N$.  Since $\hpn$ converges weak$^*$ to $\varepsilon$ on $A^H$, we have $(\hpn \otimes id)$ converges strongly to $(\varepsilon \otimes id)$ on $A^H\otimes N$.  Thus
$\lim_{n\rightarrow \infty}(\s^n \circ \bs^n)(x) = (\varepsilon \otimes id)\pi (x) = x$.

\end{proof}

For $\hbar \in \{ 0\} \cup 1/\N$,  let $N_0 = N^H$ and $N_\hbar = M^{1/\hbar}$.  We have for each $\hbar \neq 0$ a positive unital mapping $ \bs^{1/\hbar}:N_0\rightarrow N_\hbar$.

\begin{corollary}\label{Berezin2}
The collection of operator systems $\{ N_\hbar \: | \: \hbar \in \{ 0\} \cup 1/\N \}$ with the mappings $\bs^{1/\hbar } $ is a strict quantization of $N^H$.
\end{corollary}
\begin{proof}
This follows immediately from Corollary \ref{Berezin2'} and Theorem \ref{deformquant}.
\end{proof}

\end{subsection}

\begin{subsection}{Lip-norms on $N^H$ and $M$}

In the previous sections we constructed a sequence of operator systems $M^n$ together with positive unital mappings $\s^n :M^n \rightarrow N^H$ and $\bs^n :N^H\rightarrow M^n$.  We also showed that the mappings $\bs^n:N^H\rightarrow M^n$ give a strict quantization of $N^H$.  In this section, we will put Lip-norms on $N^H$ and $M^n$ and show that the quantum metric spaces $M^n$ converge to $N^H$ in quantum Gromov-Hausdorff distance.

There is one major difference between our current setup and the case where $N=A$.  In that case we had that the operator systems $\B^n$ and $A^H$ each carry actions of $A$.  The reason for this is because $A^H$ is the set of points of $A$ fixed by $H$ for the right action of $A$ on itself.  However, $A$ also acts on itself on the left.   Since this left action commutes with the right action of $A$ on itself, it decends to a left action of $A$ on $A^H$.

In our current scenario, $N$ has a right action of $A$, but no left action.  Thus after taking $N^H$ as the set of points fixed by $H$, no action of $A$ remains.  The closest result we have is that of Lemma \ref{image}; namely, that $\pi :N^H\rightarrow A^H \otimes N$.  Similarly, we have $(id \otimes \pi ):M\rightarrow \B  \otimes \A \otimes N$.  Of course, $N$ itself is an ergodic right $A$-module algebra, so any regular Lip-norm on $A$ induces a Lip-norm on $N$.  Similarly, since $(id \otimes \pi ):\B \otimes N \rightarrow \B \otimes \A \otimes N$, we obtain a seminorm on $\B \otimes N$ by
\[L(\omega )= \sup_{\psi \in S(\B ), \phi \in S(N)}  L_A[(\psi \otimes id\otimes \phi )(id \otimes \pi )(\omega )] . \]
We can then restrict these seminorms to the order unit subspaces $N^H$ and $M$.  We will see that they are Lip-norms later, but we prefer to demonstrate the computations first.

In the sequel, we will be doing a lot of computations with the induced Lip-norm, and variations of it.  Hence we develop a special notation.
\begin{notation}
Let $A$ and $B$ be order unit spaces, and let $L_A$ be a seminorm on $A$.  We denote by $L_A\otimes ||$ the seminorm on $A\otimes B$ given by:
\[ (L_A\otimes ||)(\omega ) =\sup_{\varphi \in S(B)} L_A((id\otimes \varphi )\omega )  \hspace{1.7in}  \forall \omega \in A\otimes B \]
Similarly, if $L_B$ is a seminorm on $B$, we have a seminorm $||\otimes L_B$ on $A\otimes B$, and so forth.
\end{notation}
For example, the seminorm on $\B \otimes N$ above can now be written as
$L(\omega )=(||\otimes L_A\otimes ||)(id \otimes \pi )(\omega)$, or more simply, $L=(||\otimes L_A\otimes ||)(id \otimes \pi )$.

For convenience  we include a table that rewrites our previous notions in this notation.

\begin{table}[ht]
\centering
\begin{tabular}{l  l  l}
\hline\hline
Concept & Old Formula & New Notation \\ [.5ex]
\hline
$L_A$ left-invariant & $L_A(a)=\sup_{\varphi \in S(A)}L_A(\: _\varphi\Delta (a))$ & $L_A=(||\otimes L_A)\Delta$ \\
$L_A$ right-invariant & $L_A(a)=\sup_{\varphi \in S(A)}L_A(\Delta_\varphi (a))$ & $L_A=(L_A\otimes ||)\Delta$ \\ 
Induced Lip-norm, Left Module & $L_B(b)=\sup_{\psi \in S(B)}L_A(\: _\psi\pi (b))$ & $L_B = (||\otimes L_A)\pi$ \\
$L_B$ invariant & $L_B(b)=\sup_{\varphi \in S(A)}L_B(\pi_\varphi (b))$ & $L_B=(L_B\otimes ||)\pi$ \\ 
Induced Lip-norm, Right Module & $L_N(x)=\sup_{\psi \in S(N)}L_A(\pi_\psi (x))$ & $L_N=(L_A\otimes ||)\pi$ \\
$L_N$ invariant & $L_N(x)=\sup_{\varphi \in S(A)}L_N(\: _\varphi\pi (x))$ & $L_N=(||\otimes L_N)\pi$ \\ [1ex]
\hline\hline
\end{tabular}
\end{table}

Go back to the case of the regular module $N=A$.  We have two potential Lip-norms on $A^H$.  One is the Lip-norm described just above by restricting the Lip-norm of $A$.  The other is the Lip-norm that we originally used, which is obtained from the action of $A$ on $A^H$.  Similarly, we have a Lip-norm on $\B$ obtained by restricting the seminorm above on $\B \otimes A$ as well as the Lip-norm induced from the action of $A$. 
It would be nice to relate these Lip-norms to each other. For now, we use $L_{A^H}$ and $L_\B$ to denote the Lip-norms on $A^H$ and $\B \cong M$ by the restrictions, and we let $\mathcal{L}_{A^H}$ and $\mathcal{L}_\B$ be the Lip-norms induced by their remaining actions of $A$ that were considered in Section 4.

\begin{lemma}
Let $L_A$ be a regular Lip-norm on $A$.  With the notations above, we have $L_{A^H} =\mathcal{L}_{A^H}$ whenever $L_A$ is left-invariant.  We have $L_\B = \mathcal{L}_\B$ whenever $L_A$ is right-invariant.
\end{lemma}
\begin{proof}

Let $a\in A^H$.  By definition we have: $\mathcal{L}_{A^H}(a) = \sup_{\phi \in S(A^H)}L_A(\: _\phi\Delta (a) )$.  Since $\varepsilon \in S(A^H)$, we have $\mathcal{L}_{A^H}(a)\geq L_A(\: _\varepsilon\Delta (a) ) = L_A(a)$.  On the other hand, there is a Hahn-Banach Theorem for states which gives that every state on $A^H$ is the restriction of a state on $A$.  Thus $\mathcal{L}_{A^H}(a) = \sup_{\phi \in S(A)}L_A(\: _\phi\Delta (a) )$.  Finally, $L_A$ is left-invariant says exactly that $L_A(\: _\phi\Delta (a) )\leq L_A(a)$ for all $\phi \in S(A)$.  So for $L_A$ left-invariant we have $\mathcal{L}_{A^H} (a)= L_A(a) = L_{A^H}(a)$.

Let $T\in \B$.  Then $T$ is identified with $\alpha (T)\in M$.  Assuming $L_A$ is right-invariant means that $L_A=(L_A\otimes ||)\Delta$.    We compute

\begin{eqnarray}
L_\B (T) = L_M(\alpha (T) ) & = & (||\otimes L_A \otimes ||)((\alpha \otimes id)\alpha (T)) \nonumber \\
& = & (||\otimes L_A\otimes ||)((id\otimes \Delta)\alpha (T)) \nonumber \\
& = & (||\otimes L_A)(\alpha (T)) =\mathcal{L}_\B (T). \nonumber 
\end{eqnarray}

\end{proof}

The following statement is an analog of Proposition \ref{estimate1}, and the proof remains the same.  Note however that we cannot necessarily take the same sequence $\delta_n$ that we took in Proposition \ref{estimate1}.

\begin{proposition}\label{estimate3}
Let $L$ be any Lip-norm on $N^H$.  Then there exists a sequence of positive real numbers $\{ \delta_n\}_{n\in \N}$ converging to zero such that 
$\| (\s^n \circ \bs^n)(x)-x\| < L(x)\delta_n$ for all $x\in N^H$.
\end{proposition}
\begin{proof}
As before, we show that for any $\delta >0$, there exists an $N_\delta \in \N$ such that $\| (\s^n \circ \bs^n)(x)-x\| < L(x)\delta$ for all $x\in N^H$ and all $n\geq N_\delta$.

Consider the linear operator $T: N^H\rightarrow A^H\otimes N$ by 
$T(x)= \pi (x)-1\otimes x$.  Clearly $T$ is bounded and vanishes on $\C  1_N$.  Since $L$ is a Lip-norm, we have $K=T\{ x\in N^H \: | \: L(x)\leq 1\}$ is totally bounded.
  Also $(\varepsilon \otimes id)T(x) = (\varepsilon \otimes id)(\pi (x)- 1\otimes x)=x-x=0$.  By Lemma \ref{technical}, there exists a function $f\in C(S(A^H))$ such that $f(\varepsilon )\leq \delta/2$ and 
$\| (\mu \otimes id)(\pi (x)-1\otimes x)\| \leq f(\mu )$  for all $\mu \in S(A^H)$.  
By Proposition \ref{transform} we have $(\s^n \circ \bs^n)(x) = (\hpn \otimes id)\pi (x) $.
 Then for $L(x)\leq 1$ we have
$\| (\s^n \circ \bs^n)(x)-x\| = 
\| (\hpn \otimes id)(\pi (x)-1\otimes x)\| \leq f(\hpn )$.  
Since $f$ is continuous in the $w^*$-topology and $f(\varepsilon )\leq \delta/2$, we can find a $N_\delta \in \N$ such that $f(\hpn ) <\delta$ for all $n\geq N_\delta$.   Then for all $x\in N^H$ we have 
$\| (\s^n \circ \bs^n)(x)-x\| \leq L(x) \delta$.  
\end{proof}

Our next goal is to prove the analog of Proposition \ref{estimate2}.  Since there is no action of $A$ on $N^H$ or $M^n$, there is no hope of generalizing the argument of Proposition \ref{estimate2} directly.  However, we can use our results from Lemmas \ref{functional1} and \ref{functional2} to carry the exact same sequence $\theta_n$ from Proposition \ref{estimate2} to the current situation.  This is one of the reasons that we worked out the case where $N=A$ completely first, because ultimately we would need to reduce the general case to that case anyway.  Again, we use the symbols $\varsigma$ and $\breve{\varsigma}$ to denote the Berezin symbol and the Berezin adjoint on $\B$ and on $A^H$.

\begin{lemma}\label{functional3}
Let $\psi \in S(N)$, so that $(id\otimes \psi ): M\rightarrow \B$.  Then we have 
\[ (id\otimes \psi ) (\bs \circ \s )= 
	(\breve{\varsigma} \circ \varsigma ) (id\otimes \psi ) \]
\end{lemma}
\begin{proof}
Let $\omega \in M\subset \B \otimes N$.  By Lemma \ref{functional1} we have 
$\varsigma [ (id\otimes \psi )\omega ] = \: _\psi\pi  (\s_\omega )$.  Now using Lemma \ref{functional2} we have 
$(\breve{\varsigma} \circ \varsigma ) (id\otimes \psi )(\omega ) = 
\breve{\varsigma} (\: _\psi\pi (\s_\omega)) = (id\otimes \psi )(\bs \circ 
\s)(\omega )$. 
\end{proof}

\begin{proposition}\label{estimate4}
Let $\{ \theta_n \}_{n\in \N}$ be the sequence from Proposition \ref{estimate2}.  Then for all $\omega \in M^n$ we have 
$\| (\bs^n \circ \s^n )(\omega )-\omega \| \leq L_n(\omega )\theta_n$.
\end{proposition}
\begin{proof}
Let $\psi \in S(N)$, so that $(id\otimes \psi )\omega \in \B^n$.  Using  the Lemma \ref{functional3} we obtain $\\
 \| (id\otimes \psi ) [(\bs^n \circ \s^n )(\omega ) -\omega] \| =  \| (\breve{\varsigma}^n \circ \varsigma^n ) [(id\otimes \psi )\omega ]-(id\otimes \psi )\omega ] \| \leq 
   L_{\B^n}((id\otimes \psi )\omega )\theta_n .
\\$
Thus we compute:
\begin{eqnarray}
L_{\B^n}((id\otimes \psi )\omega ) & = & (||\otimes L_A)(\alpha^n [(id\otimes \psi)\omega ]) \nonumber \\
& = & 
(||\otimes L_A)[(id\otimes id\otimes \psi )(\alpha^n \otimes id) \omega ]\nonumber \\
& \leq &
(||\otimes L_A\otimes ||)((\alpha^n \otimes id)\omega ) =
L_n(\omega ) \nonumber
\end{eqnarray}
So we have for all $\psi \in S(N)$ that 
$\| (id\otimes \psi ) [(\bs^n \circ \s^n )(\omega ) -\omega] \| \leq L_n(\omega )\theta_n.\\$  Thus 
$\|  (\bs^n \circ \s^n )(\omega ) -\omega \| \leq L_n(\omega )\theta_n$.

\end{proof}

\end{subsection}

\begin{subsection}{Equivariance and Convergence in Quantum Gromov-Hausdorff Distance}

Using Propositions \ref{estimate3} and \ref{estimate4} we can repeat the argument from Section 4.5.  The only subtlety is in using Lemma \ref{equivarLip}, because our spaces $N^H$ and $M^n$ do not have actions of $A$.  However, our Lip-norms on $N^H$ and on $M$ come from restricting the seminorms on $N$ and on $\B \otimes N$, so we can  use equivariance at that level to show our symbols are contractive for the seminorms.

We consider the extensions of $\s$ and $\bs$ to all of $\B \otimes N$ and $N$, respectively.  That is, we consider the operators $\s :\B \otimes N\rightarrow N$ defined by
\[ \s (\omega )=(tr\otimes id)[(P\otimes 1)\omega ] \]
and $\bs :N\rightarrow \B \otimes N$ by
\begin{eqnarray}
 \bs_x & = & d_\H E(P\otimes x) \nonumber \\
& = & d_\H (\alpha^\dagger \otimes id )(id\otimes \pi)(P\otimes x) \nonumber \\
& = & d_\H (id\otimes \h\circ \mathsf{m} \otimes id)((id\otimes \kappa )\alpha (P)\otimes \pi (x)). \nonumber 
\end{eqnarray}
It follows by Proposition \ref{idempotent} that $\bs$ actually maps all of $N$ into (and onto) $M$, but we will still view it as a map into $\B \otimes N$.

First off, we have an ergodic right action $\pi :N\rightarrow A\otimes N$ of $A$ on $N$. Thus given any left-invariant Lip-norm $L_A$ on $A$, we obtain by Proposition \ref{li1} an invariant Lip-norm $L_N$ on $N$ by $L_N =(L_A\otimes ||)\pi$.  By Proposition \ref{rief1} we can see that the restriction of a Lip-norm to an order unit subspace is a Lip-norm on the subspace.  So the restriction of $L_N$ to $N^H$ is a Lip-norm on $N^H$, as desired.

Almost the same argument works for $M$ as well.  It's easy to see that the action $\pi :N\rightarrow A\otimes N$ induces a right action $\Pi :\B \otimes N\rightarrow A\otimes \B \otimes N$, where $\Pi =(\vartheta \otimes id)(id\otimes \pi )$, where $\vartheta :\B \otimes A \rightarrow A\otimes \B$ is the flip map.  Also, the seminorm defined on $\B \otimes N$ above as $L_{\B \otimes N}=(||\otimes L_A \otimes ||)(id\otimes \pi)$ is the same as the usual seminorm induced on $\B \otimes N$ from $L_A$ using the right action $\Pi$.  However, this seminorm is not a Lip-norm on $\B \otimes N$ because $\Pi$ is not ergodic.  The fixed set of $\Pi$ is $\B \otimes 1$.  However, there are no nontrivial fixed elements in  $M\subseteq \B \otimes A$, and it's routine to check that the restriction of $L_{\B \otimes N}$ to $M$ is a Lip-norm.

Again, consider the extensions of the  Berezin symbol and Berezin adjoint $\s :\B \otimes N\rightarrow N$ and  $\bs :N\rightarrow \B \otimes N$.  We must show that $\s$ and $\bs$ are equivariant for the right actions $\pi$ on $N$ and $\Pi$ on $\B \otimes N$.

\begin{lemma}
The maps $\s :\B \otimes N\rightarrow N$ and  $\bs :N\rightarrow \B \otimes N$ are equivariant for the right actions $\pi$ on $N$ and $\Pi$ on $\B \otimes N$.  That is, we have:
\begin{eqnarray}
(id\otimes \s )\Pi (\omega ) & = & \pi (\s_\omega ) \hspace{1.2in} \forall \omega \in \B \otimes N \nonumber \\
(id\otimes \bs )\pi (x) & = & \Pi (\bs (x)) \hspace{1.2in} \forall x\in N. \nonumber
\end{eqnarray}
\end{lemma}
\begin{proof}
For the equivariance of $\s$, we recall that $\s_\omega =(tr\otimes id)((P\otimes 1)\omega )$.  We have:
\begin{eqnarray}
(id\otimes \s )\Pi (\omega ) & = & (id \otimes \s )(\vartheta \otimes id)
	(id\otimes \pi )(\omega ) \nonumber \\
& = & (id \otimes tr \otimes id)((1\otimes P\otimes 1)[(\vartheta \otimes id)
	(id\otimes \pi )(\omega )]) \nonumber \\
& = & (tr \otimes id \otimes id)((P\otimes 1\otimes 1)
	(id\otimes \pi )(\omega )) =\pi (\s_\omega ).\nonumber 
\end{eqnarray}

For the equivariance of $\bs$, we  recall that $\bs (x)=d_\H (\alpha^\dagger \otimes id)(P\otimes \pi (x))$.  Thus we have:
\begin{eqnarray}
(id\otimes \bs )\pi (x) & = & 
	d_\H (id\otimes \alpha^\dagger \otimes id)(id\otimes id\otimes \pi )
	((1\otimes P \otimes 1)\pi (x)_{13}) \nonumber \\
& = & d_\H (\vartheta \otimes id)(\alpha^\dagger \otimes id \otimes id)
	(id\otimes id\otimes \pi ) ( P \otimes \pi (x)) \nonumber \\
& = & d_\H (\vartheta \otimes id)(id\otimes \pi )(\alpha^\dagger \otimes id)
	 ( P \otimes \pi (x)) \nonumber \\
& = & d_\H \Pi (\alpha^\dagger \otimes id) ( P \otimes \pi (x)) =\Pi (\bs (x)). \nonumber 
\end{eqnarray}

\end{proof}

\begin{corollary}
Let $\omega \in \B \otimes N$ and $x\in N$.  We have $L_N(\s_\omega )\leq L_{\B \otimes N}(\omega )$ and $L_{\B \otimes N}(\bs (x))\leq L_N(x)$.  In particular, $\s$ and $\bs$ are contractive for the Lip-norms on $N^H$ and $M$.
\end{corollary}
\begin{proof}
The proof of Lemma \ref{equivarLip} does not use the ergodicity assumption, except in stating that the induced seminorms are Lip-norms.  This gives the first statement.  The second follows since the Lip-norms on $N^H$ and on $M$ are the restrictions of the induced seminorms on $N$ and on $\B \otimes N$.
\end{proof}

Using this corollary and Propositions \ref{estimate3} and \ref{estimate4}, it is straightforward to repeat the steps of Section 4.5.  Hence we have our main theorem.

\begin{theorem}\label{main2}
Let $A$ be a coamenable compact quantum group of Kac-type, and let $(u,\H )$ be an irreducible unitary representation of $A$ with primitive-like vector $\xi \in \H$.  Let $N$ be a C$^*$-algebra, and let $\pi :N\rightarrow A\otimes N$ be an ergodic right action of $A$ on $N$.  Finally, let $L_A$ be an invariant regular Lip-norm on $A$.    

Let $P\in B(\H )$ be the rank one projection corresponding to $\xi$, and let $H\subseteq S(A)$ be the state stabilizer of $P$ under the action $\alpha (T)=u(T\otimes 1)u^*$.  Let $N^H\subseteq N$ be the set of $H$-invariant elements of $N$, and consider the sequence of finite dimensional spaces $M^n=\{ \omega \in \B^n \otimes N \: | \: (\alpha^n\otimes id)\omega =(id\otimes \pi )\omega \}$.  Let $L$ and $L_n$ be the Lip-norms induced on $N^H$ and on $M^n$ by $L_A$.  

Then the sequence of quantum metric spaces $(M^n, L_n)$ converge to $(N^H, L)$ in quantum Gromov-Hausdorff distance as $n\longrightarrow \infty$.
\end{theorem}

\end{subsection}

  \end{section}

\begin{section}{Examples}

In Section 3 we described in detail the commutative quantum group $C(G)$ for $G$ a classical compact group, and we computed the quantum group stabilizer and state stabilizer for any vector in any unitary representation of $C(G)$.  We did the same thing for the cocommutative quantum group $C^*(\Gamma )$ for $\Gamma$ a discrete group.  For $G$ a semisimple Lie group, the Berezin Quantization coming from the right regular representation of $C(G)$ is the usual Berezin Quantization of a coadjoint orbit.  There are no interesting Berezin Quantizations that come from $C^*(\Gamma )$, because all of the irreducible unitary representations of $C^*(\Gamma )$ are one dimensional.  

In this chapter we repeat some of this program for other compact quantum groups.  In the first subsection we describe some of the other well-known quantum groups from the literature.  In Section 6.2 we describe briefly the growing literature on ergodic actions of compact quantum groups.  In Section 6.3 we compute some of the quantum group stabilizers and state stabilizers for representations of these quantum groups.  Finally, in Section 6.4 we construct some  Berezin Quantizations from ergodic actions.

\begin{subsection}{Examples of Quantum Groups}
\begin{subsubsection}{Theta-Deformations of Quantum Groups}

In \cite{Rieffel4}, Rieffel gives a general proceedure for deforming a C$^*$-algebra using an action of $\R^d$ and a skew-symmetric form on $\R^d$.    Let $A$ be a compact quantum group, and let $\Pi : A\rightarrow C(\T )$ be a toral subgroup of $A$. Let $\mathfrak{t}$ be the Lie algebra of $\mathbb{T}$.  Then we have actions of the abelian group $\mathfrak{t}$ on $A$ by either left or right translation.  We can use the proceedure of \cite{Rieffel4} to deform $A$ as a C$^*$-algebra.  A natural question is whether we can deform by some particular action so that the deformed C$^*$-algebra is actually a compact quantum group.  This question is addressed for the case $A=C(G)$ in \cite{Rieffel5}, and is discussed for  general $A$ in \cite{Wang3}.

We sketch the results here.  Let $A$ is a compact quantum group and let $\Pi :A\rightarrow C(\T )$ be
a toral quantum subgroup of $A$.  Let $\mathfrak{t}$ be the Lie algebra of $\T$. Then there are two natural left actions of the underlying abelian group of $\mathfrak{t}$ on the underlying C$^*$-algebra of $A$ given by left and right translation.  Let $s$ and  $t$ be elements of 
$ \mathfrak{t}$. Explicitly, we have  two commuting actions $\lambda :\mathfrak{t} \rightarrow \text{C$^*$-Aut}(A)$ and $\rho :\mathfrak{t} \rightarrow \text{C$^*$-Aut}(A)$ given by
\[ \lambda_s =(\delta_{\exp (-s)}\Pi \otimes id)\Delta , \hspace{.5in}
	\rho_t =(id\otimes \delta_{\exp (t)}\Pi)\Delta ,\]
where $\exp :\mathfrak{t}\rightarrow \T$ is the exponential map, and $\delta_{\exp (t)}\in C(\T )'$ is evaluation at $\exp ( t)$.  Since these actions commute, we have an action $\alpha :\mathfrak{t}\oplus \mathfrak{t} \rightarrow \text{C$^*$-Aut}(A)$ of $\mathfrak{t}\oplus \mathfrak{t}$ on the underlying C$^*$-algebra of $A$ given by
\[ \alpha_{(s,t)}=\lambda_s\rho_t .\]

 If \( \theta  \) is a skew-symmetric bilinear form
on \( \mathfrak {t}, \) we can consider the skew-symmetric form \( J\equiv \theta \oplus (-\theta ) \)
on \( \mathfrak {t}\oplus \mathfrak {t} \) and perform Rieffel's
deformation procedure of \( A \) with respect to the action of
\( \mathfrak {t}\oplus \mathfrak {t}. \) That is, we let $A^\infty \subseteq A$ be the dense subalgebra of smooth functions for the action. We then consider the product $\times_{\hbar J}$ on the algebra  $A^\infty$ given by the oscillatory integral
\[  (a\times _{\hbar J}b)(x)=\frac{1}{(\pi \hbar )^{2d}}\int _{\mathfrak {t}^{4}}\alpha_{(s,t)}(a)\alpha_{(u,v)}(b)e^{\frac{2i}{\hbar }(\theta s\cdot u-\theta t\cdot v)}dsdtdudv. \]
Here, we set \( d=dim(\mathfrak {t}). \)  We denote $A^\infty$ with this deformed product by $A^\infty_{\hbar J}$.  Then Rieffel's deformation of $A$ by the action $\alpha$ with respect to the skew-symmetric form $J$ is defined to be the universal enveloping C$^*$-algebra of $A^\infty_{\hbar J}$, and it is denoted by $A_{\hbar J}$.

In \cite{Wang3}, Wang showed that in fact the
resulting algebras \( A_{\hbar J} \) are quantum groups. 
He first shows that the polynomial subalgebra $\A$ is always contained in $A^\infty$.  He then defines
 the polynomial subalgebra \( \A_{\hbar J} \) to be equal to 
 \( \A \) as a linear space, but with the multiplication defined above.  Wang shows that  the undeformed coproduct, counit, and coinverse of $\A$ also give $\A_{\hbar J}$ the structure of a Hopf algebra, and the extension of the coproduct to $A_{\hbar J}$ makes it into a quantum group.

We note that since $\A_{\hbar J}=\A$ as a coalgebra, we have that $A_{\hbar J}$ is of Kac-type if and only if $A$ is of Kac-type.  
If $u:\H \rightarrow \H \otimes A$ is any finite dimensional unitary representation of $A$ (so that the matrix coefficients of $u$ lie in $\A$), then $u$ will automatically be a representation of $A_{\hbar J}$ also, since $\A_{\hbar J}$ has the same coalgebra structure as $\A$.  In Proposition 3.8 of \cite{Wang3}, Wang  checks that $u$ is still unitary as a representation of $A_{\hbar J}$, so the representation theory of $A_{\hbar J}$ is identical to that of $A$, except that the tensor product of representations gets twisted.  It follows from  Proposition 6.1 of \cite{Banica1} that $A_{\hbar J}$ is  coamenable if and only if $A$ is coamenable.

The description just given is that of Rieffel's deformation for arbitrary actions of $\R^d$ on a C$^*$-algebra.  Since our action comes from the action of a torus, we can actually write down a simplier expression for deformed product, as is described at the end of Chapter 2 of \cite{Rieffel4}. 
Consider the action $\alpha$ above not as an action of $\mathfrak{t}\oplus \mathfrak{t}$, but as an action of $\T \oplus \T$ on $A$.  We  write $A$ as the direct sum of its weight spaces (i.e. isotypic components) for this action of $\T^2$.   We have as a graded algebra $A^\text{alg} =\bigoplus_{n\in \Z^{2d}} A_n $, where $d=dim(\T )$.  Then the deformed product on $A^\text{alg}_{\hbar J}$ is given by 
\[ (a\times_{\hbar J} b)_n=\Sigma_{m\in \Z^d} \: a_mb_{n-m} e_\hbar ( J m\cdot n),\]
where $e_\hbar :\R \rightarrow \C$ is the exponential map $e_\hbar (t)=\exp (2\pi i\hbar t)$.  In the following sections, it will be convenient to use the oscillatory integral formalism for a few of our abstract arguments, but to avoid having to define (or compute) oscillatory integrals, we will use this graded expression for the deformed product in our computations.

\end{subsubsection}

\begin{subsubsection}{q-Deformations of Classical Lie Groups}

Again, let $G$ be a Lie group, and let  $\mathfrak{g}$ be its Lie algebra.  Then the universal enveloping algebra $U\mathfrak{g}$ has a natural Hopf algebra structure.  In the literature, a $q$-deformation of $G$ is often defined to be a deformation of $U\mathfrak{g}$ in the category of Hopf algebras.  This picture is in some sense  dual to the picture we have adopted, where one wants a deformation of the algebra of continuous functions $C(G)$.  Again, Proposition 6.1 of \cite{Banica1} shows that $q$-deformations of Lie groups are always coamenable, but they are in general not of Kac-type.  For this reason, we will not give a general construction for deforming a classical semisimple Lie group.  We will just give Woronowicz's description of quantum $SU(2)$.  The reader interested in $q$-deformations of more general Lie groups may consult \cite{Klimyk}.

We turn to Woronowicz's description of $C(SU_q(2))$.  See \cite{Woro2} for more details.  Let $q\in [-1,1]$.  We let $C(SU_q(2))$ be the universal C$^*$-algebra generated by two elements $a,c$ with the relations:
\begin{eqnarray}
a^*a +c^*c & = & 1 \nonumber \\
aa^*+q^2cc^* & = & 1 \nonumber \\
c^*c & = & cc^* \nonumber \\
ac & = & qca \nonumber \\
ac^* & = & qc^*a. \nonumber
\end{eqnarray}
Note that $SU(2) = \left\{
\left( \begin{array}{cc} a & -c^* \\ c & a^* \end{array} \right) \: | \: a^*a+c^*c=1 \right\}$.  So if we let $q=1$ in the definition of $C(SU_q(2))$ above, we get exactly the relations for the coordinate functions of $SU(2)$.  Note also that if we consider the element $U\in M_2\otimes C(SU_q(2))$ given by
\[  U= \left( \begin{array}{cc} a & -qc^* \\ c & a^* \end{array} \right) , \]
then the defining relations for $C(SU_q(2))$ say exactly that $U$ is a unitary element of   $M_2\otimes C(SU_q(2))$.

We define the coproduct on $C(SU_q(2))$  to make $U$ a unitary representation of $C(SU_q(2))$; that is, we define $\Delta$ by the equation $(id \otimes \Delta ) U=U_{(12)} U_{(13)}$.  In terms of the generators we have
\begin{eqnarray}
\Delta (a) & = & a\otimes a -qc^*\otimes c \nonumber \\
\Delta (c) & = & c\otimes a +a^* \otimes c. \nonumber
\end{eqnarray}
Woronowicz shows that for $q\neq 0$ this coproduct makes $C(SU_q(2))$ into a compact quantum group.

We will not consider the general irreducible representations of $C(SU_q(2))$ in the sequel, but we note that they are deformations of the representations of $SU(2)$.  Note that the matrix $U \in M_2\otimes C(SU_q(2))$ above is a unitary representation of $C(SU_q(2))$ that deforms the standard representation of $SU(2)$.

Finally, we shall need a result due to Woronowicz about the structure of $C(SU_q(2))$ as a C$^*$-algebra, namely that  all of the C$^*$-algebras $C(SU_q(2))$ are isomorphic for $q\in (-1,1)$, even for $q=0$.  This will make our computation of stabilizers for actions of $C(SU_q(2))$ much easier, because Summary \ref{summary} explains how the quantum group stabilizer and state stabilizer of an element in a unitary representation of a quantum group can be computed from the representation theory of the underlying C$^*$-algebra of the quantum group.

\end{subsubsection}

\begin{subsubsection}{The Universal Compact Quantum Groups}

The unitary groups $U(n)$ are more or less universal for compact Lie groups.  That is, every  compact Lie group is a subgroup of $U(n)$ for some $n\in \N$.  One is then led to ask whether there is a list of compact quantum groups with the property that any compact quantum group is a quantum subgroup of exactly one of these universal compact quantum groups.  In \cite{VanDaele} Van Daele and Wang defined two sets of universal quantum groups, $A_u(Q)$ and $B_u(Q)$. They showed that every compact quantum matrix group is a quantum subgroup of some $A_u(Q)$, and every compact quantum matrix group with a self-conjugate fundamental representation is a quantum subgroup of some $B_u(Q)$.  The question of uniqueness then comes down to understanding the structure of the quantum groups $A_u(Q)$ and $B_u(Q)$, which is worked out in \cite{Wang1}.

We recall the definitions of $A_u(Q)$ and $B_u(Q)$.  Let $Q\in M_n(\C )$ be a positive and invertible matrix.  Both $A_u(Q)$ and $B_u(Q)$ are defined in terms of $n^2$ generators $u_{ij}$.  Denote by $u$ the formal matrix with entries $u_{ij}$.  Then the defining relations for $A_u(Q)$ and $B_u(Q)$ are:
\[ \begin{array}{ccc}
A_u(Q): & u^*u=I_n=uu^*, & u^tQ\bar{u}Q^{-1}=I_n=Q\bar{u}Q^{-1}u^t; \\
B_u(Q): & u^*u=I_n=uu^*, & u^tQ{u}Q^{-1}=I_n=Q{u}Q^{-1}u^t,
\end{array} \]
where $u^t$ is the matrix $(u_{ji})$, $\bar{u}=(u^*_{ij})$, and $u^*=(u^*_{ji})$.  The coproduct in either case is given by $\Delta (u_{ij})=\Sigma_k u_{ik}\otimes u_{kj}$.  

Both $A_u(Q)$ and $B_u(Q)$ are Kac-type if and only if $Q=I_n$.  For $n=1$ both $A_u(Q)$ and $B_u(Q)$ are isomorphic to the circle $C(S^1)$.  For $n=2$ we have $B_u(Q)$ is isomorphic to $C(SU_q(2))$ for an appropriate value of $q$ \cite{Banica5}.  Other than in these degenerate cases, $A_u(Q)$ and $B_u(Q)$ are never coamenable \cite{Banica5}.

Banica computes the fusion semiring for $A_u(Q)$ for $Q\geq 0$  in \cite{Banica2}.  In particular, he shows that if $u$ is the fundamental matrix representation of $A_u(Q)$ above, then $u^{\otimes k}$ is irreducible for all $k\in \N$.  In particular, any nonzero vector $\xi \in \C^n$ is a primitive-like vector for $u$.  

\end{subsubsection}

\end{subsection}

\begin{subsection}{Ergodic Actions of Compact Quantum Groups}

Let $G$ be a compact group, and let $\alpha :G\rightarrow Aut(B)$ be an ergodic action of $G$ on a C$^*$-algebra $B$ by $*$-homomorphisms.  By ergodicity, we see that averaging over the action of $G$ gives a unique $G$-invariant state on $B$.  In \cite{Hoegh} they show first that for $\mu \in \hat{G}$, the $\mu$-isotypic component of $B$ is finite dimensional, and has multiplicity at most $dim(\mu )$.  Secondly, they show that the unique $G$-invariant state on $B$ is a trace.  In particular, we see that any operator algebra that admits an ergodic action of a compact group must be nuclear and finite.

In \cite{Boca}, Boca generalizes the first part of the HLS theorem to compact quantum groups.  Let $\alpha :B\rightarrow B\otimes A$ be an ergodic $A$-module algebra, and let $\nu \in \hat{A}$.  Then Boca shows that the multiplicity of $\nu$ in $B$ is bounded by the so-called quantum multiplicity  $qdim(\nu )$.  See \cite{Boca} for definitions and details. As before, we see that if an operator algebra admits an ergodic action  by a compact quantum group, it must be nuclear.

As in the group case, given an ergodic action $\alpha :B\rightarrow B\otimes A$ we obtain a unique $A$-invariant state $\phi$ on $B$ by
 $\phi (b)1_B= \alpha_\h (b)$.  However, we do not have that $\phi$ is generally a trace on $B$.  For example, let $Q\in M_n(\C )$ be positive, and  take  the universal algebra $A_u(Q)$  described above. Then we have an ergodic action of $A_u(Q)$ on $M_n(\C )$ by $\alpha (T)=u(T\otimes 1)u^*$.   In \cite{Wang2} Wang shows that the $A$-invariant state on $M_n(\C )$ is given by $\phi_Q(T)=tr( Q^tT)$, which is not a trace.  Wang goes on to consider the inductive limit of the module algebras $M_n(\C )^{\otimes k}$ and produces for  $\lambda \in (0,1)$ an  example of an ergodic action of some $A_u(Q)$ on the hyperfinite type III$_\lambda$ factor.

There is a duality for ergodic actions of compact quantum groups resembling the Tannaka-Krein duality for quantum groups \cite{Pinzari2}.  More or less, to an ergodic action $\alpha :B\rightarrow B\otimes A$ one associates the functor from $L_B:\text{Rep}(A)\rightarrow \text{Hil}$ which is given on irreducible representations by $L_B(\nu )=\overline{B_\nu}$, the conjugate Hilbert space of $B_\nu$,  and extended linearly for direct sums.  Then $L_B$ is a $*$-preserving quasitensor functor, and \cite{Pinzari2} classifies all such functors arising from ergodic actions and gives an inverse association.  Consequently, in the literature one will often find constructions of ergodic actions of quantum groups that actually just construct this kind of functor.

\begin{subsubsection}{Quantum Coset Spaces}

Let $G$ be an ordinary group.  The most obvious example of an ergodic $G$-action is the action of $G$ on $C(G/H)$ by left translation, for a subgroup $H<G$.  Similarly, let $A$ be a compact quantum group, and let $\Pi :A\rightarrow B$ be a quantum subgroup.  Then we have a right action of $B$ on $A$ via $(\Pi \otimes id)\Delta :A\rightarrow B\otimes A$, and we can consider the C$^*$-subalgebra of $B$-invariant elements of $A$
\[ A^\Pi =\{ a\in A \: | \: (\Pi \otimes id)\Delta (a)=1_B\otimes a\} . \]
There is also a left action of $A$ on itself by the coproduct, which commutes with the right action by $B$, so it descends to a left action of $A$ on $A^\Pi$.  That is, we have $\Delta |_{A^\Pi}:A^\Pi\rightarrow A^\Pi\otimes A$.  

Similarly, we can switch the sides that $A$ and $B$ act on.  Namely, consider the space
\[ ^\Pi \! A=\{ a\in A \: | \: (id \otimes \Pi )\Delta(a)=a\otimes 1_B \}. \]
Then we have a right action of $A$ on $^\Pi \! A$ by the restriction of the coproduct.  That is, we have $\Delta |_{^\Pi \! A}:$ $  ^\Pi\! A\rightarrow A\otimes $ $ ^\Pi\! A$.

It is shown in Section 6 of \cite{Wang2} that these actions are ergodic.

\end{subsubsection}

\begin{subsubsection}{Induced Actions}

The construction of induced representations of compact quantum groups is worked out in \cite{Pinzari1}.  The following discussion for induced actions is similar, but it is easier because we do not need to consider the $L^2$-continuity of everything.

The following construction parallels our construction of $M$ from Section 5. For motivation, let $G$ be a compact group and $H<G$ be a  subgroup.  Let $\alpha :H\rightarrow \text{Aut}(N)$ be a right action of $H$ on a C$^*$-algebra $N$.  Then we have an induced action of $G$ on the C$^*$-algebra $\{ f:G\rightarrow N\: | \: f(xh)= \alpha_h[f(x)] , \: \forall h\in H\}$ by left translation.  The more natural left translation of $G$ on the above set (the left translation without the inverse inserted) is a right action of $G$.  

Now let $A$ be a compact quantum group, and $\Pi :A\rightarrow B$ be a quantum subgroup.  Let $\alpha :N\rightarrow B\otimes N$ be a right action of $B$ on a C$^*$-algebra $N$.  Then consider the C$^*$-algebra
\[ \text{Ind}_A(N)= \{ \zeta \in A\otimes N \: | \: 
	(id\otimes \Pi \otimes id)(\Delta \otimes id)\zeta =
	(id\otimes \alpha )\zeta \} .\]
Consider $(id\otimes \Pi )\Delta :A \rightarrow A\otimes B$ as a left action of $B$ on $A$.   
Essentially the same computation as in Proposition \ref{idempotent} shows that $[((id\otimes \Pi )\Delta )^\dagger \otimes id](id\otimes \alpha )$ is an idempotent map from $A\otimes N$ onto $\text{Ind}_A(N)$, so this space is nonempty.  It is easy to see that it is a C$^*$-subalgebra of $A\otimes N$.  We will now check that $(\Delta \otimes id )\text{Ind}_A(N)\subseteq A\otimes \text{Ind}_A(N)$, which says that $\text{Ind}_A(N)$ is a right $A$-module algebra. Let $\zeta \in \text{Ind}_A(N)$.  We compute:
\begin{eqnarray}
(id\otimes id \otimes \Pi \otimes id)(id\otimes \Delta \otimes id)(\Delta \otimes id)\zeta & = &
 (id\otimes id \otimes \Pi \otimes id)(\Delta  \otimes id \otimes id)(\Delta \otimes id)\zeta \nonumber \\
& = & (\Delta  \otimes id \otimes id)(id \otimes \Pi \otimes id)(\Delta \otimes id)\zeta \nonumber \\
& = & (\Delta  \otimes id \otimes id)(id\otimes \alpha )\zeta \nonumber \\
& = & (id\otimes id\otimes \alpha )(\Delta  \otimes id)\zeta . \nonumber
\end{eqnarray}
Thus $(\Delta \otimes id)\zeta \in A\otimes \text{Ind}_A(N)$, and $\text{Ind}_A(N)$ is a right $A$-module algebra.

\begin{proposition}
Suppose that $\zeta \in \text{Ind}_A(N)$ is invariant under the action $(\Delta \otimes id)$.  Then $\zeta$ is of the form $1_A\otimes \nu$, where $\nu \in N$ is invariant under the $B$-action $\alpha$.  

In particular, if $\alpha :N\rightarrow B\otimes N$ is an ergodic action of $B$, then $(\Delta \otimes id):\text{Ind}_A(N)\rightarrow A\otimes \text{Ind}_A(N)$ is an ergodic action of $A$.
\end{proposition}
\begin{proof}
Suppose $\zeta \in \text{Ind}_A(N)$ is invariant, so that $(\Delta \otimes id)\zeta =1_A\otimes \zeta$.  Then by the definition of $\text{Ind}_A(N)$ we have
\begin{eqnarray}
(id\otimes \alpha )\zeta 
& = &
(id\otimes \Pi \otimes id)(\Delta \otimes id )\zeta \nonumber \\
& = &
(id\otimes \Pi \otimes id)(1_A\otimes \zeta ) \nonumber \\
& = &
1_A\otimes (\Pi \otimes id)\zeta .\nonumber 
\end{eqnarray}
Looking at the first component of the tensor product gives that $\zeta =1_A\otimes \nu$ for some $\nu \in N$.  Plugging this expression back in for $\zeta$ and looking at the second components then gives $\alpha (\nu )=1_B\otimes \nu$.  That is, that $\nu$ is fixed by the $B$-action $\alpha :N\rightarrow B\otimes N$.  This establishes the first claim, and the second follows immediately from the first.
\end{proof} 

One important special case is when the quantum subgroup $B$ is an ordinary torus.  A torus of dimension at least two admits an ergodic action on a noncommutative torus.  However, this action does not yield interesting Berezin quantizations because all the irreducible unitary representations of the torus are one dimensional.  However, the above proposition allows us to use this action on the noncommutative torus to induce interesting ergodic actions for any compact quantum group which contains a torus.  We work out this action explicitly for later reference.  We will see that in a sense, inducing an action commutes with Rieffel's deformation by actions of $\R^d$.

Let $\T^d$ be a torus.  Then by the Fourier Transform we have $C(\T^d)\cong C^*(\Z^d)$.  Let $\theta \in M_d(\R )$ be a skew-symmetric matrix, so that we can define the skew bicharacter $c_\theta (m,n)=\exp (2\pi i \langle \theta m,n\rangle)$ on $\Z^d$.  The noncommutative torus $C^*(\Z^d, c_\theta )$ is defined to be the universal C$^*$-algebra generated by finitely supported functions on $\Z^d$ with the usual star $f^*(n)=\overline{f(-n)}$ and with the twisted convolution product
\[ f*_\theta g (n) =\Sigma_m f(m)g(n-m)c_\theta (m, n-m). \]
There is an action of the ordinary torus on the noncommutative torus given by $(t\cdot f)(n)=\exp (2\pi i\langle t,n\rangle) f(n)$, where $t\in \T^d =(\R / \Z )^d$.  Note that this action is essentially independent of $\theta$.  It is easily seen to be ergodic.

Let $A$ be a compact quantum group, and let $\Pi :A\rightarrow C(\T^d)$ be a toral quantum subgroup of $A$.  For any $a\in A$, we can view $(id\otimes \Pi )\Delta a$ as a function from $\T^d $ to $A$.  We will denote the value of this function at a point $s\in \T^d$ by $\Delta_s (a)$.  We wish to understand the ergodic action of $A$ induced from the action of $\T^d$ on $C^*(\Z^d, c_\theta )$.

We view $A\otimes C^*(\Z^d, c_\theta)$ as the C$^*$-algebra generated by finitely supported $A$-valued functions on $\Z^d$ with the same twisted convolution product above.  We denote the value of a function $\zeta \in A\otimes C^*(\Z^d, c_\theta)$ at $n\in \Z^d$ by $\zeta_n$.  Then we have
\[ \text{Ind}_A(C^*(\Z^d, c_\theta)) = 
	\{ \zeta \in  A\otimes C^*(\Z^d, c_\theta) \: | \: \Delta_s(\zeta_n) =
	\exp (2\pi i\langle s,n\rangle)\zeta_n , \: \forall s\in \T^d, \: \forall n\in Z^d \} . \]
We see that $\text{Ind}_A(C^*(\Z^d, c_\theta))$ is precisely the set of functions $\zeta \in A\otimes C^*(\Z^d, c_\theta)$ such that the value of $\zeta$ at $n\in \Z^d$ lies in the $n$-th isotypic component of $A$, when $A$ is viewed as a $\T^d$-module via the action $(id\otimes \Pi )\Delta$. Denote this isotypic component by $A_n$. We have that as a linear space  $\text{Ind}_A(C^*(\Z^d, c_\theta))$ is a completion of $\bigoplus_{n\in \Z^d}A_n$.  Of course, $A$ itself is also a completion of  $\bigoplus_{n\in \Z^d}A_n$ as a linear space (and as a graded algebra).

  If $\zeta , \chi \in \text{Ind}_A(C^*(\Z^d, c_\theta))$, then their product is given by 
\[ (\zeta *\chi )(n)=\Sigma_m \: \zeta_m  \chi_{n-m} c_\theta (m,n-m). \]
Let $\t \cong \R^d$ be the Lie algebra of $\T^d$.  We have an action of the underlying abelian group of  $\t$ on $A$ given by $x \cdot (a)=\Delta_{\exp (x)}(a)$.  We can then use the skew-symmetric form $\theta$ and the action of $\t$ on $A$ to produce a deformed C$^*$-algebra $A_\theta$ according to Rieffel's deformation \cite{Rieffel4}.  Using the graded formula for the   product on $A_\theta$ given in Section 6.1.1, one sees that the product on $A_\theta$ is precisely the product given  above for $\text{Ind}_A(C^*(\Z^d, c_\theta))$, so that we have $\text{Ind}_A(C^*(\Z^d, c_\theta))=A_\theta$ as a C$^*$-algebra.  Our work above shows that we have an ergodic action $\Delta :A_\theta \rightarrow A\otimes A_\theta$.

Of course, the noncommutative torus $C^*(\Z^d, c_\theta )$ is the Rieffel deformation of $C(\T^d)$ by the action of $\t$ with the skew form $\theta$.  This example essentially says that the process of inducing the $\T^d$-module $C(\T^d)$ to an $A$-module commutes with Rieffel deformations.

\end{subsubsection}

\end{subsection}

\begin{subsection}{Examples of Stabilizers}
In this section we will compute some explicit examples of quantum group stabilizers and state stabilizers for actions of compact quantum groups.  Recall that our definitions and results in Section 3 did not require $A$ to be coamenable or of Kac-type.

\begin{subsubsection}{Stabilizers for $C(SU_q(2))$}

We begin by computing the quantum group stabilizer and state stabilizer for the two simplest unitary representations of $C(SU_q(2))$.  Namely, we compute them for the standard representation of $C(SU_q(2))$ on $\C^2$, and for the representation on $M_2(\C )$ by conjugation.

Consider the representation of the ordinary $SU(2)$ on $\C^2$.  The stabilizer of any vector in $\C^2$ is trivial.  For simplicity, take the standard basis vector $e_1 \in \C^2$ .  We have:
\[ \left( \begin{array}{cc} a & -\bar{c} \\ c & \bar{a} \end{array} \right)
	\left( \begin{array}{c} 1 \\ 0 \end{array} \right) =
	\left( \begin{array}{c} a \\ c \end{array} \right).  \]
So we see that the only element of $SU(2)$ that stabilizes $e_1$ is the identity.  

Recall the defining relations for $C(SU_q(2))$.  It is the universal C$^*$-algebra generated by two elements $a,c$ with the relations: 
\begin{eqnarray}
a^*a +c^*c & = & 1 \nonumber \\
aa^*+q^2cc^* & = & 1 \nonumber \\
c^*c & = & cc^* \nonumber \\
ac & = & qca \nonumber \\
ac^* & = & qc^*a. \nonumber
\end{eqnarray}

\begin{proposition}
Let $q\in [-1,1]$ be given.  Let $U: \C^2 \rightarrow \C^2 \otimes C(SU_q(2))$ be the standard representation of $C(SU_q(2))$, and let $e_1 \in \C^2$ be the standard basis vector.  Then both the quantum group stabilizer of $e_1$ and the state stabilizer of $e_1$ are given by the counit $\varepsilon :C(SU_q(2))\rightarrow \C$.
\end{proposition}
\begin{proof}
We begin by computing $\s_{e_1}=\langle U(e_1),e_1 \otimes 1 \rangle$.
\[
\left\langle \left( \begin{array}{cc} a & -qc^* \\ c & a^* \end{array} \right)
	\left( \begin{array}{c} 1 \\ 0 \end{array} \right) , 
	\left( \begin{array}{c} 1 \\ 0 \end{array} \right) \otimes 1 \right\rangle
  = 
\left\langle	\left( \begin{array}{c} a \\ c \end{array} \right) ,
	\left( \begin{array}{c} 1 \\ 0 \end{array} \right) \otimes 1 \right\rangle 
 = a. \]

 To compute the quantum group stabilizer of $e_1$, we  consider the  $*$-representations $(\pi ,\H )$  of $C(SU_q(2))$ that send $\s_{e_1}=a$ to the identity operator.  Looking at the defining relations for $C(SU_q(2))$, one sees that $\pi $ must statisfy
\[
\pi (a)=1=\pi (a^*) \hspace{1.5in} \pi (c) =0=\pi (c^*).\] 
Since both $\pi$ and $\varepsilon$ are multiplicative, we have $\pi (b) =\varepsilon (b)1$ for any $b\in C(SU_q(2))$.  Thus $\varepsilon :C(SU_q(2))\rightarrow \C$ is the universal map of C$^*$-algebras stabilizing $e_1$. Thus it is the quantum group stabilizer of $e_1$ by Proposition \ref{cor3}. 

Recall Summary \ref{summary}. To compute the state stabilizer, we must find the $*$-representations $(\pi ,\H )$  of $C(SU_q(2))$ that have a fixed vector for $\s_{e_1}=a$.  It suffices to find the irreducibles, since $H_{e_1}$ is the closed convex hull of its pure states.  Let $(\pi ,\H )$ be an irreducible $*$-representation of $C(SU_q(2))$ with a unit vector $v\in \H$ such that $\pi (a)v=v$.  By Lemma \ref{lem4}
this implies $\pi (a^*)v=v$ as well.  

From the relations $a^*a+c^*c=1$ and $c^*c=cc^*$, we see that $\pi (c^*c)v=\pi (cc^*)v=0$.  It follows that $\pi (c)v=\pi(c^*)v=0$.  In total, we have shown 
\[
\pi (a)v=v=\pi (a^*)v \hspace{1.5in} \pi (c)v =0=\pi (c^*)v.\] 
It follows that $\C v$ is an invariant subspace for $\pi$.  Since we are only considering irreducible $\pi$, we have that $\H =\C v$.  Again, we have that $\pi$ and $\varepsilon$ are both multiplicative.  Therefore the relations above imply that $\pi (b)v=\varepsilon (b)v$ for all $b\in C(SU_q(2))$.  Thus $\varepsilon $ is the only pure state in $H_{e_1}$, and so $H_{e_1}=\{ \varepsilon \}$. Therefore $\varepsilon :C(SU_q(2))\rightarrow \C$ is the state stabilizer of $e_1$.

\end{proof}

We move on to a less trivial example.  Consider the representation of the ordinary $SU(2)$ on $M_2(\C )$ by conjugation.  Let $\lambda =\exp ( i \theta )$ be a complex number of modulus 1.  We wish to compute the stabilizer of the diagonal matrix $D_\lambda =diag[ \lambda ,\overline{\lambda} ]$.  For simplicity we take  $\theta /2\pi$ to be irrational, so that the set $\{ \lambda^n\}_{n\in \Z}$ is dense in the circle.  Any unitary matrix that stabilizes $D_\lambda$ must also stabilize $diag[\lambda^n ,\overline{\lambda^{n}}]$ for any $n\in \Z$, so the stabilizer of $D_\lambda$ must stabilize the whole set $\T =\{ diag[\exp (i\alpha ), \exp (-i\alpha )] \: | \: \alpha \in \R \}$, which is the standard maximal torus for $SU(2)$.  Thus the stabilizer of $D_\lambda$ is the usual maximal torus  $\T \subseteq SU(2)$.  

We show that for the representation of $C(SU_q(2))$ on $M_2(\C )$ via conjugation by the standard representation, the quantum group stabilizer of $D_\lambda$ is still the standard copy of $\T \hookrightarrow SU_q(2)$.  However, we will see for the state stabilizer, the answer is quite different.

We begin by computing the element $\s_{D_\lambda}$.

\begin{lemma}\label{sD}
Consider the unitary representation of $C(SU_q(2))$ on $M_2(\C)$ given by $T\mapsto U(T\otimes 1)U^*$.  Let $D_\lambda$ be the diagonal matrix considered above.  We have 

\[ \s_{D_\lambda}=2 \cdot 1_A +2i\sin (\theta )(\lambda -q^2\overline{\lambda})c^*c \]
\end{lemma}
\begin{proof}
We have $\s_{D_\lambda}=\langle U(D_\lambda \otimes 1)U^*,D_\lambda \otimes 1 \rangle $.  We compute:
\begin{eqnarray*}
 U(D_\lambda \otimes 1)U^* 
& = & 	\left( \begin{array}{cc} a & -qc^* \\ c & a^* \end{array} \right)
	\left( \begin{array}{cc} \lambda & 0 \\ 0 & \overline{\lambda} 
	\end{array} \right)
	\left( \begin{array}{cc} a^* & c^* \\ -qc & a \end{array} \right) \\
& = & \left( \begin{array}{cc} a & -qc^* \\ c & a^* \end{array} \right)
	\left( \begin{array}{cc} \lambda a^* & \lambda c^* \\ 
	-\overline{\lambda} qc & \overline{\lambda} a \end{array} \right) \\
& = & \left( \begin{array}{cc} \lambda aa^*+\overline{\lambda}q^2c^*c & (*) 
	\\ (*) & \lambda cc^*+\overline{\lambda }a^*a \end{array} \right) .
\end{eqnarray*}
So we have: 
\begin{eqnarray*}
\s_{D_\lambda} & = & \langle  U(D_\lambda \otimes 1)U^* , D_\lambda \otimes 1 
	\rangle \\
& = & \left\langle \left( \begin{array}{cc} \lambda aa^*+\overline{\lambda}q^2
	c^*c & (*) 
	\\ (*) & \lambda cc^*+\overline{\lambda }a^*a \end{array} \right) ,
	\left( \begin{array}{cc} \lambda & 0 \\ 0 & \overline{\lambda} 
	\end{array} \right) \otimes 1 \right\rangle \\
& = & aa^* +\overline{\lambda}^2q^2c^*c +\lambda^2cc^* +a^*a \\
& = & 1_A-q^2cc^* +\overline{\lambda}^2q^2c^*c +\lambda^2cc^* +1_A-c^*c \\
& = & 2\cdot 1_A +(\lambda^2 -1 +q^2(\overline{\lambda}^2-1))c^*c.
\end{eqnarray*}

Notice that $\lambda^2-1=\lambda (\lambda -\lambda^{-1}) =2i\lambda \sin (\theta )$, where $\lambda =\exp (i\theta )$.  Similarly,
 $\overline{\lambda}^2-1=-2i\overline{\lambda} \sin (\theta )$. Inserting these expressions into the equation above gives the desired result.

\end{proof}

\begin{proposition}\label{grpex}
Consider the unitary representation of $C(SU_q(2))$ on $M_2(\C)$ given by $T\mapsto U(T\otimes 1)U^*$.  Let $D_\lambda$ be the diagonal matrix considered above.  Then the quantum group stabilizer of $D_\lambda$ is the quantum subgroup $\Pi :C(SU_q(2))\rightarrow C(\T )$ characterized  by $\Pi (c)=0$.
\end{proposition}

\begin{proof}

The Hilbert-Schmidt norm of $D_\lambda$ is $\sqrt{2}$, so we need to find the universal $*$-homomorphism $\Pi :C(SU_q(2))\rightarrow B$ such that $\Pi (\s_{D_\lambda})=2\cdot 1_B$.  By the preceeding lemma, we know that a $*$-homomorphism $\pi :A\rightarrow C$ will send $\s_{D_\lambda}$ to $2\cdot 1_C$ if and only if $\pi (c^*c)=0$.   Since $\pi$ is a $*$-homomorphism, this implies $\pi (c)=0=\pi (c^*)$.  Thus the quantum group stabilizer $\Pi : C(SU_q(2))\rightarrow B$ is the universal $*$-homomorphism such that $\Pi (c)=0$.

This is the standard copy of $\T $ inside $SU_q(2)$.  Namely, if one looks at the defining relations for $C(SU_q(2))$ and adds the relation $c=0$, one is left with the universal $C^*$-algebra with one generator $a$ and the relations $a^*a=aa^*=1$.  These are the defining relations for $C(\T )$, and we have the quantum group stabilizer of $D_\lambda$ is the quantum subgroup $\Pi : C(SU_q(2))\rightarrow C(\T )$.

\end{proof}

To compute the state stabilizer of $D_\lambda$, we need to review some facts about the structure of $C(SU_q(2))$ as a $C^*$-algebra.  In the  appendix of \cite{Woro2}, Woronowicz shows that as a $C^*$-algebra, we have $C(SU_q(2))\cong C(SU_0(2))$ for all $q\in (-1,1)$.   Although $C(SU_0(2))$ is not actually a quantum group (it is a quantum semigroup), it doesn't matter because  Summary \ref{summary} explains how to compute the state stabilizer from the $*$-representation theory of the underlying C$^*$-algebra of $C(SU_q(2))$, which is just $C(SU_0(2))$.  

Notice that the element $\s_{D_\lambda}$ above depends on the value of $q \in (-1,1)$.  However, we see that regardless of the value of $q$, we have $\phi \in S(C(SU_q(2)))$ has $\phi (\s_{D_\lambda})=2$ if and only if $\phi (c^*c)=0$.  So up to Woronowicz's isomorphism, the state stabilizer of $D_\lambda$ is independent of $q$.  

An explicit faithful $*$-representation of $C(SU_0(2))$ is given in \cite{Sheu}.  Let $\H$ be a separable Hilbert space and let $\{ e_1, e_2, e_3,... \}$ be a choosen orthonormal basis for $\H$.  Then the unilateral shift operator $S\in B(\H )$ is defined by $S(e_n)=e_{n+1}$ for all $n\in \N$.  It's easy to see that $1-S^n(S^*)^n$ is the projection onto the span of $\{ e_1,...,e_n\}$, so that $C^*(S)$ contains all of the compact operators on $\H$.  It's well-known that $C^*(S)$ is the universal C$^*$-algebra generated by a single element $\mathfrak{S}$ with the relation $\mathfrak{S}^*\mathfrak{S}=1$.  Of course, if we were to add the relation $\mathfrak{S}\mathfrak{S}^*=1$ we would have the universal C$^*$-algebra generated by a single unitary element, which is $C(\T )$.  Thus we have a natural map $\rho :C^*(S)\rightarrow C(\T )$.  It is well known that this gives  a short exact sequence:

\[ 0\longrightarrow \K (\H ) \longrightarrow C^*(S) \stackrel{\rho}\longrightarrow  C(\T ) 
	\longrightarrow 0. \]

In \cite{Sheu}, Sheu shows that $C(SU_0(2))$ is isomorphic to the algebra
\[ \{ f:\T \rightarrow C^*(S) \: | \: \rho (f(x))=\rho (f(1)), \: \forall x\in \T \} .\]
This isomorphism is determined by sending $a$ to the constant function $a(x)=S^*$ and sending $c$ to the function  $c(x)=xP$, where $P=1-SS^*$ is the projection onto $\C e_1$.  There is then a short exact sequence 
\[ 0\longrightarrow C(\T )\otimes \K (\H ) \longrightarrow C(SU_0(2)) \stackrel{id\otimes \rho}\longrightarrow  C(\T ) 
	\longrightarrow 0. \]
From this short exact sequence, we see that the pure states of $C(SU_0(2))$ fall into two families. View the elements of $C(SU_0(2))$ as functions on $\T$, as described above.  If $x,y \in \T$ and $\xi \in \H$, we let
\begin{eqnarray*}
	\phi_{x,\xi }(f) & = & \langle f(x) \xi ,\xi \rangle ,   \\
\text{and} \hspace{.5in}	\psi_y(f) & = & [\rho (f(1))](y) .
\end{eqnarray*}

We now determine which of these states are in the state stabilizer of $D_\lambda$.

\begin{proposition}\label{stateex}
Consider again the unitary representation of $C(SU_q(2))$ on $M_2(\C)$ given by $T\mapsto U(T\otimes 1)U^*$.  Let $D_\lambda$ be the diagonal matrix $diag[\lambda ,\overline{\lambda}]$.  Then all of the pure states $\psi_y$ are in the state stabilizer of $D_\lambda$.  The pure state $\phi_{x,\xi }$ is in the state stabilizer of $D_\lambda$ if and only if $\langle \xi ,e_1 \rangle =0$.
\end{proposition}
\begin{proof}
Let $\varphi \in S(C(SU_0(2)))$ be arbitrary.  By Lemma \ref{lem4} and Lemma \ref{sD} we have that $\varphi \in H_{D_\lambda}$ if and only if $\varphi (c^*c)=0$.  Under Sheu's isomorphism, we have $c^*c$ maps to the constant function $c^*c(x)=P$.  Since $\K (\H )=\ker (\rho )$, we have that $\psi_y(c^*c)=0$ for any $y\in \T$. This gives the first claim. Also we have $\phi_{x,\xi }(c^*c)=\langle P\xi ,\xi \rangle = |\langle e_1,\xi \rangle |^2$.  This gives the second claim.

\end{proof}

We note that $\{ \psi_y \: | \: y\in \T \}$ is the pullback of the pure state space of the quantum group stabilizer $\Pi :C(SU_0(2))\rightarrow C(\T )$ of $D_\lambda$.  Thus, the state stabilizer of $D_\lambda$ is much larger than the quantum group stabilizer of $D_\lambda$.

Finally, we note that the two matrices 
$\left( \begin{array}{cc} 1 & 0 \\ 0 & 1 \end{array} \right)$ and
 $\left( \begin{array}{cc} i & 0 \\ 0 & -i \end{array} \right)$ span the diagonal subalgebra of $M_2(\C )$.  Since adding a scalar multiple of the identity will not change the stabilizers, we see that any diagonal matrix which is not a scalar multiple of the identity has its quantum group stabilizer given by Proposition \ref{grpex} and its state stabilizer given by Proposition \ref{stateex}.  In particular, we have computed the state stabilizer of  the rank-1 projection $\left( \begin{array}{cc} 1 & 0 \\ 0 & 0 \end{array} \right)$, which should be important for generalizing our Berezin quantization to  quantum groups which are not of Kac-type.

\end{subsubsection}

\begin{subsubsection}{Theta-Deformations of Stabilizers}

In Section 6.1.1, we recalled Wang's proceedure for deforming a compact quantum group $A$ along a toral subgroup via Rieffel deformation.  We noted that if $u:\H \rightarrow \H \otimes A$ is a finite dimensional unitary representation of $A$, then we can view $u:\H \rightarrow \H \otimes A_{\hbar J}$ as a  representation of $A_{\hbar J}$.  When viewing $u$ as a representation of $A_{\hbar J}$, we will denote it by $u_\hbar$.  Again, by Proposition 3.8 of \cite{Wang3}, we have that $u_\hbar $ is a unitary representation of $A_{\hbar J}$.

This suggests the following question.  Let $u:\H \rightarrow \H \otimes A$ be a finite dimensional unitary representation of $A$, and let $\xi \in \H$.  Suppose we know the quantum group stabilizer $\Pi : A\rightarrow B$ and the state stabilizer $H_\xi$ of $\xi$ for the representation $u$ of $A$.  Now consider the corresponding unitary representation $u_\hbar :\H \rightarrow \H \otimes A_{\hbar J}$.  Can we compute the quantum group stabilizer $\Pi_\hbar :A_{\hbar J}\rightarrow B_\hbar$ and the state stabilizer $H_\xi^\hbar$ of $\xi$ for the representation $u_\hbar$ of $A_{\hbar J}$ in terms of $\Pi $ and $H_\xi$?

If one thinks about Section 5 of Varilly's paper \cite{Varilly}, it seems natural to assume that the toral quantum subgroup of $A$ fixes $\xi$.  In this case, we find that $\Pi_\hbar$ is  just a $\theta$-deformation of $\Pi$ in a sense we describe below. 
We also give evidence to suggest $H_\xi^\hbar$ is a natural deformation of $H_\xi$ as well.

Let $A$ be a quantum group, and let $\Pi :A\rightarrow B$ be a quantum subgroup of $A$.  Suppose that $B$ has a toral subgroup $\Psi :B\rightarrow C(\T )$.  If $\theta$ is a skew-symmetric form on the Lie algebra $\mathfrak{t}$ of $\T$, then we can deform $B$ with respect to the form $J=\theta \oplus -\theta$ on $\mathfrak{t}\oplus \mathfrak{t}$ as described in Section 6.1.1 to obtain a family of quantum groups $B_{\hbar J}$.  Of course, $\Psi \circ \Pi :A\rightarrow C(\T )$ is a toral quantum subgroup of $A$, so we can deform $A$ using the same skew-form $J$ to obtain a family of quantum groups $A_{\hbar J}$.  We show that there is a natural map $\Pi_\hbar :A_{\hbar J}\rightarrow B_{\hbar J}$ that makes $B_{\hbar J}$ into a quantum subgroup of $A_{\hbar J}$.

\begin{lemma}\label{subgrdeform}
Let $\Pi :A\rightarrow B$ and $\Psi :B\rightarrow C(\T )$ be as described above.  Let $A_{\hbar J}$ and $B_{\hbar J}$ be the deformations of $A$ and $B$ using the action of $\mathfrak{t}\oplus \mathfrak{t}$ with the same skew-form $J$.  Let $\Pi_\hbar :\A_{\hbar J}\rightarrow B_{\hbar J}$ be the composition $\A_{\hbar J}=\A \stackrel{\Pi}\longrightarrow \B =\B_{\hbar J} \subseteq B_{\hbar J}$.  Then $\Pi_\hbar $ extends by continuity to all of $A_{\hbar J}$, and $\Pi_\hbar :A_{\hbar J}\rightarrow B_{\hbar J}$ is a quantum subgroup of $A_{\hbar J}$.
\end{lemma}
\begin{proof}
We have actions of $\mathfrak{t}$ on $A$ and $B$ by left translation.  Namely, for $s\in \mathfrak{t}$ we have 
\begin{eqnarray*}
 \lambda_s^A(a) & = & (\delta_{\exp (-s)}\circ \Psi \circ \Pi \otimes id)\Delta_A (a) \hspace{.5 in}  a\in A, \\
 \text{and} \hspace{.3in} \lambda_s^B(b) & = & (\delta_{\exp (-s)}\circ \Psi  \otimes id)\Delta_B (b)	\hspace{.78 in}	 b\in B.
\end{eqnarray*}
Let $a\in A$.  We have that 
\begin{eqnarray*}
\lambda_s^B(\Pi (a)) & = & (\delta_{\exp (-s)}\circ \Psi  \otimes id)\Delta_B (\Pi (a)) \\
& = & (\delta_{\exp (-s)}\circ \Psi \circ \Pi \otimes \Pi)\Delta_A (a) \\
& = & \Pi (\lambda_s^A(a)),
\end{eqnarray*}
so that $\Pi $ intertwines the actions of $\mathfrak{t}$ on $A$ and $B$ by left translation.  Similarly, $\Pi $ intertwines the actions of $\mathfrak{t}$ on $A$ and $B$ by right translation.  So if $\alpha^A$ and $\alpha^B$ are the actions of $\mathfrak{t}\oplus \mathfrak{t}$ on $A$ and on $B$ by left and right tranlations as described in Section 6.1.1, we have 
\[ \alpha_{(s,t)}^B(\Pi (a))=\Pi ( \alpha_{(s,t)}^A( a)) .\]
Since $\Pi :A\rightarrow B$ is a homomorphism of the undeformed algebras, a look at the structure of the deformed product
\[  (a_1\times _{\hbar J}a_2)(x)=\frac{1}{(\pi \hbar )^{2n}}\int _{\mathfrak {t}^{4}}\alpha_{(s,t)}(a_1)\alpha_{(u,v)}(a_2)e^{\frac{2i}{\hbar }(\theta s\cdot u-\theta t\cdot v)}dsdtdudv \]
shows that $\Pi_\hbar (a_1\times^A _{\hbar J}a_2)= \Pi_\hbar (a_1)\times^B _{\hbar J}\Pi_\hbar (a_2)$ for all $a_1, a_2 \in \A_{\hbar J}$.

Since $\Pi :A\rightarrow B$ is  $*$-preserving for the undeformed algebras, and because the involutions are not deformed in Rieffel deformation, we have that $\Pi_\hbar :\A_{\hbar J}\rightarrow B_{\hbar J}$ is $*$-preserving.

Thus $\Pi_{\hbar }:\A_{\hbar J}\rightarrow B_{\hbar J}$ is a $*$-homomorphism from $\A_{\hbar J}$ to the C$^*$-algebra $B_{\hbar J}$.  Since $A_{\hbar J}$ is defined to be the universal enveloping C$^*$-algebra of $\A_{\hbar J}$, we have that $\Pi_\hbar$ extends to a $*$-homomorphism $\Pi_\hbar :A_{\hbar J}\rightarrow B_{\hbar J}$.

Because $\Pi :A\rightarrow B$ preserves the coproduct for the undeformed algebras, and because the coproducts are not deformed on the dense subalgebras $\A$ and $\B$, we have that $\Pi_\hbar :A_{\hbar J}\rightarrow B_{\hbar J}$ preserves the coproduct.  Thus $\Pi_\hbar$ is a morphism of quantum groups.  Finally, $\Pi_\hbar:A_{\hbar J}\rightarrow B_{\hbar J}$ is surjective since its image contains the dense set $\B_{\hbar J}$, so that
   $\Pi_\hbar :A_{\hbar J}\rightarrow B_{\hbar J}$ is a quantum subgroup of $A_{\hbar J}$.

\end{proof}

Because of this lemma, we can compute the quantum group stabilizer of $\xi$ for the unitary representation $u_\hbar :\H \rightarrow A_{\hbar J}$ using a bunch of abstract nonsense.  

\begin{proposition}\label{defqstab}
Let $u:\H \rightarrow \H \otimes A$ be a finite dimensional unitary representation of a compact quantum group $A$, and let $\xi \in \H$.  Let $\Pi :A\rightarrow B$ be the quantum group stabilizer of $\xi$.  Let $\Psi :B\rightarrow C(\T )$ be a toral quantum subgroup of $B$ (so it stabilizes $\xi$).  Let $\theta$ be a skew-symmetric bilinear form on the Lie algebra $\mathfrak{t}$ of $\T$, and let $A_{\hbar J}$ and $B_{\hbar J}$ be the corresponding Rieffel deformations of $A$ and $B$.  

Consider the corresponding unitary representation $u_\hbar :\H \rightarrow \H \otimes A_{\hbar J}$.  Then the quantum group stabilizer of $\xi \in \H$ for the unitary representation $u_\hbar$ is given by the quantum subgroup $\Pi_\hbar :A_{\hbar J}\rightarrow B_{\hbar J}$ of $A_{\hbar J}$.
\end{proposition}

\begin{proof}

Since $\H$ is finite dimensional, we have $u(\xi )\in \H \otimes \A =\H \otimes \A_{\hbar J}$.  We also have $(id \otimes \Pi )u(\xi )=\xi \otimes 1_B$.  Since $u=u_\hbar$ and $\Pi = \Pi_\hbar $ on $\A =\A_{\hbar J}$, we have $(id \otimes \Pi_\hbar )u_\hbar (\xi )=\xi \otimes 1_{B_{\hbar J}}$ as well.  Thus $\Pi_\hbar :A_{\hbar J}\rightarrow B_{\hbar J}$ stabilizes $\xi$ for the unitary representation $u_\hbar$ of $A_{\hbar J}$.

We must show that $\Pi_\hbar $ is a universal morphism of quantum groups stabilizing $\xi$.  Let $\Sigma :A_{\hbar J}\rightarrow C$ be the quantum group stabilizer of $\xi$.  Since we know $\Pi_\hbar$ stabilizes $\xi$, we have $\Pi_\hbar=\Phi \circ \Sigma$, for some surjective morphism of quantum groups $\Phi :C\rightarrow B_{\hbar J}$.  Rieffel deformations are reversable.  That is, if we deform $A_{\hbar J}$ using the skew form $-J$ on $\mathfrak{t}\oplus \mathfrak{t}$, we have $(A_{\hbar J})_{-\hbar J}=A$.  Since the effect of the deformation on the quotient maps $\Pi ,\Sigma , $ or $\Phi$ is trivial on the polynomial subalgebras of the corresponding quantum subgroups, we have $\Pi = (\Pi_\hbar)_{-\hbar}=\Phi_{-\hbar} \circ \Sigma_{-\hbar}$.  By the reasoning of the first paragraph, since $\Sigma :A_{\hbar J}\rightarrow C$ stabilizes $\xi $ for $u_\hbar$, we have $\Sigma_{-\hbar} :A\rightarrow C_{-\hbar J}$ stabilizes $\xi$ for $u$.  Since $\Pi :A\rightarrow B$ is the quantum stabilizer for $\xi$, we have $\Sigma_{-\hbar}$ factors through $\Pi$, which implies that $\Phi_{-\hbar}:C_{-\hbar J}\rightarrow B $ is an isomorphism of quantum groups.  Hence $  \Phi :C\rightarrow B_{\hbar J} $ is an isomorphism of quantum groups as well.  This proves that $\Pi_\hbar $ is a universal morphism of quantum groups stabilizing $\xi$.

\end{proof}

Next we show that the state stabilizer $H_\xi$ seems to deform in a similar way.
The first thing we need to know is how the state space of $A$ is deformed under Rieffel deformation.  This question is addressed by Kaschek, Neumaier, and Waldmann in \cite{Kaschek1}.  We briefly recall their results here adapted to our scenario.

Let $V=\mathfrak{t}\oplus \mathfrak{t}$, so that we have an action $\alpha :V\rightarrow \text{C$^*$-Aut}(A)$.  Choose an inner product $g:V\otimes V\rightarrow \R$ on $V$, and consider the normalized Gaussian function
\[ G_\hbar (u)=\frac{\sqrt{\text{det}G}}{(\pi \hbar )^n}e^{-\frac{g(u,u)}{\hbar}} ,\]
where $G$ is the volume form determined by $g$, and $\text{det}G$ is the determinant of $G$ with respect to the Haar measure on $V$ and $n=\text{dim}(V)/2 =\text{dim}(\mathfrak{t})$.  

Now consider the operator $S_\hbar :A^\infty_{\hbar J}\rightarrow A^\infty$ given by 
\[ S_\hbar (a) = \int_V G_\hbar (u)\alpha_u(a) \: du. \]
The integral exists because the action $\alpha$ is isometric and because the Gaussian kernel decays rapidly.  It is clear that $S_\hbar$ is unital.
In Lemma 2.3 and Theorem 3.3 of \cite{Kaschek1}, the authors show that $S_\hbar$ has the following properties.

\begin{proposition}
Consider the linear operator $S_\hbar :A^\infty_{\hbar J}\rightarrow A^\infty$ considered above.  Let $a\in A^\infty$.  We have
\begin{eqnarray*}
\lim_{\hbar \searrow 0} S_\hbar (a) & = & a , \\
\text{and} \hspace{.3in} S_\hbar  & \cong_{\hbar \searrow 0} & \exp \left( \frac{\hbar}{4}
	\nabla_g^2  \right)  \hspace{.25in} \text{asymptotically,}
\end{eqnarray*}
both with respect to the Frechet topology of $A^\infty$, where $ \nabla_g^2$ is the Laplacian for the inner product $g$.

Suppose further that the inner product $g$ is chosen so that $h=g+iJ$ is a hermitian metric for some complex structure on $V$.  Then the map $S_\hbar :A^\infty_{\hbar J}\rightarrow A^\infty$ is positive.  

\end{proposition}

In \cite{Kaschek1}, the authors go on to show that $S_\hbar$ is continuous for the norm on $A$, so that if $\phi \in S(A)$, then one has $\phi_\hbar =\phi \circ S_\hbar \in S(A_\hbar )$ for all $\hbar >0$.  The continuity in the Frechet topology is interesting for us, because the formula for the asymptotic expansion of $S_\hbar$ suggests that $S_\hbar$ should be invertible for small $\hbar$ and for a sufficiently smooth class of elements.
For our quantum group we have $S(A)\hookrightarrow S(A^\infty )\hookrightarrow S(\A )$ by density of the inclusion maps $\A \subseteq A^\infty \subseteq A$, and so we have $S(A)= S(A^\infty )= S(\A )$ because our quantum groups are full \cite{Bedos1}.  So we see that as long as $S_\hbar$ is invertible on some sufficiently smooth class of elements, one would expect the corresponding pullback map on the state spaces to be invertible.

\begin{lemma}\label{insects}
Let $A$ be a compact quantum group, and let $\Pi :A\rightarrow C(\T )$ be a torus in $A$.  Let $u:\H \rightarrow \H \otimes A$ be a unitary representation of $A$, and let $\xi \in \H$ be a vector stabilized by $C(\T )$.  Consider the  action $\rho :\mathfrak{t} \rightarrow \text{C$^*$-Aut}(A)$ by right translations.  Then for any $t\in \mathfrak{t}$ we have
\[ (id\otimes \rho_t )u(\xi )=u(\xi ) .\]

\end{lemma}
\begin{proof}
Since $C(\T )$ stabilizes $\xi$, we have $(id \otimes \Pi)u(\xi )=\xi \otimes 1_{C(\T )}$.  We compute
\begin{eqnarray*}
(id\otimes \rho_t )u(\xi ) & = & (id \otimes id \otimes \delta_{\exp (t)}\Pi )(id\otimes \Delta)u (\xi ) \\
& = & (id \otimes id \otimes \delta_{\exp (t)}\Pi )(u\otimes id)u (\xi ) \\
& = & u[(id \otimes \delta_{\exp (t)}\Pi )u (\xi )] \\
& = & u[(id \otimes \delta_{\exp (t)} ) (\xi \otimes 1_{C(\T )} )] =u(\xi ).
\end{eqnarray*}

\end{proof}

\begin{lemma}
Let $A$ be a compact quantum group, and let $\Pi :A\rightarrow C(\T )$ be a torus in $A$.  Let $u:\H \rightarrow \H \otimes A$ be a unitary representation of $A$, and let $\xi \in \H$ be a vector stabilized by $C(\T )$.  Consider the  action $\lambda :\mathfrak{t} \rightarrow \text{C$^*$-Aut}(A)$ by left translations.  Then for any $s\in \mathfrak{t}$  and for any $\phi \in H_\xi$ we have
\[ (id\otimes \phi \circ \lambda_s )u(\xi )=\xi  .\]
\end{lemma}

\begin{proof}
Again, since $C(\T )$ stabilizes $\xi$, we have $(id \otimes \Pi)u(\xi )=\xi \otimes 1_{C(\T )}$.  Because $\phi \in H_\xi$ we have $u_\phi (\xi )=\xi$.  We compute
\begin{eqnarray*}
(id\otimes \phi \circ \lambda_s )u(\xi ) & = & (id \otimes \delta_{\exp (-s)}\Pi \otimes \phi )(id\otimes \Delta)u (\xi ) \\
 & = & (id \otimes \delta_{\exp (-s)}\Pi \otimes \phi )(u\otimes id)u (\xi ) \\
& = & (id \otimes \delta_{\exp (-s)}\Pi )u[(id\otimes \phi)u (\xi )] \\
& = & (id \otimes \delta_{\exp (-s)}\Pi )u (\xi ) \\
& = & (id \otimes \delta_{\exp (-s)} )(\xi \otimes 1_{C(\T )} ) = \xi .
\end{eqnarray*}

\end{proof}

\begin{proposition}\label{bugs}
Let $u:\H \rightarrow \H \otimes A$ be a unitary representation of $A$, and let $\xi \in \H$.  Consider the corresponding unitary representation $u_\hbar :\H \rightarrow \H \otimes A_{\hbar J}$.  Let $\phi \in H_\xi$ be an element of the state stabilizer of $\xi$ for $u$.  Then $\phi_\hbar =\phi \circ S_\hbar$ is an element of the state stabilizer of $\xi$ for $u_\hbar$.  Conversely, if $\phi_\hbar \in H_\xi^\hbar$, then we have that $ \phi \in H_\xi$ as well.
\end{proposition}
\begin{proof}
Recall that $\alpha_{(s,t)}=\lambda_s\rho_t$.  Using the preceed lemmas we have if $\phi \in H_\xi$, then 
\begin{eqnarray*}
(id\otimes \phi \circ \alpha_{(s,t)})u(\xi ) & =   & 
	(id\otimes \phi \circ \lambda_s\rho_t)u(\xi ) \\	
& =   & 
	(id\otimes \phi \circ \lambda_s)u(\xi ) =\xi .
\end{eqnarray*}
Thus we have 
\begin{eqnarray*}
(id\otimes \phi_\hbar ) u_\hbar (\xi ) & = &  (id\otimes \phi \circ S_\hbar ) u_\hbar (\xi ) = \\ 
& = &\int_V G_\hbar (s,t)(id\otimes \phi \circ \alpha_{(s,t)})u(\xi )\:  dsdt \\
& = &\int_V G_\hbar (s,t)\xi \: dsdt = \xi .
\end{eqnarray*}

For the converse, we note that $(id\otimes \phi_\hbar )u_\hbar (\xi ) =\xi$ implies that 
\[ \int_V \langle (id\otimes \phi \circ \alpha_{(s,t)})u(\xi ),\xi \rangle G_\hbar (s,t)\:  dsdt = \langle \xi , \xi  \rangle .\]
As usual, we have that $|| (id\otimes \phi \circ \alpha_{(s,t)})u(\xi ) || \leq ||\xi ||$, since all the maps we're applying have norm one.  
Since $G_\hbar (s,t)\:  dsdt$ is a probability measure on $V$, Minkowski's inequality gives that $\langle (id\otimes \phi \circ \alpha_{(s,t)})u(\xi ),\xi \rangle = \langle \xi , \xi  \rangle$ on the support of $G_\hbar (s,t)\:  dsdt$, which is all of $V$.  Now the Cauchy-Schwarz inequality gives that $(id\otimes \phi \circ \alpha_{(s,t)})u(\xi )= \xi$ for all $(s,t)\in V$.  In particular, setting $(s,t)=(0,0)$ gives $(id\otimes \phi )u(\xi )= \xi$, so that $\phi \in H_\xi$.

\end{proof}

Let $S^*_\hbar :S(A)\rightarrow S(A_{\hbar J})$  be the pullback map of $S_\hbar :A_{\hbar J} \rightarrow A$ on states.
The preceeding proposition says in particular that we have $S^*_\hbar (H_\xi )\subseteq H_\xi^\hbar$.  
One would actually have $S^*_\hbar (H_\xi )=H_\xi^\hbar$ if the map $S^*_\hbar :S(A)\rightarrow S(A_{\hbar J})$ were bijective.  This is unlikely for general Rieffel deformations of C$^*$-algebras, because the map   $S_\hbar :A_{\hbar J} \rightarrow A$ is probably not bijective.  It cannot be bijective for the whole smooth algebras, because it smooths out elements too much.  However, since $S_\hbar$ is given by convolution with a Gaussian, by Fourier analysis one would expect it to be bijective on some dense subalgebra of Schwarz functions.  For example, in \cite{Kaschek2} Kaschek shows that $S_\hbar $ is injective for $A=C^\infty (V,\R )$. He also gives the range and inverse of $S_\hbar$ for $A=\mathcal{S}(\R^2 )$, and it seems reasonable that corresponding results should hold for general $A$.

The point is this.  Since our C$^*$-algebra $A$ is actually a quantum group, we have by \cite{Bedos1} and \cite{Wang3} that $S(A)=S(A^\infty )=S(\A )$.  The polynomial subalgebra is really at the heart of the geometry for a quantum group.  If $S_\hbar$ is invertible on any class of elements, it seems quite likely that this class of elements should contain the polynomial subalgebra $\A$.  If this is the case, then we have that the corresponding pullback map $ S^*_\hbar :S(A)\rightarrow S(A_{\hbar J})$ will be bijective.  

\begin{conjecture}\label{conjecture1}
Let $u:\H \rightarrow \H \otimes A$ be a finite dimensional unitary representation of a compact quantum group $A$.  Let $\xi \in \H$ have state stabilizer $H_\xi \subseteq S(A)$.  Let $\Pi :A\rightarrow C(\T )$ be a toral subgroup of $A$ that stabilizes $\xi$.  Let $\theta$ be a skew-symmetric bilinear form on the Lie algebra $\mathfrak{t}$ of $\T$, and let $A_{\hbar J}$  be the corresponding Rieffel deformation of $A$.  

Consider the corresponding unitary representation $u_\hbar :\H \rightarrow \H \otimes A_{\hbar J}$.  Then the state stabilizer of $\xi$ for the representation $u_\hbar$ is given by $H^\hbar_\xi =S^*_\hbar (H_\xi )$, where $ S^*_\hbar :S(A)\rightarrow S(A_{\hbar J})$ is the map defined above.

\end{conjecture}

We will give evidence for this conjecture in the next section when we compute the Berezin Quantization for coadjoint orbits of $A_{\hbar J}$.  We will not obtain that $H^\hbar=H\circ S_\hbar$, but we will see that $N^{H^\hbar} =N^{H\circ S_\hbar}$ for any ergodic $A$-module algebra $N$.

\end{subsubsection}

\end{subsection}

\begin{subsection}{Examples of Berezin Quantization}

\begin{subsubsection}{$A=C^*(\Gamma )$, the Trivial Example}

Let $\Gamma$ be a discrete group, so that $A=C^*(\Gamma )$ is a cocommutative compact quantum group.  We have noted before that $A$ is then Kac-type and coamenable.  Furthermore, since all irreducible unitary (co)representations of $A$ are one-dimensional, we have that any vector in an irreducible unitary representation is primitive-like.  Thus our Berezin Quantization scheme should apply to some quantum homogeneous space $A^H$.  In this section we show that in this case we have $A^H=\C 1_A$, we have $\B^n=\C$ for all $n\in \N$, and each of the Berezin symbols and Berezin adjoint maps are trivial.

  Every  irreducible representation of $C^*(\Gamma )$ comes from an element of $\Gamma$, in the following sense. If $\gamma \in \Gamma$, then there is an irreducible representation $u:\C \rightarrow \C\otimes C^*(\Gamma )=C^*(\Gamma )$ by $u(w)=w\gamma$, and every irreducible representation of $C^*(\Gamma )$ is of this form.

Consider the irreducible representation $u(w)=w\gamma$ of $C^*(\Gamma )$.  Let $z\in \C$ be a vector of norm one.  We have the tensor product representation $u^{\otimes n}$ is given by $u^{\otimes n}(w)=w\gamma^n$.  In particular, $u^{\otimes n}$ is irreducible for all $n\in \N$, so that $z$ is a primitive-like vector.  The rank-1 projection corresponding to $z$ is $P=1_\C \in B(\C )\cong \C$.  We have $\alpha (P)=uPu^*=1_\C \otimes 1_A$.  Thus the state stabilizer $H$  of $P$ is the whole state space $S(C^*(\Gamma))$.  If $a\in C^*(\Gamma )^H$, then we have $(\phi \otimes id)\Delta a=a$ for all $\phi \in S(C^*(\Gamma ))$.  Taking $\phi$ to be the Haar state gives $a=\h (a)1_A$, so we have that $A^H=\C 1_A$.

Similarly, we have that $P^n=1_\C^{\otimes n}=1_\C$ for all $n\in N$.  Also we have $\alpha^n=id_{B(\C )}^{\otimes n}\otimes 1_A=id_{B(\C )}\otimes 1_A$.  Thus the cyclic subspace of $\C$ generated by $P^n$ is all of $\C$ for all $n$, and we have $\B^n=\C$ for all $n\in N$.

Finally we compute the Berezin symbol and its adjoint.  Let $z\in \C$.  We have $\s^n:\C \rightarrow \C 1_A$ is given by
\begin{eqnarray*}
\s^n(z) & = & (tr\otimes id)((P^n\otimes 1_A)\alpha^n(z)) \\
	& = & (tr\otimes id)((1_\C\otimes 1_A)(z\otimes 1_A)) \\
	& = & z1_A .
\end{eqnarray*}
Let $a=z1_A \in A^H=\C 1_A$.  We also have $\bs^n:A^H\rightarrow \C$ is given by
\begin{eqnarray*}
\bs^n(z1_A) & = & d_\C (id\otimes \h )([(id\otimes \kappa)\alpha^n(P^n)](1_\C\otimes z1_A)) \\
	& = & (id\otimes \h )((1_\C\otimes 1_A)(1_\C \otimes z1_A) \\
	& = & z1_\C =z.
\end{eqnarray*}

\end{subsubsection}

\begin{subsubsection}{Invariant Subspaces and Quantization of Orbifolds}

Our first main theorem of this paper was the construction of Berezin quantizations for quantum ``coadjoint orbits.'' See Proposition \ref{prop10} and Theorem \ref{main1}.  That is, we showed that if $u:\H \rightarrow \H \otimes A$ is an irreducible unitary representation of a quantum group $A$ with a primitive-like vector  $\xi \in \H$  (see Definition \ref{def6}), and $H\subseteq S(A)$ is the state stabilizer of the rank-1 projection $P\in B(\H )$ corresponding to $\xi$, then there is an explicit sequence of finite dimensional operator systems $\B^n$ converging to the quantum homogeneous space $A^H$ (see Definition \ref{quantumhomo}).  We now use Theorem \ref{main2} to show that if $\Pi :A\rightarrow B$ is a quantum subgroup of $A$, then the corresponding spaces of invariant elements also converge.  That is, we have $^\Pi\B^n$ is a Berezin quantization for $^\Pi\! A^H$.  In particular, this will allow us to give Berezin quantizations for certain double cosets of ordinary Lie groups, and we show that the latter can be orbifolds.

\begin{proposition}
In the notation of Theorem \ref{main2}, let $N$ be the ergodic $A$-module algebra $^\Pi\! A$, where $\Pi :A\rightarrow B$ is a quantum subgroup of $A$.  Then we have $N^H=\: ^\Pi\! A^H$, and the finite dimesional algebras $M_n$ are given by $^\Pi\B^n$.  Thus by Corollary \ref{Berezin2} and Theorem \ref{main2} we have that $\{ \: ^\Pi\B^n \} \cup \{ \: ^\Pi\! A^H \}$ is a strict quantization  of $^\Pi\! A^H$, and that for any invariant Lip norm $L_A$ on $A$, the corresponding quantum metric spaces $(\: ^\Pi\B^n ,L_n)$ converge to $(\: ^\Pi\! A, L)$ in quantum Gromov-Hausdorff distance.
\end{proposition}
\begin{proof}
It's clear that for $N= \: ^\Pi\! A$, we have $N^H =\: ^\Pi\! A ^H$, where
\[ ^\Pi\! A^H =\{ a\in A \: | \: (id\otimes \Pi )\Delta a=a\otimes 1_B \text{ and } _\phi\Delta (a)=a, \hspace{4pt} \forall \phi \in H \} . \]
Next we compute $M$.  See Definition \ref{defM}. We have 
\[ M_n =\{ \omega \in \B \otimes \: ^\Pi\! A \: | \: (\alpha \otimes id)\omega =(id\otimes \Delta )\omega \} . \] 
By the reasoning in the example immediately following Definition \ref{defM}, we have that $(id\otimes \varepsilon ):M\hookrightarrow \B$.  To find the subspace of $\B$ that is the image of $M$, we must find for what $T\in \B$ do we have $\alpha (T)\in M$.

We claim that $\alpha (T)\in M$ if and only if $(id\otimes \Pi )\alpha (T)=T\otimes 1_B$; that is, if and only if $T\in \: ^\Pi\B$.  We have 
\begin{eqnarray*}
\alpha (T)\in M & \Leftrightarrow & 
	(id\otimes id\otimes \Pi )(id\otimes \Delta )\alpha (T)
	 = \alpha (T)\otimes 1_B \\
& \Leftrightarrow & (id\otimes id\otimes \Pi )(\alpha \otimes id)\alpha (T)
	 = \alpha (T)\otimes 1_B \\
& \Leftrightarrow & (\alpha \otimes id)( id\otimes \Pi )\alpha (T)
	 =( \alpha \otimes id)(T\otimes 1_B )\\
& \Leftrightarrow & ( id\otimes \Pi )\alpha (T)
	 =T\otimes 1_B ,
\end{eqnarray*}
where for the last equivalence we have used that $(\alpha \otimes id)$ is injective.
\end{proof}

\begin{example}
Let $A=C(G)$ for $G$  a compact semisimple Lie group, and let $U:G\rightarrow B(\H )$ be an irreducible unitary representation of $G$.  Let $K<G$ be a subgroup of $G$, and let $H<G$ be the stabilizer of the rank-1 projection corresponding to the highest weight vector of $U$ under conjugation by $U$.  Then we have 
$^K\! A^H=C(K\backslash G/H)$ is a double coset space, and $^K\! \B=U(K)'\subseteq B(\H )$ is the commutant of $U(K)$ in $B(\H )$.
\end{example}

\begin{proposition}
In the above example, suppose that $K\subseteq gHg^{-1}$ for some $g\in G$.  Then the action of $K$ on $G/H$ fixes the point $gH \in G/H$.  If the action of $K$ on  $G/H$ has an isolated fixed point, and if the action is properly discontinuous, then the double coset $K\backslash G/H $ is not a differentiable manifold for the induced differential structure.
\end{proposition}
\begin{proof}
Let $k\in K\subseteq gHg^{-1}$ be arbitrary. Then we have  $k\cdot (gH)=gH$.  Thus $K$ fixes $gH$.

The second statement is really a general principle.  Intuitively, if $gH$ is an isolated fixed point, then in any neighborhood of $gH$ we can find at least two points which map to the same point in $K\backslash G/H$.  If the action is properly discontinuous, we can find two such points that are ``seperated'' from each other.  At the infinitesimal level, there are directional derivatives that are distinct  just off $gH$ but become identified at $gH$.  Thus the ``tangent bundle'' of $K\backslash G/H$ is not locally trivial for this differential structure, and $K\backslash G/H$ is not a differentiable manifold.
\end{proof}

\begin{example}
Let $G=SU(2)$ and let $U:SU(2)\rightarrow M_2(\C )$ be the standard representation. The highest weight vector of $U$ is $e_1$, so that 
 $P=\left( \begin{array}{cc} 1 & 0 \\ 0 & 0 \end{array} 
	\right) $.  Then we have the stabilizer of $P$ is  the maximal torus $H<G$ consisting of diagonal matrices in $SU(2)$.  We compute
\[
\left( \begin{array}{cc} a & -\overline{c} \\ c & \overline{a} \end{array} 
	\right) 
\left( \begin{array}{cc} \lambda & 0 \\ 0 & \overline{\lambda} \end{array} 
	\right) =
\left( \begin{array}{cc} a\lambda & -\overline{c\lambda} \\ c\lambda & \overline{a\lambda} \end{array} 
	\right) .
\]
Thus we have 
\[ SU(2)/H =\frac{\{ (a,c)\in \C^2 \: : \: |a|^2+|c|^2=1 \}}{(a,c)\sim (a\lambda ,c\lambda ) \text{ for } \lambda \in S^1 } =\mathbb{CP}^1 . \]
Via the equivalence relation, we see that $(0,c)\sim (0,1)$ for all $|c|=1$.  For $|c|<1$ we can choose a unique $\lambda \in S^1$ to make $a\lambda >0$. Thus we can use the stereographic projection map to obtain
\[  SU(2)/H = \{ (2|a|\mathfrak{Re}(c), 2|a|\mathfrak{Im}(c), |a|^2-|c|^2 )\in \R^3 \} = S^2, \] 
as is well known.

Now take $K=\Z_n\hookrightarrow H$ to be the copy of the cyclic group inside $H$:
\[ \Z_n = \left\{ \left( \begin{array}{cc} \zeta^j & 0 \\ 0 & \overline{\zeta^j} \end{array} \right) \in SU(2) \: | \: j=0,1,...,n-1 \right\} ,\hspace{.5in} \text{where } \zeta=\exp(2\pi i/n). \] 
We have $\left( \begin{array}{cc} \zeta & 0 \\ 0 & \overline{\zeta} \end{array} \right) \cdot (a,c) = (\zeta a,\overline{\zeta}c)\sim (a,\overline{\zeta^2}c)$.  Looking at the formula for the stereographic projection above, we see that the generator of $\Z_n$ acts on $S^2$ by  rotating $4\pi /n$ in the plane spanned by  first two coordinate vectors. For simplicity take $n\in \N$ to be odd.  Then we have each point of $S^2$ is identified with it's image after rotating by $2\pi /n$.  Thus we see that $S^2/\Z_n$ is isometric to the ellipsoid  with radii $1/n, 1/n,$ and $1$.  As the quotient of a manifold by a finite group action, $S^2/\Z_n$ is a prototypical example of an orbifold. The action is properly discontinuous since $\Z^n$ is finite.  Because the north and south poles are isolated fixed points for the action of $\Z_n$, we see that $S^2/\Z_n$ is not a manifold for the induced differential structure.

Now we look at its finite dimensional approximations $^{\Z_n}\! M_N(\C )=U^N(\Z^n)'$.  The $N$ dimensional representation of $SU(2)$ has a basis $\{ v_{N-1},v_{N-3},...,v_{-N+1} \}$ such that $\left( \begin{array}{cc} \zeta & 0 \\ 0 & \overline{\zeta} \end{array} \right) \cdot v_j=\zeta^j v_j$.  For simplicity, take $N=rn$ for some $r\in \N$.  Still assuming that $n$ is odd, we have that exactly $r$ of the basis vectors are fixed by $\left( \begin{array}{cc} \zeta & 0 \\ 0 & \overline{\zeta} \end{array} \right)$, exactly $r$ are multiplied by $\lambda$, ..., and exactly $r$ are multiplied by $\lambda^{n-1}$.  Thus the commutant of the image of $\Z^n$ in $M_N(\C )$ is isomorphic to $M_r(\C )^{\oplus n}$. So we have a limit
\[ \lim_{r\rightarrow \infty} M_r(\C )^{\oplus n} = C(S^2/\Z_n) \]
in a continuous field of operator systems. Given a metric on $SU(2)$, we have a limit as quantum metric spaces.

\end{example}

\end{subsubsection}

\begin{subsubsection}{Coadjoint Orbits for $C(G)_{\hbar J}$}

The goal of this section is to show that our Berezin Quantization commutes with Wang's use of Rieffel deformation to deform compact quantum groups \cite{Wang3}.  We begin by recalling our initial data.

As usual, we let $A$ be a coamenable compact quantum group of Kac-type.  We let $u:\H \rightarrow  \H \otimes A$ be an irreducible unitary representation of $A$ with a primitive-like vector $\xi \in \H$, see Definition \ref{def6}. 
   We let $\H^n$ be the cyclic subspace of $\H^{\otimes n}$ generated by $\xi^n\equiv \xi^{\otimes n}$ for the unitary representation $u^{\otimes n}$.  
We let $P^n\in B(\H^n )$ be the rank-1 projection corresponding to $\xi^n$. 
We let  $\alpha^n :B(\H^n )\rightarrow B(\H^n )\otimes A$ be the action of $A$ given by conjugation by $u^{\otimes n}$, and we let $\B^n\subseteq B(\H^n)$ be the cyclic subspace generated by $P^n$.  
We let $H\subseteq S(A)$ to be the state stabilizer of $P$ under  $\alpha$ (see Definition \ref{defstabilizer}).

Now let $\pi :N\rightarrow A\otimes N$ be an ergodic right $A$-module algebra (see Definition \ref{rightmodalg}), and let $N^H=\{ x\in N \: | \: _\phi\pi (x)=x \text{ for all } \phi \in H\}$ be the set of $H$-invariant elements of $N$.  Let $M^n=\{ \omega \in \B^n \otimes N\: | \: (\alpha^n\otimes id)(\omega )=(id\otimes \pi )(\omega ) \}$.  Finally, we have the Berezin symbol $\s^n:M^n\rightarrow N^H$ and the Berezin adjoint $\bs^n :N^H\rightarrow M^n$ given by 
\begin{eqnarray*}
\s^n (\omega ) & = & (tr\otimes id)((P^n\otimes id)\omega ),  \\
\text{  and} \hspace{15pt}   \bs^n (x) & = & d_{\H^n} ((\alpha^n)^\dagger \otimes id)(P^n\otimes \pi (x))\\
& = & d_{\H^n}(id\otimes \h \circ \mathsf{m} \otimes id)([(id\otimes \kappa)\alpha^n(P^n)]\otimes \pi (x)).
\end{eqnarray*}

Corollary \ref{Berezin2'} says that the sequence of Berezin Transforms $\s^n \circ \bs^n:N^H\rightarrow N^H$ converges strongly to the identity operator.   Theorem \ref{main2} adds that for any invariant Lip-norm $L_A$ on $A$, the corresponding quantum metric spaces $(M^n, L_n)$ converge to $(N^H,L)$ in quantum Gromov-Hausdorff distance.

Suppose now we have a toral quantum subgroup $\Pi :A\rightarrow C(\T )$ of $A$.  We saw in Section 6.1.1 that we can deform $A$ into a family of coamenable quantum groups $A_{\hbar J}$ all of Kac-type, and we've seen that we have a deformed family of unitary representations $u_\hbar :\H \rightarrow \H \otimes A_{\hbar J}$.  In this section we show that the rest of this picture deforms as well.  We will find $M^n_\hbar =M^n$ for all $n\in N$.  We have $N^H_\hbar $ and $N^H$ have the same smooth subspace, so that $N^H_\hbar $ and $N^H$ are merely completions of the same space in different norms.  Finally the Berezin symbol and its adjoint are undeformed on this smooth subspace.

We  saw in Section 6.3.2  that if $A=C(G)$ is the algebra of continuous functions on an ordinary compact group $G$, and if $K<G$ is the usual stabilizer of $P$ under $\alpha$, then  the quantum group stabilizer of $P$ under the deformed action $\alpha_\hbar$ is given by the quantum subgroup $\Pi_\hbar :C(G)_{\hbar J}\rightarrow C(K)_{\hbar J}$.   
Since for the undeformed $C(G)$ we know that both the quantum group stabilizer and the state stabilizer of $P$ are equal to $C(K)$, Conjecture \ref{conjecture1} suggests that the state stabilizer of $P$ under the deformed action $\alpha_\hbar$ should be equal to $C(K)_{\hbar J}$ as well.  We will not be able to obtain this result, but we will show that the set of $H^\hbar$-invariant element elements of $C(G)_{\hbar J}$ is the same as the set of $C(K)_{\hbar J}$-invariant elements, which gives strong support for our conjecture in this case.
 Also, we will see that the quantum coadjoint orbit $C(G)_{\hbar J}^{H^\hbar}$ is just the $\theta$-deformed homogeneous space $C(G/K)_{\hbar \theta}$ from Section 5 of \cite{Varilly}.

Let $\theta$ be a skew-symmetric form on the Lie algebra $\mathfrak{t}$ of $\T$.  As described in Section 6.1.1, we have two actions of $\mathfrak{t}$ on $A$ given by 
\begin{eqnarray*}
\lambda_s (a) & = & (\delta_{\exp (-s)}\Pi \otimes id)\Delta a , \\
\text{and} \hspace{15pt} \rho_t (a) & = & (id\otimes \delta_{\exp (t)})\Delta a.
\end{eqnarray*}
These two actions commute with each other, so there is an action of  $\mathfrak{t}\oplus \mathfrak{t}$ on $A$ given by $\alpha_{(s,t)}=\lambda_s\rho_t$.
 We can deform $A$ by the action of $\mathfrak{t}\oplus \mathfrak{t}$ with the skew form $J=\theta \oplus (-\theta )$ to obtain a family of compact quantum groups $A_{\hbar J}$.  

We also have an action of $\mathfrak{t}$ on $N$ given by
\[ l_s(x)= (\delta_{\exp (-s)}\Pi \otimes id )\pi (x) .\]
Thus we can deform $N$ by this action of $\mathfrak{t}$ using the skew form $\theta$ to obtain a family of C$^*$-algebras $N_{\hbar \theta}$.  We show that $N_{\hbar \theta}$ is an $A_{\hbar J}$-module algebra.

First, we stop to discuss the algebra of regular functions of an $A$-module algebra. We define algebra of regular functions of $N$ to be the algebraic direct sum of the isotypic components on $N$, and we denote it by $\mathcal{N}$, see \cite{Li}. Clearly, $\mathcal{N}$ is a dense $*$-subalgebra of $N$.   By the argument of \cite{Podles}, the action $\pi :N\rightarrow A\otimes N$ of the compact quantum group $A$ restricts  to an action $\pi :\mathcal{N}\rightarrow \A \otimes \mathcal{N}$ of the algebraic quantum group $\A$ on $\mathcal{N}$.  Finally, we denote by $\mathcal{N}_{\hbar \theta}$ the dense subalgebra of $N_{\hbar \theta}$ which is equal to $\mathcal{N}$ as a linear space.

\begin{proposition}\label{deformmodule}
Let $\pi :N\rightarrow A\otimes N$ be a right $A$-module algebra, and let $\Pi :A\rightarrow C(\T )$ be a toral quantum subgroup of $A$.  Let $\theta$ be a skew-symmetric form on the Lie algebra of $\T$.  Then the natural map 
\[ \mathcal{N}_{\hbar \theta}=\mathcal{N} \stackrel{\pi}\longrightarrow 
	\A \otimes \mathcal{N} =\A_{\hbar J} \otimes \mathcal{N}_{\hbar \theta}
	\subseteq A_{\hbar J} \otimes {N}_{\hbar \theta} \]
lifts to an action $\pi_\hbar :N_{\hbar \theta} \rightarrow A_{\hbar J}\otimes N_{\hbar \theta} $ of $A_{\hbar J}$ on $N_{\hbar \theta}$.  Moreover, one has that $\pi_\hbar $ is ergodic if and only if $\pi$ is ergodic.
\end{proposition} 
\begin{proof}
The first half of this proof is very similiar to our proof of Lemma \ref{subgrdeform} above.  The remainder  mimics Rieffel proof that $\Delta_\hbar :C(G)_{\hbar J}\rightarrow C(G)_{\hbar J}\otimes C(G)_{\hbar J}$ is a $*$-homomorphism in \cite{Rieffel5}, or Wang's corresponding proof in \cite{Wang3}.  We give a quick sketch.

First of all, by \cite{Podles} we do have $\pi :\mathcal{N}\rightarrow \A \otimes \mathcal{N}$.  By the same argument as in Proposition 3.2 of  \cite{Wang3}, it is easy to see that  $\mathcal{N}\subseteq N^\infty$, where $N^\infty$ is the dense subalgebra of smooth elements for the action of $\mathfrak{t}$.  Hence the string of linear maps above makes sense if we define $\mathcal{N}_{\hbar \theta}$  to be $\mathcal{N}$ as a linear space endowed with the deformed product of $N_{\hbar \theta}$.  

Since the coproduct $\Delta_{\hbar }:A_{\hbar J}\rightarrow A_{\hbar J}$ is undeformed on the polynomial subalgebra $\A_{\hbar J}$, it is easy to see that $\pi_\hbar$ defined above satisfies $(id\otimes \pi_\hbar )\pi_\hbar =
(\Delta_{\hbar }\otimes id)\pi_\hbar$ on the algebra of regular functions $\mathcal{N}_{\hbar \theta} $.  Since the involution of $N$ is undeformed on $\mathcal{N}$, we have that $\pi_\hbar$ is $*$-preserving.  

It remains to show that $\pi_\hbar :\mathcal{N}_{\hbar \theta}\rightarrow \A_{\hbar J}\otimes \mathcal{N}_{\hbar \theta}$ is a homomorphism for the deformed products.  As in \cite{Rieffel5}, we let 
\[ C=\{ F\in A\otimes N \: | \: (\rho_t \otimes id)F=(id\otimes l_{-t})F , \hspace{16pt} \text{for all} \hspace{7pt} t\in \mathfrak{t} \} .\]
It is easy to check that $\pi (N)\subseteq C$. We denote $\pi$ by $\psi$ when we view it as a map from $N$ to $C$.   Since the actions $\rho$ and $\lambda$ of $\mathfrak{t}$ on $A$ commute, there is also an action of $\mathfrak{t}$ on $C$ given by 
\[ \beta_t (F)=(\lambda_t \otimes id)(F) .\]
A quick check gives that $\psi \circ l_t =\beta_t \circ \psi$, so that $\psi$ intertwines the actions of $\mathfrak{t}$ on $N$ and on $C$.  Thus there is a natural map  
$\psi_\hbar :N^l_{\hbar \theta}\rightarrow C^\beta_{\hbar \theta}$ since Rieffel deformation is functorial.  

Next, define an action $\gamma$ of $\mathfrak{t}^{ 3}$ on  $A\otimes C$ by 
\[ \gamma_{(s,t,u)}=\alpha_{(s,t )}\otimes l_u .\]
We have that $\gamma$ restricts to an action on $C$ because $\T$ is abelian and because $\lambda$ and $\rho$ commute with each other.  By the definition of $C$ we have that the subspace $\{ (0,t,t) \in \mathfrak{t}^3 \}$ acts trivially on $C$. Consider the skew form $L=\theta \oplus (-\theta )\oplus \theta$ on $\mathfrak{t}^3$.
 The argument of \cite{Rieffel5} gives that $C^\gamma_L=C^\beta_\theta$.  We let
\[ \pi_\hbar = \iota \circ \psi_\hbar ,\]
where $\iota$ is the inclusion  $\iota : C^\gamma_L\hookrightarrow (A\otimes N)^\gamma_L$.  Since $(A\otimes L)^\gamma_L =A^\alpha_J\otimes N^l_\theta$, and because $\psi_\hbar$ and $\iota$ 
are both unital $*$-homomorphisms, we have that $\pi_\hbar$ is a $*$-homomorphism from $N_{\hbar \theta}$ to $A_{\hbar J}\otimes N_{\hbar \theta}$.

Since $\psi =\pi$ on $\mathcal{N}$, it's easy to see that $\pi_\hbar$ is given by the composition stated in the proposition when restricted to $\mathcal{N}_{\hbar \theta}$.

Suppose now that $\pi :N\rightarrow A\otimes N$ is ergodic.  We have the canonical expectations $E:N\rightarrow N^\text{inv}$ and $E_\hbar :N_{\hbar \theta}\rightarrow N_{\hbar \theta}^\text{inv}$ onto the invariant subalgebras for $\pi$ and $\pi_\hbar$ given by $E=(\h \otimes id)\pi$ and  $E_\hbar =(\h \otimes id)\pi_\hbar$.  We see that $E=E_\hbar$  on the subalgebras of regular functions $\mathcal{N}=\mathcal{N}_{\hbar \theta}$.  By ergodicity, we have $E(N)=\C 1_N$.  Thus we have $E_\hbar (\mathcal{N}_{\hbar \theta})\subseteq \C 1_{N_{\hbar \theta}}$.  Since $\mathcal{N}_{\hbar \theta}$ is dense in $N_{\hbar \theta}$, we have $E_\hbar (N_{\hbar \theta})= \C 1_{N_{\hbar \theta}}$ also, which says that $\pi_\hbar $ is ergodic.  The converse is similar.

\end{proof}

Now that we understand how to deform the $A$-module algebra $N$, let us return to the action $\alpha :B(\H ) \rightarrow B(\H ) \otimes A$.  We should have a deformed action $\alpha_\hbar :B(\H)_{\hbar \theta} \rightarrow B(\H )_{\hbar \theta} \otimes A_{\hbar J}$ which agrees with $\alpha$ on the dense subalgebra of regular functions of $B(\H )$.  Since $B(\H )$ is finite dimensional it is equal to its subalgebra of regular functions, and we have $\alpha =\alpha_\hbar$ when consider $B(\H )$ and $B(\H )_{\hbar \theta}$ merely as linear spaces.  Thus the cyclic subspaces $\B$ and $\B_{\hbar \theta}$ generated by $P$ are not deformed as linear spaces.  Only the order structure of $\B$ is deformed.

Now we move on to the deformation of the finite dimensional approximations $M^n$ of $N^H$.  As one might expect from the corresponding statement for $\B$, we have that $M^n$ is not deformed as a linear space.  Only its order structure is deformed.

\begin{proposition}
We have $M=M_{\hbar \theta}$ as linear spaces.  The natural identification between them is given by the restriction of the ``identity'' map $\B \otimes \mathcal{N} \rightarrow \B_\hbar \otimes \mathcal{N}_{\hbar \theta}$.
\end{proposition}
\begin{proof}
Recall that $M\subseteq \B \otimes N$ is given by
\[ M=\{ \omega \in \B \otimes N \: | \: (\alpha \otimes id)\omega =(id\otimes \pi )\omega \} .\]
Since $\B$ is finite dimensional, it is easy to see that any $\omega \in M$ must actually be an element of $\B \otimes \mathcal{N}$.   Since $\alpha_\hbar$ and $\pi_\hbar$ are not deformed on the regular function algebras $\B$ and $\mathcal{N}$, we have that $(\alpha_\hbar \otimes id)\omega_\hbar =(id\otimes \pi_\hbar )\omega_\hbar$, where $\omega_\hbar$ is $\omega$ viewed as an element of $\B_\hbar \otimes \mathcal{N}_{\hbar \theta}$.  Thus $\omega \mapsto \omega_\hbar$ gives a mapping $M\rightarrow M_{\hbar \theta}$.  Similarly, one can check that $\omega_\hbar \in M_{\hbar \theta}$ implies $\omega \in M$, so we have $M=M_{\hbar \theta}$ as linear spaces.

\end{proof}

Finally, we check that the Berezin symbol and Berezin adjoint are undeformed, in the obvious sense.  Because it is not clear a priori how the state stabilizer of $P$ should deform, we view the Berezin symbol as a map $\s :M\rightarrow N$ and the Berezin adjoint as $\bs :N\rightarrow M$.  We recall a basic fact from Fourier analysis.  We use the shorthand $e_\hbar(t)=\exp(2 \pi i \hbar t)$ for the rest of this section.  Recall from Section 6.1.1 that if $C$ is a C$^*$-algebra with an action of $\R^d$ that comes from the action of a torus, then the product on $C_{\hbar \theta}$ can be written in terms of the weight spaces for $\T$ as
\[ (\xi\times_{\hbar \theta}\eta )_n=\Sigma_{m\in \Z^d} \: \xi_m\eta_{n-m} e_\hbar (\theta m \cdot n), \hspace{.4in} \xi ,\eta \in C . \]

We begin by showing that the trace on $B(\H )_{\hbar \theta}$ is the same as the trace on $B(\H )$.

\begin{lemma}
Let $\alpha :\T\rightarrow C$ be an action of a torus on a C$^*$-algebra $C$. Let $\varphi \in S(C)$ be a distinguished positive linear functional on $C$, and suppose that the action $\alpha$ is unitary for the inner product $\langle \xi , \eta \rangle =\varphi (\eta^*\xi )$ on $C$.  Then for any $\xi ,\eta \in C^\text{alg}$ we have 
\[ \varphi (  \xi \times_{\hbar \theta}\eta )=\varphi (\xi \eta ) .\]
Let $\varphi_\hbar$ be $\varphi$ considered as a linear functional on $C^\text{alg}_{\hbar \theta}$.  In particular, we have that if $\varphi $ is a trace on $C$, then $\varphi_\hbar$ is a trace on $C^\text{alg}_{\hbar \theta}$.
\end{lemma}
\begin{proof}
As in the statement, we let $\langle \xi , \eta \rangle =\varphi (\eta^*\xi )$.  Since $\alpha$ is assumed to be unitary for this inner product, we have 
\[ \langle \alpha_t(\xi ) , \alpha_t(\eta )\rangle =\varphi (\eta^*\xi ) \hspace{.4in} \forall t\in \T .\]
Let $m $ and $n$ be in $ \Z^d$, and suppose that $\xi$ is in the weight space $C_m$ and $\eta \in C_n$.  Then we also have 
\[ \langle \alpha_t(\xi ) , \alpha_t(\eta )\rangle =e(t\cdot (m-n))\varphi (\eta^*\xi ) \hspace{.4in} \forall t\in \T .\]
Combining this with the previous equation gives that the weight spaces $C_m$ and $C_n$ are orthogonal for $m\neq n$.  Since $1_C\in C_0$, we have for any $\xi \in C$ that 
\begin{eqnarray*}
\varphi (\xi ) & = & \varphi (1_C^*\xi ) \\
& = & \varphi (1_C^*\xi_0 ) \\
& = & \varphi (\xi_0 ),
\end{eqnarray*}
where $\xi_0$ is the component of $\xi$ inside $C_0$.  In particular, we have that $\varphi (  \xi \times_{\hbar \theta}\eta ) = \varphi ( (  \xi \times_{\hbar \theta}\eta )_0)$.  But since $\theta$ is skew-symmetric,  we have
\begin{eqnarray*}
(  \xi \times_{\hbar \theta}\eta )_0 & = & \Sigma_m \: \xi_m\eta_{-m} e_\hbar 
(-\theta m \cdot m) \\ 
& = & \Sigma_m \: \xi_m\eta_{-m} =  (  \xi \eta )_0,
\end{eqnarray*}
so that $\varphi (  \xi \times_{\hbar \theta}\eta )=\varphi (\xi \eta )$.
\end{proof}

\begin{corollary}\label{trac}
Let $tr_\hbar$ be the trace on $B(\H )$, considered as a linear functional on $B(\H )_{\hbar \theta}$.  Then $tr_\hbar$ is the trace of $B(\H )_{\hbar \theta}$.
\end{corollary}
\begin{proof}
The only part of this statement that is not covered in the lemma is that $tr_\hbar $ is continuous on all of $B(\H )_{\hbar \theta}$.  This is clear since $B(\H )_{\hbar \theta}$ is finite dimensional.
\end{proof}
\begin{corollary}\label{Ha}
Let $\h$ be the Haar state of $A$, and let $\mathsf{m}:A\otimes A\rightarrow A$ be the multiplication map.  Then we have $\h \circ \mathsf{m} =\h \circ \times_{\hbar J}$ on the polynomial subalgebra $\A =\A_{\hbar J}$.
\end{corollary}

Using these two corollaries, we can now show that the Berezin symbol and Berezin adjoint are not deformed on the smooth subalgebra of $N$.

\begin{proposition}
Consider the Berezin symbols $\s :M\rightarrow N$ and $\s_\hbar :M_{\hbar \theta}\rightarrow N_{\hbar \theta}$.  Let $\omega \in M$.  Then $\s (\omega)\in \mathcal{N}$.  Furthermore, we have
\[ \s_\hbar (\omega_\hbar )=\s (\omega )_\hbar .\]
On the left hand side we are denoting $\omega$ viewed as an element of $M_{\hbar \theta}$ by $\omega_\hbar$.  On the right hand side we are denoting $\s (\omega )$ viewed as an element of $\mathcal{N}_{\hbar \theta}$ by $\s (\omega )_\hbar$.
\end{proposition}

\begin{proof}

We recall that $\s :M\rightarrow N$ is given by
\[ \s (\omega )=(tr\otimes id)((P\otimes 1)\omega ) .\]
As in the proof of Proposition 6.15, we note that since $M$ is finite dimensional any $\omega \in M$ must be an element of $\B \otimes \mathcal{N}$, so that $\s (\omega )\in \mathcal{N}$. 

We are not changing our choice of primitive-like vector $\xi$, so the rank-1 projection $P$ is not deformed. By Corollary \ref{trac} the trace is not deformed, so the result is clear.
\end{proof}

\begin{proposition}
Let $x\in \mathcal{N}$.  Denote by $x_\hbar$ the element $x$ viewed as a member of $\mathcal{N}_{\hbar \theta}$.  Consider the Berezin adjoints $\bs :N\rightarrow M$ and $\bs_\hbar :N_{\hbar \theta}\rightarrow M_{\hbar \theta}$.  Then we have
\[ \bs_\hbar (x_\hbar )=\bs (x)_\hbar ,\]
in the evident notation.
\end{proposition}
\begin{proof}
Again, we look at the formula
 \[ \bs_x =d_\H (id\otimes \h \circ \mathsf{m} \otimes id)((id\otimes \kappa )\alpha (P)\otimes \pi (x)) .\]
We know that $\kappa , \alpha , P, \pi $, and $x$ are all essentially not being deformed.  By Corollary \ref{Ha} we have that $\h \circ \mathsf{m} = \h \circ \times_{\hbar \theta}$, so the result follows.
\end{proof} 

From these two propositions, we see that Berezin Transform $\s \circ \bs : N\rightarrow N$ is undeformed on the subalgebra of regular functions $\mathcal{N}$.  By Corollary \ref{Berezin2'}, we have that the sequence of Berezin Transforms $(\s^n\circ \bs^n)$ converges strongly to the identity operator on $N^H$.  This suggests that $N^H\cap \mathcal{N}$ is undeformed as a linear space; only its order structure is deformed.  We check that this is the case.

\begin{proposition} 
Let $N^H\cap \mathcal{N}=\mathcal{N}^H$.  Then $\mathcal{N}^H$ is dense in $N^H$.  
Under the $\theta$-deformation of $N$, we have that $\mathcal{N}^H=\mathcal{N}_{\hbar \theta}^{H^\hbar}$ as linear spaces.  That is, only the order structure of $\mathcal{N}^H$, and hence its norm, is deformed.  Intuitively,  we have that $N^H$ and $N_{\hbar \theta}^{H^\hbar}$ are just completions of $\mathcal{N}^H$ with respect to different norms.
\end{proposition}
\begin{proof}
We know that $(\s^n\circ \bs^n)(x)=(\hpn \otimes id)\pi (x)$ for some sequence of states $\hpn$ given before.  We shall see in  Corollary \ref{Haar} of the appendix that these states actually converge weak$^*$ to a state $\mathbi{h}\in S(A)$, and that this $\mathbi{h}$ is a Haar state for the state stabilizer $H$.   Thus we have an expectation $E:N\rightarrow N^H$ given by
\[ E(x)=(\mathbi{h}\otimes id)\pi (x) .\]
We have seen that $\pi :\mathcal{N}\rightarrow \A \otimes \mathcal{N}$, so that $E:\mathcal{N}\rightarrow \mathcal{N}^H$.  In particular, using that $E$ is idempotent gives $E(\mathcal{N})=\mathcal{N}^H$.  Since $E$ is continuous, we have that $\mathcal{N}^H$ is dense in $N^H$. 

As was noted before, the Berezin Transforms are undeformed on the subalgebra of regular functions $\mathcal{N}$.  As $E$ is the strong limit of these Berezin Transforms, we have that $E(x)=E_\hbar (x_\hbar )$ for all $x\in \mathcal{N}$, in the evident notation.  Thus we see that $\mathcal{N}^H=E(\mathcal{N})=E_\hbar (\mathcal{N}_{\hbar \theta} )=\mathcal{N}_{\hbar \theta}^{H^\hbar}$.
\end{proof}

\begin{remark}
Slight modifications of the proofs of Lemma \ref{insects} through Proposition \ref{bugs} show that we also have $\A^H=\A^{H\circ S_\hbar}$.  Combining this with the above proposition shows that $\A^{H^\hbar}=\A^{H\circ S_\hbar}$.  We still only know that $H\circ S_\hbar \subseteq H^\hbar$, but we see that for this example $H\circ S_\hbar$ is a large enough subset of $H^\hbar$ to determine its quantum homogeneous space.  
\end{remark}

\begin{example}
Suppose that $A=C(G)$ is the algebra of continuous functions on a compact semisimple Lie group $G$.  As usual, we take $P$ to be the rank-1 projection corresponding to the highest weight vector in some irreducible unitary representation $U$ of $G$.   Let $K<G$ be the ordinary group stabilizer of $P$ for the representation $\alpha_g(T)=U_gTU_g^*$ of $G$ on $B(\H )$.  We have seen before that the restriction map $\Pi :C(G)\rightarrow C(K)$ is both the quantum group stabilizer and the state stabilizer of $P$.

By Proposition \ref{defqstab} the quantum group stabilizer of $P$ for the representation $\alpha_\hbar :B(\H )\rightarrow B(\H )\otimes C(G)_{\hbar J}$ is given by the quantum subgroup $\Pi_\hbar :C(G)_{\hbar J}\rightarrow C(K)_{\hbar J}$.  We have the undeformed homogeneous space is $C(G/K)$, which is a coadjoint orbit of $G$.  By the above proposition we have that the regular functions of $C(G/K)$ are dense in $C(G)^{H^\hbar}_{\hbar \theta}$.  By Section 5 of \cite{Varilly}, we have that the regular functions  of $C(G/K)$ are also dense in the quantum coset space $C(G/K)_{\hbar \theta}$.  Since the state stabilizer maps onto the quantum group stabilizer we have by general principle that $ C(G)^{H^\hbar}_{\hbar \theta} \hookrightarrow C(G/K)_{\hbar \theta}$.  Hence Varilly's result implies that $ C(G)^{H^\hbar}_{\hbar \theta} = C(G/K)_{\hbar \theta}$.  That is, the quantum coadjoint orbit of $C(G)_{\hbar J}$ is simply the quantum coset space $C(G/K)_{\hbar \theta}$ of Varilly.  In particular, for Rieffel deformations of classical Lie groups we see that the coadjoint orbits  are  C$^*$-algebras and not just  operator systems.

We remark that in the non-Kac type case,
the results of \cite{Jurco} suggest that the quantum coadjoint orbits for $q$-deformations of classical Lie groups should turn out to be C$^*$-algebras as well.
\end{example}

The results of this section are somewhat strange.  We have that the quantum homogeneous space $N^H$ has a dense subalgebra which is undeformed as a linear space under Rieffel's deformation.  This is not too surprising, because this is how Rieffel's deformation works.  We also have that the finite dimensional approximations $M^n$ of $N^H$ are undeformed as linear spaces, only their order structure is deformed.  What is very surprising is that the Berezin symbol and its adjoint are essentially undeformed.  Basically, the results of this section show that we have linear maps $\s^n:M^n\rightarrow \mathcal{N}^H$ and $\bs^n:\mathcal{N}^H\rightarrow M^n$ that give a Berezin Quantization of the algebraic quantum homogeneous space $\mathcal{N}^H$, which we can view as just being a linear $*$-space.  Furthermore, it suggests that  for any order structure on $\mathcal{N}^H$,  we can choose a corresponding order structure on $M^n$  so as to make these maps positive.  So it seems that these maps not only give a quantization that commutes with changes of the metric on $\mathcal{N}^H$, as we see from Theorem \ref{main2}, but also commutes with changes of the topology of $\mathcal{N}^H$.

\end{subsubsection}

\end{subsection}

 \end{section}

\begin{section}{Appendix}

Throughout this paper we have considered the state stabilizer of a vector in a unitary representation of a compact quantum group $A$.  See Definitions \ref{def4} and  \ref{defstabilizer}.  We have noted that since the state stabilizer $\Phi :A\rightarrow B$ is merely a map of order unit spaces, there's no reason to believe that $B$ should have a Haar state.  In this section we investigate when $B$ has a Haar state.

For motivation,  consider a quantum subgroup $\Pi :A\rightarrow B$  of $A$, and let $\h_B$ be the Haar state of $B$.  Let $A^\Pi=\{ a\in A\: | \: (\Pi \otimes id)\Delta a=1_B\otimes a\}$.  Then one has $(\h_B\otimes id)\Delta :A\rightarrow A^\Pi$ is an idempotent map from $A$ to $A^\Pi$.  Intuitively, we are averaging an element of $A$ over the Haar state of $B$ to get an element of $A^\Pi$.

Let $A$ be coamenable and Kac-type.  Let $(u,\H )$ be a unitary irreducible representation of $A$ with a primitive-like vector $\xi \in \H$.  Let $P$ be the rank 1 projection corresponding to $\xi$, and let $H$ be the state stabilizer of $P$ under conjugation by $u$.  Then by Proposition \ref{prop10} we have $\lim_{n\rightarrow \infty}(\hpn \otimes id)\Delta a=a$ for all $a\in A^H$.  Since $S(A)$ is weak$^*$-compact, we can let $\mathbi{h} \in S(A)$ be a limit point of the sequence $\hpn$.  Then $\mathbi{h}$ seems like a reasonable canidate for a Haar measure of $H$.  We show that this is indeed the case.

\begin{proposition}
Let $\phi \in H$.  Then we have $\hpn *\phi =\phi *\hpn =\hpn$ for any $n\in \N$.  In particular, we have $\mathbi{h}*\phi =\phi *\mathbi{h}=\mathbi{h}$.
\end{proposition}
\begin{proof}
It suffices to show the case when $n=1$, since the state stabilizer of $P^n$ is contained in the state stabilizer of $P$ by Proposition \ref{nthpower}.  We first compute $\phi *\hp$ for $\phi \in H$.  We have
\begin{eqnarray}
(\phi *\hp )(a) & = & d_\H (\phi \otimes \h )((1\otimes \s_P )\Delta a) \nonumber \\
 & = & d_\H (\phi \otimes \h )([(\kappa \otimes id) \Delta \s_P ](1\otimes  a)) \nonumber \\
 & = & d_\H  \h ([(\phi \circ \kappa \otimes id) \Delta \s_P ] a) \nonumber \\
& = & d_\H \h (\s_Pa) = \hp (a), \nonumber 
\end{eqnarray}
where we have used the facts that $\phi \circ \kappa \in H$ and $\s_P \in A^H$. See Propositions \ref{prop4} and \ref{prop5}.

Similarly,  we can compute $\hp *\phi$.

\begin{eqnarray}
(\hp *\phi )(a) & = & d_\H (\h \otimes \phi )((\s_P\otimes 1)\Delta a) \nonumber \\
& = & d_\H (\h \otimes \phi )([(id\otimes \kappa )\Delta \s_P ](a\otimes 1)) \nonumber \\
& = & d_\H \h ([(id\otimes \phi \circ \kappa )\Delta \s_P ]a) \nonumber \\
& = & d_\H \h ([(\phi \otimes \kappa^{-1} )\Delta (\kappa \s_P) ]a) \nonumber \\
& = & d_\H \h ([(\phi \otimes \kappa^{-1} )\Delta \s_P ]a) \nonumber \\
& = & d_\H \h (\kappa^{-1}(\s_P)a) \nonumber \\
& = & d_\H \h (\s_Pa) = \hp (a), \nonumber 
\end{eqnarray}
where we have used $\kappa \s_P=\langle (id\otimes \kappa )\alpha (P), P\otimes 1_A \rangle =\langle P\otimes 1_A ,\alpha (P) \rangle =\s_P^*=\s_P$.

\end{proof}

This proposition shows that each $\hpn$ is like a Haar state for $H$.  The only problem is that $\hpn \notin H$.  If one thinks about the statement of Lemma \ref{lem8}  geometrically, intuitively it says that the states $\hpn$ become more concentrated on $H$ as $n\rightarrow \infty$, so that in the limit, it's as if one is just integrating over a Haar state of $H$.  To see that $\mathbi{h}$ really is a Haar state of $H$, we just need to check that it really is an element of $H$.  But this follows from Proposition \ref{prop10} and Proposition \ref{dual}.

\begin{corollary}\label{Haar}
For $A$ coamenable and Kac-type, the state $\mathbi{h}$ constructed above is in the state stabilizer $H$, and so is a Haar state for $H$.
\end{corollary}

We have not discussed compact quantum hypergroups in this paper, but essentially a compact quantum hypergroup is like a compact quantum group, except that the coproduct is not assumed to be a $*$-homomorphism.  See \cite{Chapovsky} for the actual definition and the basic properties of compact quantum hypergroups. Any quantum hypergroup has a Haar state.  So if $\Pi :A\rightarrow B$ is a quantum subhypergroup of $A$, then the Haar state of $B$ pulls back to an idempotent state on $A$. 
Based on the results of \cite{Delvaux} and \cite{Landstad}, Franz and Skalski suggest that any idempotent state on a compact quantum group $A$ is the pullback of the Haar state on some quantum subhypergroup of $A$, and they prove that this is the case for finite quantum groups \cite{Franz}.  The full conjecture seems to be out of reach at the moment, but since our state stabilizers have Haar states, it is natural to ask whether or not they are quantum subhypergroups of $A$.

\begin{question}
Let $A$ be a coamenable compact quantum group of Kac-type.  Let $(u,\H )$ be an irreducible representation of $A$ with a primitive-like vector $\xi$, and let $P$ be the rank one projection corresponding to $\xi$.   View the state stabilizer of $P$ as a map of order unit spaces $\Pi :A\rightarrow B$.  Is $B$ actually a quantum subhypergroup of $A$?
\end{question}

We can also use this analysis to show that our coamenability assumption cannot be dropped.  Let $A$ be the universal Kac-type quantum group $A_u(m)$.  Suppose that the concentration result of Lemma \ref{lem8} holds without assuming coamenability.  Then the rest of the results of Section 4.2 also hold without assuming coamenability.  Let $(u, \C^m )$ be the fundamental representation for $A$.  By \cite{Banica2}, each of the representations $u^{\otimes k}$ are irreducible so that any nonzero $\xi \in \C^m$ is a primitive-like vector for $u$.  Thus we can take $P \in M_m(\C )$ to be any rank one projection.

Again, we let $\mathbi{h}$ be a weak$^*$-limit of the states $\hpn$.  Using the same analysis as above, we have that $\mathbi{h}$ is a Haar state for the state stabilizer of $P$.  However, we can show that $\mathbi{h} \notin H$.

\begin{proposition}
For $A=A_u(m)$, the state $\mathbi{h}$ is not in $H$.
\end{proposition}
\begin{proof}
By Proposition \ref{lem4}, we have that $\mathbi{h} \in H$ if and only if $\mathbi{h}(\s_P)=1$.
Since $\hpn (a)=dim_{\H^n}\h ((\s_P)^na)$ is a state, we have $ \h ((\s_P)^n)=(dim_{\H^n})^{-1}$.  Thus we have 
\[\hpn (\s_P)=\frac{dim_{\H^n}}{dim_{\H^{n+1}}}. \]

By Banica, we also have $\H^n=\H^{\otimes n}$.  Thus $\mathbi{h}(\s_P)=1/m \neq 1$, so $\mathbi{h}$ is not an element of $H$.
\end{proof}

\begin{corollary}
The conclusion of Lemma \ref{lem8} does not hold if $u:\C^m \rightarrow \C^m \otimes A_u(m)$ is the fundamental representation of $A_u(m)$.
\end{corollary}

\begin{conjecture}
Let $P$ be a rank 1 projection in $M_m(\C )$, and let $\alpha :M_m(\C )\rightarrow M_m(\C )\otimes A_u(m)$ be conjugation by the fundamental representation of $A_u(m)$.  Let $H$ be the state stabilizer of $P$.  Then $H$ has no Haar state. 
\end{conjecture}

If this conjecture is true, it shows that the state stabilizer of an element of a unitary representation of a compact quantum group is generally not even a quantum subhypergroup.

\end{section}


\begin{thebibliography}{100}



\bibitem[Ar]{Arveson} W. Arveson.  ``The Noncommutative Choquet Boundary III: Operator Systems in Matrix Algebras,'' arXiv:0810.4343

\bibitem[Bal]{Balachandran} A. Balachandran, B. Dolan, J. Lee, X. Martin, D. O'Connor.  ``Fuzzy Complex Projective Spaces and their Star-products.'' \emph{J.Geom.Phys.} 43 (2002) 184-204. 	arXiv:hep-th/0107099v2

\bibitem[Ban1]{Banica1} T. Banica.  	``Representations of compact quantum groups and subfactors.''  \emph{J. Reine Angew. Math.} 509 (1999), 167-198. math.QA/9804015 

\bibitem[Ban2]{Banica2} T. Banica. ``Le groupe quantique compact libre $U(n)$.'' \emph{Commun. Math. Phys.} 190 (1997), 143-172

\bibitem[Ban3]{Banica3}  T. Banica. ``Subfactors associated to compact Kac algebras.''  \emph{Integral Equations Operator Theory} 39 (2001), 1-14.  	math.OA/9806054

\bibitem[Ban4]{Banica4} T. Banica. ``Compact Kac algebras and commuting squares.'' \emph{ J. Funct. Anal.} 176 (2000), 80-99.  math.OA/9909156

\bibitem[Ban5]{Banica5} T. Banica. ``Fusion rules for representations of compact quantum groups.'' \emph{Exposition. Math.} 17 (1999), 167-198. math.QA/9811039

\bibitem[Be1]{Bedos1} E. Bedos, G. Murphy,  L. Tuset. ``Co-Amenability of Compact Quantum Groups,'' math.OA/0010248

\bibitem[Be2]{Bedos2} E. Bedos, G. J. Murphy,  L. Tuset. ``Amenability and co-amenability of algebraic quantum groups.''  math.OA/0111025

\bibitem[Bh1]{Bhowmick1} J. Bhowmick, D. Goswami. ``Quantum Group of Orientation preserving Riemannian Isometries.''  arXiv:0806.3687

\bibitem[Bh2]{Bhowmick2} J. Bhowmick, D. Goswami, A. Skalski. ``Quantum Isometry Groups of 0- Dimensional Manifolds.''  arXiv:0807.4288

\bibitem[Bh3]{Bhowmick3} J. Bhowmick, D. Goswami. ``Quantum Isometry groups of the Podles Spheres.''   arXiv:0810.0658


\bibitem[Bi]{Bietenholz} W. Bietenholz. 	 ``Simulations of a supersymmetry inspired model on a fuzzy sphere.''  arXiv:0808.2387

\bibitem[Bo]{Boca} F. P. Boca. ``Ergodic Actions of Compact Matrix Pseudogroups on C$^*$-algebras.'' \emph{Recent advances in operator algebras} Asterisque 232 (1995), 93-109.

\bibitem[Ch]{Chapovsky} Y. Chapovsky,  L. Vainerman. ``Compact Quantum Hypergroups.''  \emph{J. Operator Theory}  41  (1999),  no. 2, 261--289. 

\bibitem[De]{Delvaux} L. Delvaux, A. Van Daele. ``Algebraic Quantum Hypergroups.'' math/0606466

\bibitem[Di]{Di} P. Di Francesco, P. Mathieu, D. Senechal. \underline{Conformal field theory}. Graduate Texts in Contemporary Physics. Springer-Verlag (1997). ISBN 0-387-94785-X

\bibitem[Do1]{Dolan} B. Dolan, I. Huet, S. Murray, D. O'Connor.  ``A universal Dirac operator and noncommutative spin bundles over fuzzy complex projective spaces.''  arXiv:0711.1347

\bibitem[Do2]{Dolan2} B. Dolan, R. Szabo. ``Dimensional Reduction, Monopoles and Dynamical Symmetry Breaking.''  arXiv:0901.2491 


\bibitem[Fr]{Franz} U. Franz, A. Skalski. ``On Idempotent States on Quantum Groups.'' arXiv:0808.1683

\bibitem[Fl]{Flores} F. Flores, X. Martin, D. O'Connor. ``Simulation of a scalar field on a fuzzy sphere.''  arXiv:0903.1986 


\bibitem[Go]{Goswami} D. Goswami. ``Quantum isometry group of a compact metric space.''  arXiv:0811.0095



\bibitem[Haw]{Hawkins} E. Hawkins.
``An obstruction to quantization of the sphere.''
\emph{Comm. Math. Phys.} 283 (2008), no. 3, 675--699. 	arXiv:0706.2946 

\bibitem[Har]{Harikumar} E. Harikumar, Amilcar R. Queiroz, P. Teotonio-Sobrinho. ``Dirac operator on the q-deformed Fuzzy sphere and Its spectrum.''  hep-th/0603193

\bibitem[Ho]{Hoegh} R. Hoegh-Krohn, M.B. Lanstad, E. Stormer, ``Compact ergodic groups of automorphisms.'' \emph{Ann. of Math.} 114 (1981) 75-86.

\bibitem[Ju]{Jurco}   B. Jurco, P. Stovicek. ``Coherent States for Quantum Compact Groups.'' CERN-TH.7201/94. \emph{Commun.Math.Phys.} 182 (1996) 221-251.  hep-th/9403114

\bibitem[Ka1]{Kaschek1}  D. Kaschek, N. Neumaier, S. Waldmann.  ``Complete Positivity of Rieffel's Deformation Quantization by Actions of $\mathbb{R}^d$.''  arXiv:0809.2206

\bibitem[Ka2]{Kaschek2} D. Kaschek. ``Nichtperturbative Deformationstheorie physikalischer Zustaende.'' Ph.D. thesis

\bibitem[Ke]{Kerr} D. Kerr.  ``Matricial Quantum Gromov-Hausdorff Distance.'' math.OA/0207282


\bibitem[Kl]{Klimyk} A. Klimyk, K. Schmudgen. \underline{Quantum Groups and Their Representations}. Texts and Monographs in Physics.  Springer-Verlag, Berlin, 1997.



\bibitem[Lan]{book} N. P. Landsman. \underline{Mathematical Topics Between Classical and Quantum Mechanics}.  Springer. 1998.

\bibitem[Lad]{Landstad}  M. B. Landstad, A. Van Daele. ``Compact and discrete subgroups of algebraic quantum groups I.''  math/0702458

\bibitem[Li]{Li} H. Li. ``Compact Quantum Metric Spaces and Ergodic Actions of Compact Quantum Groups,''  math/0411178

\bibitem[Lo]{Lozano} Y. Lozano, D. Rodriguez-Gomez. ``Pp-wave Matrix Models from Point-like Gravitons.'' hep-th/0701133


\bibitem[Mae]{Maes} A. Maes,  A. Van Daele. ``Notes on Compact Quantum Groups,'' math.FA/98031225

\bibitem[Maj]{Majid} S. Majid.  ``q-Fuzzy spheres and quantum differentials on $B_q[SU_2]$ and $U_q(su_2)$.''  arXiv:0812.4942


\bibitem[Mas1]{Masuda1}  T. Masuda, R. Tomatsu. ``Classification of minimal actions of a compact Kac algebra with amenable dual.'' math.OA/0604348

\bibitem[Mas2]{Masuda2} T. Masuda, R. Tomatsu. ``Classification of minimal actions of a compact Kac algebra with amenable dual on injective factors of type III.'' arXiv:0806.4259

\bibitem[Mo]{Mondragon}  M. Mondragon, G. Zoupanos. ``Unified Gauge Theories and Reduction of Couplings: from Finiteness to Fuzzy Extra Dimensions.''  arXiv:0802.3454





\bibitem[Ni]{Nikshych} D. Nikshych. ``Duality for actions of weak Kac algebras and crossed product inclusions of II$_1$ factors.'' math.QA/9810049




\bibitem[Pa]{Paulsen} V. Paulsen,  M. Tomforde.  ``Vector Spaces With An Order Unit,'' arXiv:0712.2613

\bibitem[Pe]{Peskin} M. Peskin, D. Schroeder. \underline{An Introduction To Quantum Field Theory}.  Westview Press (1995).

\bibitem[Pi1]{Pinzari1} C. Pinzari. ``Embedding ergodic actions of compact quantum groups on C*-algebras into quotient spaces.''  math.OA/0604308

\bibitem[Pi2]{Pinzari2} C. Pinzari, J. Roberts. ``A duality theorem for ergodic actions of compact quantum groups on C$^*$-algebras.'' \emph{Comm. Math. Phys.}, 277 (2008), 385-421.

\bibitem[Pi3]{Pinzari3}  C. Pinzari, J. Roberts.  ``Ergodic actions of $S_\mu U(2)$ on C*-algebras from II$_1$ subfactors.'' 	arXiv:0806.4519 

\bibitem[Po]{Podles} P. Podles.  ``Symmetries of quantum spaces.  Subgroups and quotient spaces of quantum SU(2) and SO(3) groups.'' 1/92 KMMF UW. Commun.Math.Phys. 170 (1995) 1-20.  hep-th/9402069

\bibitem[Re1]{Rieffel1} M. Rieffel.  ``Gromov-Hausdorff Distance for Quantum Metric Spaces.'' \emph{Mem. Amer. Math. Soc.} 168 (2004) no. 796, 1-65. math.OA/0011063

\bibitem[Re2]{Rieffel2} M. Rieffel.  ``Metrics on States from Actions of Compact Groups,'' \emph{Documenta Mathematica} 3 (1998) 215-229.  math.OA/9807084


\bibitem[Re3]{Rieffel3} M. Rieffel.  ``Matrix Algebras Converge to the Sphere for Quantum Gromov-Hausdorff Distance,'' \emph{Mem. Amer. Math. Soc.} 168 (2004) no. 796, 67-91.  math.OA/0108005

\bibitem[Re4]{Rieffel4} M. Rieffel. ``Deformation quantization for actions of $R\sp d$.'' \emph{Mem. Amer. Math. Soc.}  106  (1993),  no. 506.

\bibitem[Re5]{Rieffel5} M. Rieffel. ``Compact quantum groups associated with toral subgroups.''  \emph{Representation theory of groups and algebras},  465--491, Contemp. Math., 145, \emph{Amer. Math. Soc., Providence, RI}, 1993.

\bibitem[Re6]{Rieffel6}  M. Rieffel.	``Leibniz seminorms for `Matrix algebras converge to the sphere'''.  arXiv:0707.3229 


\bibitem[Sh]{Sheu} A. Sheu.  ``Quantization of the Poisson ${\rm SU}(2)$ and its Poisson homogeneous space---the $2$-sphere.'' \emph{ Comm. Math. Phys. } 135  (1991),  no. 2, 217--232. 



\bibitem[Va]{Varilly}  J. Varilly. ``Quantum symmetry groups of noncommutative spheres.''  \emph{ Commun.Math.Phys. } 221 (2001) 511-523. math.QA/0102065

\bibitem[VD]{VanDaele} A. Van Daele, S. Wang, ``Universal quantum groups.'' \emph{ Internat. J. Math.}  7  (1996),  no. 2, 255--263. 


\bibitem[Wa1]{Wang1} S. Wang.  ``Structure and Isomorphism Classification of Compact Quantum Groups $A_u(Q)$ and $B_u(Q)$,'' \emph{Journal of Operator Theory} 48 (2002), 573-583.  math.OA/9807095

\bibitem[Wa2]{Wang2} S. Wang.  ``Ergodic Actions of Universal Quantum Groups on Operator Algebras,'' \emph{Commun. Math. Phys.} 203 (1999), 481-498. math.OA/9807093

\bibitem[Wa3]{Wang3}  S. Wang.  ``Deformations of compact quantum groups via Rieffel's quantization.''  \emph{Comm. Math. Phys.}  178  (1996),  no. 3, 747--764.




\bibitem[Wo1]{Woro1} S.L. Woronowicz. ``Compact Matrix Pseudogroups.'' \emph{Comm. Math. Phys.} Volume 111, Number 4 (1987), 613-665.

\bibitem[Wo2]{Woro2}  S. L. Woronowicz. ``Twisted $ SU(2)$ group. An example of a noncommutative differential calculus.''  \emph{Publ. Res. Inst. Math. Sci.}  23  (1987),  no. 1, 117--181. 




\end{thebibliography}
\end{document}